\documentclass[a4paper]{amsart}
\usepackage{amssymb}
\usepackage{stmaryrd}
\usepackage[all]{xy}
\usepackage{epsf}
\usepackage{combelow}
\usepackage{enumerate}
\newcommand{\printname}[1] {}

\newtheorem{theorem}{Theorem}[section]
\newtheorem{Introtheorem}{Theorem}

\newtheorem*{conntheorem}{Conn's Theorem}
\newtheorem*{startheorem}{Theorem}
\newtheorem*{stardefinition}{Definition}

\newtheorem{proposition}[theorem]{Proposition}
\newtheorem{lemma}[theorem]{Lemma}
\newtheorem{corollary}[theorem]{Corollary}

\theoremstyle{remark}
\newtheorem{remark}{Remark}

\newcommand{\plin}{\pi_{\mathrm{lin}}}
\newcommand{\rmap}{\longrightarrow}
\newcommand{\diffto}{\xrightarrow{\raisebox{-0.2 em}[0pt][0pt]{\smash{\ensuremath{\sim}}}}}

\usepackage{amsmath,amscd}

\begin{document}
\title{Rigidity around Poisson submanifolds}
\author{Ioan M\u arcu\cb{t}}
%\address{Depart. of Math., Utrecht University, 3508 TA Utrecht, The Netherlands}
\address{University of Illinois at Urbana-Champaign, Urbana, IL 61801 USA} \email{marcut@illinois.edu}
\begin{abstract}
We prove a rigidity theorem in Poisson geometry around compact Poisson submanifolds, using the Nash-Moser fast
convergence method. In the case of one-point submanifolds (fixed points), this implies a stronger version of Conn's
linearization theorem \cite{Conn}, also proving that Conn's theorem is a manifestation of a rigidity phenomenon;
similarly, in the case of arbitrary symplectic leaves, it gives a stronger version of the local normal form theorem
\cite{CrMa}. We can also use the rigidity theorem to compute the Poisson moduli space of the sphere in the dual of a
compact semisimple Lie algebra \cite{MarDef}.
\end{abstract}
\maketitle

\setcounter{tocdepth}{2}

%\tableofcontents
\section*{Introduction}
Recall that a \textbf{Poisson structure} on a manifold $M$ is a Lie bracket $\{\cdot,\cdot\}$ on the space $C^{\infty}(M)$ of smooth functions
on $M$ which acts as a derivation in each entry:
\[\{f,gh\}=\{f,g\}h+\{f,h\}g, \quad  f,g,h \in C^{\infty}(M).\]
A Poisson structure can be given also by a bivector $\pi\in\mathfrak{X}^2(M)$ satisfying $[\pi,\pi]=0$ for the Schouten bracket. The Lie bracket
is related to $\pi$ by the formula
\[\langle\pi,df\wedge dg\rangle=\{f,g\},\quad  f,g \in C^{\infty}(M).\]
The \textbf{Hamiltonian} vector field of a function $f\in C^{\infty}(M)$ is
\[X_f=\{f,\cdot\} \in\mathfrak{X}(M).\]
These vector fields span an involutive singular distribution on $M$, which integrates to a partition of $M$ into regularly immersed submanifolds
called \textbf{symplectic leaves}. These leaves are symplectic manifolds, the symplectic structure on the leaf $S$ is given by
$\omega_S:=\pi|_{S}^{-1}\in\Omega^2(S)$.

The zero-dimensional symplectic leaves are the points $x\in M$ where $\pi$ vanishes. At such a fixed point $x$, the
cotangent space $\mathfrak{g}_x=T^*_xM$ carries a Lie algebra structure, called the \textbf{isotropy Lie algebra} at
$x$, with bracket given by
\[[d_xf,d_xg]:=d_x\{f,g\}, \ f,g\in C^{\infty}(M).\]
Conversely, starting from a Lie algebra $(\mathfrak{g},[\cdot,\cdot])$ there is an associated Poisson structure $\pi_{\mathfrak{g}}$ on the
vector space $\mathfrak{g}^*$, called \textbf{the linear Poisson structure}, defined by
\[\{f,g\}_{\xi}:=\langle\xi,[d_{\xi}f,d_{\xi}g]\rangle, \ f,g\in C^{\infty}(\mathfrak{g}^*).\]
So, at a fixed point $x$, the tangent space $T_xM=\mathfrak{g}_x^*$ carries a canonical Poisson structure
$\pi_{\mathfrak{g}_x}$ which plays the role of the first order approximation of $(M,\pi)$ around $x$ in the realm of
Poisson geometry. We recall Conn's linearization theorem \cite{Conn}:
\begin{conntheorem}\label{Theorem_Conn}
Let $(M,\pi)$ be a Poisson manifold and $x\in M$ be a fixed point of $\pi$. If the isotropy Lie algebra $\mathfrak{g}_x$
is semisimple of compact type then a neighborhood of $x$ in $(M,\pi)$ is Poisson-diffeomorphic to a neighborhood of
the origin in $(\mathfrak{g}_{x}^{*},\pi_{\mathfrak{g}_{x}})$.
\end{conntheorem}

Conn's proof is analytic, it uses the fast convergence method of Nash and Moser. A new proof of Conn's theorem, which
uses Poisson-geometric techniques, is now available in \cite{CrFe-Conn}. This geometric proof was adapted to the case
of general symplectic leaves \cite{CrMa}, and the outcome will be explained in the sequel.

Recall that the cotangent bundle of a Poisson manifold $(M,\pi)$ is canonically a Lie algebroid $(T^*M,
[\cdot,\cdot]_{\pi},\pi^{\sharp})$ with anchor given by the map
\[\pi^{\sharp}:T^*M\rmap TM,\ \  \pi^{\sharp}(\alpha):=\pi(\alpha,\cdot), \ \quad \ \alpha\in T^*M,\]
and the Lie bracket given by the expression
\[[\alpha, \beta]_{\pi}= L_{\pi^{\sharp}(\alpha)}(\beta)-L_{\pi^{\sharp}(\beta)}(\alpha)-d\pi(\alpha,\beta),\quad \alpha,\beta\in \Gamma(T^*M).\]

Generalizing the isotropy algebra from the case of fixed points, one associates to a symplectic leaf $(S,\omega_S)$ a
transitive Lie algebroid $A_S:=T^*M|_{S}$ over $S$, which is the restriction of $T^*M$ to $S$, and is called
\textbf{the restricted Lie algebroid}.

Conversely, using the data of a transitive Lie algebroid $(A,[\cdot,\cdot],\rho)$ over a symplectic manifold
$(S,\omega_S)$, Vorobjev constructed in \cite{Vorobjev} a Poisson manifold $(N(A),\pi_A)$ which serves as \textbf{the
first order local model} of a Poisson structure around a symplectic leaf. The space $N(A)$ is an open set in
$\mathfrak{g}(A)^*$, where $\mathfrak{g}(A):=\ker(\rho)$ is the isotropy bundle. The Poisson manifold $(N(A),\pi_A)$
has $(S,\omega_S)$ (viewed as the zero section) as a symplectic leaf, and $A$ can be recovered as the transitive Lie
algebroid corresponding to this leaf: $A\cong A_S$. The construction depends on the choice of a linear left inverse to
the inclusion $\mathfrak{g}(A)\subset A$, but, up to isomorphisms around $S$, the outcome does not depend on this
choice (see subsection \ref{subsection_the_local} for more details).

In this setting, we recall the following normal form result (Theorem 1 \cite{CrMa}):

\begin{startheorem}[The normal form theorem from \cite{CrMa}]\label{The_normal_form_theorem_from_CrMa}
Let $(M,\pi)$ be a Poisson manifold, with $(S,\omega_S)$ a compact symplectic leaf. If the restricted Lie algebroid
$A_S:=T^*M|_S$ is integrable and the 1-connected Lie groupoid integrating it is compact and its $s$-fibers have
vanishing de Rham cohomology in degree two, then a neighborhood of $S$ in $(M,\pi)$ is Poisson-diffeomorphic to a
neighborhood of the zero section in the local model $(N(A_S),\pi_{A_S})$.
\end{startheorem}

In the case of fixed points this is equivalent to Conn's result.\\

%\vspace{0.1 cm}

The original goal of this research was to reprove this theorem with methods similar to those of Conn's original
approach. The main incentive for this is that Conn's analytic techniques are apparently more powerful than the
geometric ones from \cite{CrMa}; in particular, as suggested to the author by Crainic, an analytic proof should imply
rigidity of the Poisson structure. This is indeed the case, and the precise rigidity property that we obtain is the
following:

\begin{stardefinition}
A Poisson structure $\pi$ on $M$ is called $C^p$-$C^1$-\textbf{rigid around} the compact submanifold $N\subset
M$, if there are small enough open neighborhoods $U$ of $N$, such that for all open sets $O$ with $N\subset O\subset
\overline{O}\subset U$, there exist
\begin{itemize}
\item an open neighborhood $\mathcal{V}_{O}\subset \mathfrak{X}^2(U)$ of $\pi|_{U}$ in the compact-open $C^p$-topology,
\item a function $\widetilde{\pi}\mapsto \psi_{\widetilde{\pi}}$, which associates to a Poisson structure $\widetilde{\pi}\in
\mathcal{V}_{O}$ a map $\psi_{\widetilde{\pi}}:\overline{O}\to M$ which extends to an embedding of a
neighborhood of $\overline{O}$,
\end{itemize}
such that $\psi_{\widetilde{\pi}}$ is a Poisson diffeomorphism
\[\psi_{\widetilde{\pi}}:(O,\pi|_{O})\diffto (\psi_{\widetilde{\pi}}(O),\widetilde{\pi}|_{\psi_{\widetilde{\pi}}(O)}),\]
and $\psi$ is continuous at $\widetilde{\pi}=\pi$ (with $\psi_{\pi}=\mathrm{Id}_{\overline{O}}$), with respect to the
$C^p$-topology on the space of Poisson structures and the $C^1$-topology on $C^{\infty}(\overline{O},M)$.
\end{stardefinition}

We prove the following improvement of \cite{CrMa}, which also includes rigidity:

\begin{Introtheorem}\label{The_normal_form_theorem}
Let $(M,\pi)$ be a Poisson manifold and $(S,\omega_S)$ a compact symplectic leaf. If the Lie algebroid
$A_S:=T^*M|_{S}$ is integrable by a compact Lie groupoid whose $s$-fibers have vanishing de Rham cohomology in
degree two, then
\begin{enumerate}[(a)]
\item in a neighborhood of $S$, $\pi$ is
Poisson diffeomorphic to its local model around $S$,
\item $\pi$ is $C^p$-$C^1$-rigid around $S$.
\end{enumerate}
\end{Introtheorem}

Already in the case of fixed points, the first part of this theorem gives a slight generalization of Conn's result, which
cannot be obtained by an immediate adaptation of the arguments in \cite{CrFe-Conn,CrMa}. Namely, a Lie algebra is
integrable by a compact group with vanishing second de Rham cohomology if and only if it is compact and its center is
at most one-dimensional (see Lemma \ref{Lemma_cohomology_lie_groups}). The case when the center is trivial is
Conn's result, and the one-dimensional case is a consequence of a result of Monnier and Zung on smooth Levi
decomposition of Poisson manifolds \cite{Monnier}.

However, the main advantage of the approach of this paper over \cite{CrMa} is that it allows for a rigidity theorem
around an arbitrary Poisson submanifold. Recall that a submanifold $N$ of $(M,\pi)$ such that $\pi$ is tangent to $N$
is called a \textbf{Poisson submanifold}. The symplectic leaves are the simplest type of Poisson
submanifolds. The main result of this paper is the following rigidity theorem for integrable Poisson manifolds. % and Theorem \ref{The_normal_form_theorem} is just an immediate consequence of it.

\begin{Introtheorem}\label{Main_Theorem_intro}
Let $(M,\pi)$ be a Poisson manifold for which the Lie algebroid $T^*M$ is integrable by a Hausdorff Lie groupoid
whose $s$-fibers are compact and their de Rham cohomology vanishes in degree two. For ever compact Poisson
submanifold $N$ of $M$ we have that
\begin{enumerate}[(a)]
\item $\pi$ is $C^p$-$C^1$-rigid around $N$,
\item up to isomorphism, $\pi$ is determined around $N$ by its first order jet at $N$.
\end{enumerate}
\end{Introtheorem}

We prove Theorem \ref{The_normal_form_theorem} by applying part (b) of this result to the local model.

In both theorems, $p$ has the following (most probably not optimal) value:
\[p=7(\lfloor \mathrm{dim}(M)/2\rfloor+5).\]

In part (b) of Theorem \ref{Main_Theorem_intro} we prove that every Poisson structure $\widetilde{\pi}$, defined
around $N$, that satisfies $j^1\pi|_{N}=j^1\widetilde{\pi}|_{N}$ is isomorphic to $\pi$ around $N$ by a
diffeomorphism which is the identity on $N$ up to first order.

The structure encoded by the first order jet of $\pi$ at $N$ can be organized as an extension of Lie algebroids (see
Remark 2.2 \cite{MaFormal})
\begin{equation}\label{EQ_I_1}
0\rmap \nu_N^* \rmap T^*M|_{N}\rmap T^*N\rmap 0,
\end{equation}
where $\nu_N^*  \subset T^*M|_{N}$ is the conormal bundle and $T^*N$ is the cotangent Lie algebroid of the Poisson
manifold $(N,\pi|_{N})$. With this, Theorem \ref{The_normal_form_theorem} follows easily from Theorem
\ref{Main_Theorem_intro}: if $S:=N$ is a compact symplectic leaf, then the Poisson structures $(M,\pi)$ and
$(N(A_S),\pi_{A_S})$ have the same first order jet around $S$ (they induce the same exact sequence (\ref{EQ_I_1}));
moreover, the hypothesis of Theorem \ref{The_normal_form_theorem} implies that Theorem \ref{Main_Theorem_intro}
can be applied to the local model $(N(A_S),\pi_{A_S})$ (see Lemma \ref{Lemma_small_neighborhoods}).

One might try to follow the same line of reasoning and use Theorem \ref{Main_Theorem_intro} to prove a normal form
theorem around Poisson submanifolds. Unfortunately, around general Poisson submanifolds, a first order local model
does not seem to exist. Actually, there are Lie algebroid extensions as in (\ref{EQ_I_1}) which do not arise as the first
jet of Poisson structures (see Example 2.3 in \cite{MaFormal}). Nevertheless,
one can use Theorem \ref{Main_Theorem_intro} to prove normal form results around particular classes of Poisson submanifolds.\\

\noindent \textbf{The paper is organized as follows.} In section \ref{section_main_results}, after recalling some
properties of Lie groupoids and Lie algebroids, we describe in detail the local model around a leaf and a symplectic
groupoid integrating it. We end the section by proving that Theorem \ref{Main_Theorem_intro} implies Theorem
\ref{The_normal_form_theorem}. Section \ref{section_examples} is an extended introduction to the paper, we give a
list of applications, examples and connections with related literature. In section \ref{section_technical_rigidity} we
prove Theorem \ref{Main_Theorem_intro} by using the Nash-Moser method. The appendix contains three general
results on Lie groupoids: existence of invariant tubular neighborhoods, integrability of the adjoint representation on a
proper ideal, and the Tame Vanishing Lemma. This last result provides tame homotopy operators for Lie algebroid
cohomology with coefficients and, when combined with the Nash-Moser techniques, it is a very useful tool for handling
similar geometric problems (see the appendix in \cite{teza}).\\

 \noindent \textbf{About the proof.} The proof of the rigidity theorem is inspired mainly by Conn's paper
\cite{Conn}. Conn uses a technique due to Nash and Moser to construct a sequence of changes of coordinates in which
$\pi$ converges to the linear Poisson structure $\pi_{\mathfrak{g}_x}$. At every step the new coordinates are found
by solving some equations which are regarded as belonging to the complex computing the Poisson cohomology of
$\pi_{\mathfrak{g}_x}$. To account for the ``loss of derivatives'' phenomenon during this procedure he uses
smoothing operators. Finally, he proves uniform convergence of these changes of coordinates and of their higher
derivatives on some ball around $x$.

Conn's proof has been formalized in \cite{Miranda,Monnier} into an abstract Nash Moser normal form theorem. It is likely that part (a) of our
Theorem \ref{Main_Theorem_intro} could be proven using Theorem 6.8 in \cite{Miranda}. Due to some technical issues (see Remark
\ref{Remark_SCI}), we cannot apply this result to conclude neither part (b) of our Theorem \ref{Main_Theorem_intro} nor the normal form Theorem
\ref{The_normal_form_theorem}, therefore we follow a direct approach.

We also simplified Conn's argument by giving coordinate free statements and working with flows of vector fields. For
the expert: we gave up on the polynomial-type inequalities using instead only inequalities which assert tameness of
certain maps, i.e.\ we work in Hamilton's category of tame Fr\'echet spaces. Our proof deviates the most from Conn's
when constructing the homotopy operators. Conn recognizes the Poisson cohomology of $\pi_{\mathfrak{g}_x}$ as the
Chevalley-Eilenberg cohomology of $\mathfrak{g}_x$ with coefficients in the Fr\'echet space of smooth functions. By
passing to the Lie group action on the corresponding Sobolev spaces, he proves existence of tame (in the sense of
Hamilton \cite{Ham}) homotopy operators for this complex. We, on the other hand, regard this cohomology as Lie
algebroid cohomology, and prove a general tame vanishing result for the cohomology of Lie algebroids integrable by
groupoids with compact $s$-fibers. This is done by further identifying this complex with the invariant part of the de
Rham complex of $s$-foliated forms on the Lie groupoid, and by using the fiberwise inverse of the
Laplace-Beltrami operator in order to construct the homotopy operators.\\

\noindent \textbf{Acknowledgments.} This project is part of my PhD thesis and was proposed by my advisor Marius
Crainic. I would like to thank him for his constant help and support throughout my work. Many thanks as well to Eva
Miranda, Florian Sch\"atz and Ivan Struchiner for useful discussions. The referee's suggestions improved the initial
version. This research was supported by the ERC Starting Grant no. 279729.

%%%%%%%%%%%%%%%%%%%%%%%%%%%%%%%%%%%%%%%%%%%%%%%%%%%%%%%%%%%%%%%%%%%%%%%%%%%%%%%%%%%%%%%%%%%%%%%%%%%%
%%%%%%%%%%%%%%%%%%%%%%%%%%%%%%%%%%%%%%%%%%%%%%%%%%%%%%%%%%%%%%%%%%%%%%%%%%%%%%%%%%%%%%%%%%%%%%%%%%%%
%%%%%%%%%%%%%%%%%%%%%%%%%%%%%%%%%%%%%%%%%%%%%%%%%%%%%%%%%%%%%%%%%%%%%%%%%%%%%%%%%%%%%%%%%%%%%%%%%%%%
%%%%%%%%%%%%%%%%%%%%%%%%%%%%%%%%%%%%%%%%%%%%%%%%%%%%%%%%%%%%%%%%%%%%%%%%%%%%%%%%%%%%%%%%%%%%%%%%%%%%
%%%%%%%%%%%%%%%%%%%%%%%%%%%%%%%%%%%%%%%%%%%%%%%%%%%%%%%%%%%%%%%%%%%%%%%%%%%%%%%%%%%%%%%%%%%%%%%%%%%%
%%%%%%%%%%%%%%%%%%%%%%%%%%%%%%%%%%%%%%%%%%%%%%%%%%%%%%%%%%%%%%%%%%%%%%%%%%%%%%%%%%%%%%%%%%%%%%%%%%%%
%%%%%%%%%%%%%%%%%%%%%%%%%%%%%%%%%%%%%%%%%%%%%%%%%%%%%%%%%%%%%%%%%%%%%%%%%%%%%%%%%%%%%%%%%%%%%%%%%%%%
%%%%%%%%%%%%%%%%%%%%%%%%%%%%%%%%%%%%%%%%%%%%%%%%%%%%%%%%%%%%%%%%%%%%%%%%%%%%%%%%%%%%%%%%%%%%%%%%%%%%
\section{Proof of the normal form theorem (Theorem \ref{Main_Theorem_intro} $\Rightarrow$ Theorem \ref{The_normal_form_theorem})}\label{section_main_results}
%%%%%%%%%%%%%%%%%%%%%%%%%%%%%%%%%%%%%%%%%%%%%%%%%%%%%%%%%%%%%%%%%%%%%%%%%%%%%%%%%%%%%%%%%%%%%%%%%%%%
%%%%%%%%%%%%%%%%%%%%%%%%%%%%%%%%%%%%%%%%%%%%%%%%%%%%%%%%%%%%%%%%%%%%%%%%%%%%%%%%%%%%%%%%%%%%%%%%%%%%
%%%%%%%%%%%%%%%%%%%%%%%%%%%%%%%%%%%%%%%%%%%%%%%%%%%%%%%%%%%%%%%%%%%%%%%%%%%%%%%%%%%%%%%%%%%%%%%%%%%%
%%%%%%%%%%%%%%%%%%%%%%%%%%%%%%%%%%%%%%%%%%%%%%%%%%%%%%%%%%%%%%%%%%%%%%%%%%%%%%%%%%%%%%%%%%%%%%%%%%%%
%%%%%%%%%%%%%%%%%%%%%%%%%%%%%%%%%%%%%%%%%%%%%%%%%%%%%%%%%%%%%%%%%%%%%%%%%%%%%%%%%%%%%%%%%%%%%%%%%%%%
%%%%%%%%%%%%%%%%%%%%%%%%%%%%%%%%%%%%%%%%%%%%%%%%%%%%%%%%%%%%%%%%%%%%%%%%%%%%%%%%%%%%%%%%%%%%%%%%%%%%
%%%%%%%%%%%%%%%%%%%%%%%%%%%%%%%%%%%%%%%%%%%%%%%%%%%%%%%%%%%%%%%%%%%%%%%%%%%%%%%%%%%%%%%%%%%%%%%%%%%%
%%%%%%%%%%%%%%%%%%%%%%%%%%%%%%%%%%%%%%%%%%%%%%%%%%%%%%%%%%%%%%%%%%%%%%%%%%%%%%%%%%%%%%%%%%%%%%%%%%%%
In this section, first we recall some basic properties of Lie algebroids and Lie groupoids, next we describe the local
model around a symplectic leaf from three different perspectives, and we conclude by showing that Theorem
\ref{The_normal_form_theorem} is a consequence of Theorem \ref{Main_Theorem_intro}.

\subsection{Lie groupoids and Lie algebroids}

We recall here some standard results about Lie groupoids and Lie algebroids, for definitions and other basic properties
we recommend \cite{MackenzieGT,MM}. To fix notations, the anchor of a Lie algebroid $A\to M$ will be denoted by
$\rho$; the source and target maps of a Lie groupoid $\mathcal{G}\rightrightarrows M$ by $s$ and $t$ respectively,
the unit map by $u$.

A Lie groupoid $\mathcal{G}\rightrightarrows M$ has an associated Lie algebroid $A(\mathcal{G})$ over $M$; as a vector bundle $A(\mathcal{G})$
is the restriction to $M$ (i.e.\ pullback by $u$) of the subbundle $T^s\mathcal{G}$ of $T\mathcal{G}$ consisting of vectors tangent to the
$s$-fibers. The anchor is given by the differential of $t$. The Lie bracket comes from the identification between sections of $A(\mathcal{G})$
and right invariant vector fields on $\mathcal{G}$.

A Lie algebroid $(A,[\cdot,\cdot],\rho)$ is called \textbf{integrable} if it is isomorphic to the Lie algebroid $A(\mathcal{G})$ of a Lie
groupoid $\mathcal{G}\rightrightarrows M$. Not every Lie algebroid is integrable (see \cite{CrFe1}). If a Lie algebroid is integrable, then, as
for Lie algebras, there exists up to isomorphism a unique Lie groupoid with 1-connected $s$-fibers integrating it.

A Lie algebroid $A\to M$ is called \textbf{transitive} if $\rho$ is surjective. A Lie groupoid is called \textbf{transitive} if the map
$(s,t):\mathcal{G}\to M\times M$ is a surjective submersion. If $\mathcal{G}$ is transitive then also $A(\mathcal{G})$ is transitive.
Conversely, if $A\to M$ is transitive and $M$ is connected, then every Lie groupoid integrating it is transitive as well.

Out of a principal bundle $q:P\to S$ with structure group $G$ one can construct a transitive Lie groupoid $\mathcal{G}(P)$, called the
\textbf{gauge groupoid} of $P$, as follows:
\[\mathcal{G}(P):=P\times_GP\rightrightarrows S,\]
with structure maps given by
\[s([p_1,p_2]):=q(p_2),\ \ t([p_1,p_2]):=q(p_1),\ \ [p_1,p_2][p_2,p_3]:=[p_1,p_3].\]
The Lie algebroid of $\mathcal{G}(P)$ is $TP/G$, where the Lie bracket is obtained by identifying sections of $TP/G$
with $G$-invariant vector fields on $P$. Conversely, every transitive Lie groupoid $\mathcal{G}$ is the gauge groupoid
of a principal bundle: the bundle is any $s$-fiber of $\mathcal{G}$ and the structure group is the isotropy group. So, a
transitive Lie algebroid $A$ is integrable if and only if there exists a principal $G$-bundle $P$ such that $A$ is
isomorphic to $TP/G$.

A \textbf{symplectic groupoid} $(\mathcal{G},\omega)\rightrightarrows M$ is a Lie groupoid
$\mathcal{G}\rightrightarrows M$ endowed with a symplectic structure $\omega\in \Omega^2(\mathcal{G})$ for
which the graph of the multiplication is a Lagrangian submanifold:
\[\{(g_1,g_2,g_3): g_1g_2=g_3\}\subset (\mathcal{G}\times\mathcal{G}\times\overline{\mathcal{G}},\textrm{pr}_1^*(\omega)+\textrm{pr}_2^*(\omega)-\textrm{pr}_3^*(\omega)).\]
This condition has several consequences. It implies that the base carries a Poisson structure $\pi$ such that source
map is Poisson and the target map is anti-Poisson; and moreover, that $\mathcal{G}$ integrates the cotangent Lie
algebroid $T^*M$ of $\pi$. Conversely, if for a given Poisson manifold $(M,\pi)$ the Lie algebroid $T^*M$ is
integrable, then the $s$-fiber 1-connected integration of $T^*M$ is canonically a symplectic groupoid \cite{MackXu}.

\subsection{The local model}\label{subsection_the_local}

Consider a Poisson manifold $(M,\pi)$ and let $(S,\omega_S)$ be an embedded symplectic leaf. The local model of $\pi$
around $S$, constructed first by Vorobjev in \cite{Vorobjev}, is a Poisson structure defined on some open
neighborhood of $S$ in $M$, which plays the role of a first order approximation of $\pi$ around $S$.

The local model depends (up to diffeomorphisms around $S$ that fix $S$) only on the first jet of $\pi$ at $S$, denoted
by $j^1\pi|_S$. Consider the transitive Lie algebroid associated to $S$
\[A_S:=T^*M|_S.\]
Note that the anchor of $A_S$ is given by the inverse of the symplectic structure $\omega_S$, and that the isotropy
bundle of $A_S$ is the conormal bundle $\nu_S^*\subset A_S$. In fact, $j^1\pi|_S$ encodes precisely the Lie algebroid
structure on $A_S$ (see Proposition 4.1.13 in \cite{teza}):
\begin{proposition}\label{Prop_jet}
Let $\pi_1$ and $\pi_2$ be two Poisson structures defined around $S$, such that $S$ is a symplectic leaf for both. Then
$\pi_1$ and $\pi_2$ induce the same Lie algebroid structure on $A_S=T^*M|_S$ if and only if
$j^1\pi_1|_S=j^1\pi_2|_S$.
\end{proposition}

We give three different description of the local model, each of them bringing different insight into the construction. All
three constructions avoid the explicit use of Vorobjev triples, by using instead Dirac geometric techniques. For the
proofs of the claims made here, we refer the reader to sections 4.1 and 4.2 in \cite{teza}.

\subsubsection*{Description 1}

Our first approach to the local model is a Dirac geometric interpretation of the linearization procedure from
\cite{CrFe-stab}; and it is very useful for explicit computations of the local model. Consider a tubular neighborhood of
$S$ in $M$
\[\Psi:\nu_S\rmap M ,\]
where $\nu_S:=TM|_S/TS$ is the normal bundle to $S$. Denote by $E:=\Psi(\nu_S)$, by $\mu_t:E\to E$ the map
corresponding to multiplication by $t\in\mathbb{R}$ on $\nu_S$, and by $p:E\to S$ the corresponding projection map.
Consider the following path of Poisson structures:
\begin{equation}\label{path}
\pi_t:=t\mu_t^*(\pi^{(t-1)p^*(\omega_S)}),\ \ t\in (0,1],
\end{equation}
where, for a closed 2-form $\beta$, $\pi^{\beta}$ denotes the gauge transform of $\pi$ by $\beta$ (i.e.\ the leaves of
$\pi^{\beta}$ are the leaves of $\pi$, but the symplectic structures on them differ by the restrictions of $\beta$). In
fact, $\pi_t$ is well-defined on the entire $E$ only as a Dirac structure (see \cite{BR} for the basics of Dirac geometry),
which is given by
\[L_t:=t\mu_t^*(L_{\pi}^{(t-1)p^*(\omega_S)})\subset TE\oplus T^*E,\]
where $L_{\pi}$ is the Dirac structure corresponding to $\pi$, and, for $L$ a Dirac structure and $\lambda\in
\mathbb{R}\backslash\{0\}$, we denote by $\lambda L$ the Dirac structure $\{\lambda X+\xi: X+\xi\in L\}$. Now
$L_t$ extends smoothly at $t=0$, and we let $L_0:=\lim_{t\to 0}L_t$. On the other hand, we have that $L_t$ has
$(S,\omega_S)$ as a (pre)-symplectic leaf, for all $t\in\mathbb{R}$, and therefore there is an open neighborhood $U$ of
$S$ such that $L_t$ corresponds to a Poisson structure $\pi_t$ on $U$ for all $t\in [0,1]$. The limit
Poisson structure
\[\pi_0:=\lim_{t\to 0}\pi_t,\]
defined on $U$, is the local model of $\pi$ around $S$. We also have that
\[j^1\pi_t|_S=j^1\pi|_S, \ \ t\in\mathbb{R},\]
and in particular, by Proposition \ref{Prop_jet}, the local model $\pi_0$ induces the same Lie algebroid structure on
$A_S=T^*M|_S$.

Different choices of tubular neighborhoods of $S$ give rise to local models that are isomorphic around $S$ by
diffeomorphisms that fix $S$.

Note also that the Dirac-geometric nature of this construction, allows one to define in a similar fashion the local model
of a Dirac structure around an embedded presymplectic leaf; the outcome is a Dirac structure which is globally defined
on $E$.

\subsubsection*{Description 2}
The second description comes closest to Vorobjev's original construction \cite{Vorobjev}. The construction uses the
data encoded by the first jet of a Poisson structure at a leaf: a symplectic manifold $(S,\omega_S)$ and a transitive Lie
algebroid $(A,[\cdot,\cdot]_A,\rho)$ over $S$. Similar to the linear Poisson structure on the dual of a Lie algebra, the
dual vector bundle $A^*$ carries a linear Poisson structure $\plin(A)$, with Poisson bracket determined by
\[\{p^*(f),p^*{(g)}\}=0, \  \  \  \{\widetilde{\alpha},p^*(g)\}=p^*(L_{\rho(\alpha)}g),\ \ \ \{\widetilde{\alpha},\widetilde{\beta}\}=[\widetilde{\alpha,\beta}]_A,\]
for all $f,g\in C^{\infty}(S)$ and $\alpha,\beta\in \Gamma(A)$, where $p:A^*\to S$ denotes the projection, and
$\widetilde{\alpha},\widetilde{\beta}\in C^{\infty}(A^*)$ denote the corresponding fiberwise linear functions on
$A^*$. Consider the gauge transform of $\plin(A)$ by ${p^*(\omega_S)}$:
\[\plin^{p^*(\omega_S)}(A).\]
A priori, this gauge transform is defined only as a Dirac structure on $A^*$, but because of the particular structure of
the linear Poisson structure, $\plin^{p^*(\omega_S)}(A)$ is in fact a well-defined Poisson structure on $A^*$.

Denote by $\mathfrak{g}(A):=\ker(\rho)\subset A$ the isotropy bundle. Consider a linear spitting $\sigma: A\to
\mathfrak{g}(A)$ of the short exact sequence
\begin{equation}\label{short}
0\rmap \mathfrak{g}(A)\rmap A\stackrel{\rho}{\rmap} TS\rmap 0.
\end{equation}
Using the dual of $\sigma$, we regard $\mathfrak{g}(A)^*$ as a subbundle of $A^*$. An open neighborhood $N(A)$ of
$S$ in $\mathfrak{g}(A)^*$ is a Poisson transversal for $\plin^{p^*(\omega_S)}(A)$ (also called a cosymplectic
submanifold in the literature), i.e.\ for each symplectic leaf $(L,\omega_L)$ of $\plin^{p^*(\omega_S)}(A)$, we have
that $N(A)$ is transverse to $L$, and that $L\cap N(A)$ is a symplectic submanifold of $L$. This property allows to pull
back $\plin^{p^*(\omega_S)}(A)$ to a Poisson structure $\pi_A$ on $N(A)$: the leaves of $\pi_A$ are $(L\cap N(A),
\omega_L|_{L\cap N(A)})$, where, as before, $(L,\omega_L)$ is a leaf of $\plin^{p^*(\omega_S)}(A)$. The Poisson
manifold
\[(N(A),\pi_A)\]
represents the second description of the local model. Also, $(S,\omega_S)$, identified with the zero section, is a
symplectic leaf of $\pi_A$ and the induced transitive Lie algebroid $A_S$ is isomorphic to $A$ via the maps
\[A_{S}=T^*\mathfrak{g}(A)^*|_{S}\cong T^*S\oplus \mathfrak{g}(A)\stackrel{(\omega_S^{-1,\sharp}+\sigma)}\rmap A.\]

Different choices of the splitting $\sigma$ give rise to local models that are isomorphic around $S$ by diffeomorphisms
that fix $S$.

We describe now an isomorphism between the two Poisson manifolds resulting from the two descriptions of the local
model. Let $(S,\omega_S)$ be an embedded symplectic leaf of the Poisson manifold $(M,\pi)$. Consider a tubular
neighborhood of $S$, denoted by $\Psi:\nu_S\to M$, and let $\pi_0$ be the corresponding local model from the first
description. Note that the Lie algebroid $A_S:=T^*M|_S$ has isotropy bundle $\mathfrak{g}(A)=\nu_S^*$, and that
the dual of the differential of $\Psi$ along $S$ gives a splitting of the anchor for $A_S$:
\[\sigma:A_S\rmap \nu^*_S,\ \ \sigma:=(d\Psi|_S)^*.\]
Consider the local model $\pi_{A_S}$ on a neighborhood of $S$ in $\nu_S=\mathfrak{g}(A)^*$, constructed with the
aid of $\sigma$. The map $\Psi$ gives a Poisson diffeomorphism in a neighborhood of $S$ between the two descriptions
of the local model:
\[\Psi_*(\pi_{A_S})=\pi_0.\]

We remark that, in general, the submanifold $\mathfrak{g}(A)^*\subset A^*$ is not Poisson transverse everywhere.
Nevertheless, one can always pull back the Poisson structure $\plin^{p^*(\omega_S)}(A)$ to a globally defined Dirac
structure on $\mathfrak{g}(A)^*$, which is Poisson on $N(A)$. Actually, also this second construction works in the
Dirac setting; and the outcome is a second description of the local model of a Dirac structure around a presymplectic leaf.

\subsubsection*{Description 3}

The third description works only when the restricted Lie algebroid is integrable, and as remarked by Vorobjev in
\cite{Vorobjev}, the resulting Poisson manifold appeared already in the work of Montgomery \cite{Montgomery}. The
construction is standard in symplectic geometry as it represents the local form of a Hamiltonian space around the zero
set of the moment map (see e.g.\ \cite{GuilleminSternberg}).

The starting data is: an integrable transitive Lie algebroid $A$ over a symplectic manifold $(S,\omega_S)$. Since $A$ is
transitive, it is isomorphic to $TP/G$ for a principal $G$-bundle $P\to S$. So, the relevant first order data becomes a
principal $G$-bundle $p:P\to S$ over a symplectic manifold $(S,\omega_S)$. Let $\theta\in\Omega^1(P,\mathfrak{g})$
be a principal connection on $P$, where $\mathfrak{g}$ denotes the Lie algebra of $G$. Consider the following closed
2-form on $P\times \mathfrak{g}^*$, which is invariant under the diagonal action of $G$:
\[\Omega=p^*(\omega_S)-d\langle \mu| \theta\rangle, \ \ \textrm{where} \ \mu(p,\xi):=\xi.\]
The open set $\Sigma$ where $\Omega$ is nondegenerate is $G$-invariant and contains $P\times\{0\}$. The action of
$G$ is Hamiltonian with $G$-equivariant moment map $\mu:\Sigma\to \mathfrak{g}^*$. The local model is obtained as
the quotient Poisson manifold:
\[(N(P),\pi_{P}):=(\Sigma,\Omega)/G,\]
where $N(P):=\Sigma/G$ is an open neighborhood of the zero section in the associated coadjoint bundle
$P[\mathfrak{g}^*]:=(P\times \mathfrak{g}^*)/G$. The resulting Poisson structure $\pi_P$ has $(S,\omega_S)$
(regarded as $(P\times\{0\})/G$) as a symplectic leaf, and its restricted Lie algebroid $T^*N(P)|_{S}$ is isomorphic to
$TP/G$.

To relate this construction to the second, note that the isotropy bundle of $A=TP/G$ can be identified with the quotient
$\mathfrak{g}(A)=(P\times \mathfrak{g})/G$, and so $\mathfrak{g}(A)^*=P[\mathfrak{g}^*]$. Also, note that there
is a natural one-to-one correspondence between
\begin{itemize}
\item $\theta\in\Omega^1(P,\mathfrak{g})$, connection 1-forms on $P$,
\item $\sigma:A\to \mathfrak{g}(A)$, linear splittings of the sequence (\ref{short}).
\end{itemize}
Now, under these isomorphisms and this correspondence, the Poisson manifold $(N(P),\pi_P)$, constructed with the aid
of $\theta$, and the Poisson manifold $(N(A),\pi_{A})$, constructed using the corresponding $\sigma$, coincide.

The Poisson manifold $(N(P),\pi_{P})$ is integrable, and we describe below a symplectic groupoid integration it. Since
this result fits into a more general framework, we state the following lemma, which is a direct consequence of results
in \cite{FeOrRa}:

\begin{lemma}
Let $(\Sigma,\Omega)$ be a symplectic manifold endowed with a proper, free Hamiltonian action of a Lie group $G$,
and equivariant moment map $\mu:\Sigma\to \mathfrak{g}^*$. Then the Poisson manifold $\Sigma/G$ is integrable,
and a symplectic Lie groupoid integrating it is
\[(\Sigma\times_{\mu}\Sigma)/G\rightrightarrows \Sigma/G,\]
and the symplectic structure pulls back to $\Sigma\times_{\mu}\Sigma$ as
$(s^*(\Omega)-t^*(\Omega))|_{\Sigma\times_{\mu}\Sigma}$.
\end{lemma}
\begin{proof}
Consider the symplectic groupoid $\Sigma\times\Sigma\rightrightarrows \Sigma$, with symplectic structure
$s^*(\Omega)-t^*(\Omega)$. Then $G$ acts on $\Sigma\times\Sigma$ by symplectic groupoid automorphism with
equivariant moment map $J:=s^*\mu-t^*\mu$, which is also a groupoid 1-cocycle. By Proposition 4.6 in \cite{FeOrRa},
the Marsden-Weinstein reduction
\[(\Sigma\times\Sigma)//G=J^{-1}(0)/G\]
is a symplectic groupoid integrating the Poisson manifold $\Sigma/G$. In our case
$J^{-1}(0)=\Sigma\times_{\mu}\Sigma$, and the symplectic form pulls back to $\Sigma\times_{\mu}\Sigma$ as
$s^*(\Omega)-t^*(\Omega)|_{\Sigma\times_{\mu}\Sigma}$
\end{proof}

In our setting, the lemma shows that the groupoid integrating the local model $(N(P),\pi_P)$ is just the restriction to
$N(P)$ of the action groupoid
\[\mathcal{G}:=(P\times P\times\mathfrak{g}^*)/G\rightrightarrows P[\mathfrak{g}^*],\]
corresponding to the representation of $P\times_GP$ on $P[\mathfrak{g}^*]$. If $P$ is compact, note that $N(P)$
contains arbitrarily small open sets of the form $P[V]:=(P\times V)/G$, where $V$ is a $G$-invariant neighborhood of
$0$ in $\mathfrak{g}^*$. These neighborhoods are $\mathcal{G}$-invariant, and the restriction of $\mathcal{G}$ to
$P[V]$ is $(P\times P\times V)/G$. In particular, all its $s$-fibers are diffeomorphic to $P$. This proves the following:
\begin{proposition}\label{Lemma_small_neighborhoods}
The local model $(N(P),\pi_P)$ associated to a principal bundle $P$ over a symplectic manifold $(S,\omega_S)$ is
integrable by a Hausdorff symplectic Lie groupoid. If $P$ is compact, then there are arbitrarily small invariant open
neighborhoods $U$ of $S$, such that all $s$-fibers over points in $U$ are diffeomorphic to $P$.
\end{proposition}

\subsection{Proof of Theorem \ref{Main_Theorem_intro} $\Rightarrow$ Theorem \ref{The_normal_form_theorem}}
Consider a tubular neighborhood $\Psi:\nu_S\to  M$ of $S$ in $M$, and denote by $\pi_0$ the resulting local model
constructed using the first description. So $\pi_0$ is a Poisson structure on some open neighborhood of $S$, which
coincides with $\pi$ up to first order. On the other hand, $\pi_0$ is isomorphic around $S$ to the local model
$\pi_{A_S}$ corresponding to the transitive Lie algebroid $A_S:=T^*M|_{S}$. By assumption, $A_S$ is integrable, and
so, there is a principal $G$-bundle $P\to S$ such that $A_S\cong TP/G$. Moreover, we can choose $P$ to be compact
with vanishing second de Rham cohomology. By Proposition \ref{Lemma_small_neighborhoods}, for arbitrary small
open neighborhoods $U$ of $S$ in $N(P)$, we have that $(U,\pi_P|_U)$ satisfies the assumption of Theorem
\ref{Main_Theorem_intro}. Since $\pi_P$ is isomorphic to $\pi_{A_S}$ around $S$, and also $\pi_{A_S}$ is isomorphic
to $\pi_0$ around $S$, we conclude that $S$ has arbitrary small neighborhoods $U$ in $M$ for which $(U,\pi_0|_U)$
also satisfies the hypothesis of Theorem \ref{Main_Theorem_intro}. By part (a), $\pi_{0}$ is $C^p$-$C^1$-rigid
around $S$, and by part (b) $\pi_0$ and $\pi$ are Poisson diffeomorphic around $S$. Thus $\pi$ is also
$C^p$-$C^1$-rigid around $S$.

%%%%%%%%%%%%%%%%%%%%%%%%%%%%%%%%%%%%%%%%%%%%%%%%%%%%%%%%%%%%%%%%%%%%%%%%%%%%%%%%
%%%%%%%%%%%%%%%%%%%%%%%%%%%%%%%%%%%%%%%%%%%%%%%%%%%%%%%%%%%%%%%%%%%%%%%%%%%%%%%%
%%%%%%%%%%%%%%%%%%%%%%%%%%%%%%%%%%%%%%%%%%%%%%%%%%%%%%%%%%%%%%%%%%%%%%%%%%%%%%%%
%%%%%%%%%%%%%%%%%%%%%%%%%%%%%%%%%%%%%%%%%%%%%%%%%%%%%%%%%%%%%%%%%%%%%%%%%%%%%%%%
%%%%%%%%%%%%%%%%%%%%%%%%%%%%%%%%%%%%%%%%%%%%%%%%%%%%%%%%%%%%%%%%%%%%%%%%%%%%%%%%
%%%%%%%%%%%%%%%%%%%%%%%%%%%%%%%%%%%%%%%%%%%%%%%%%%%%%%%%%%%%%%%%%%%%%%%%%%%%%%%%
\section{Remarks, examples and applications}\label{section_examples}
%%%%%%%%%%%%%%%%%%%%%%%%%%%%%%%%%%%%%%%%%%%%%%%%%%%%%%%%%%%%%%%%%%%%%%%%%%%%%%%%
%%%%%%%%%%%%%%%%%%%%%%%%%%%%%%%%%%%%%%%%%%%%%%%%%%%%%%%%%%%%%%%%%%%%%%%%%%%%%%%%
%%%%%%%%%%%%%%%%%%%%%%%%%%%%%%%%%%%%%%%%%%%%%%%%%%%%%%%%%%%%%%%%%%%%%%%%%%%%%%%%
%%%%%%%%%%%%%%%%%%%%%%%%%%%%%%%%%%%%%%%%%%%%%%%%%%%%%%%%%%%%%%%%%%%%%%%%%%%%%%%%
%%%%%%%%%%%%%%%%%%%%%%%%%%%%%%%%%%%%%%%%%%%%%%%%%%%%%%%%%%%%%%%%%%%%%%%%%%%%%%%%
%%%%%%%%%%%%%%%%%%%%%%%%%%%%%%%%%%%%%%%%%%%%%%%%%%%%%%%%%%%%%%%%%%%%%%%%%%%%%%%%

In this section we give a list of examples and applications for our two theorems and we also show some links with other results from the
literature.

\subsection{A global conflict}

Theorem \ref{Main_Theorem_intro} does not exclude the case when the Poisson submanifold $S$ is the total space
$M$. In conclusion, a compact Poisson manifold $(M,\pi)$ for which $T^*M$ is integrable by a compact Lie groupoid
whose $s$-fibers have trivial second de Rham cohomology is globally rigid. Nevertheless, this result is useless, since no
such Poisson manifolds exist in dimension greater than one. In the case when the groupoid has 1-connected $s$-fibers,
this conflict was pointed out in \cite{CrFe0}, and we explain below the general case. In symplectic geometry, this
non-rigidity phenomenon is expressed by the fact that, on a compact symplectic manifold $(M,\omega)$, the symplectic
structure allows the smooth deformation $t\omega$, for $t>0$, which is nontrivial because the symplectic volume
changes.

\begin{proposition}\label{Proposition_one_dim}
Consider a compact connected Poisson manifold $(M,\pi)$ for which $T^*M$ is integrable by a compact Lie groupoid
$\mathcal{G}$ whose $s$-fibers have trivial second de Rham cohomology. Then $M$ is at most one-dimensional.
\end{proposition}

In the proof of the proposition we will use the following volume-function:
\begin{lemma}\label{Lemma_volume_holonomy}
Consider the setting of Proposition \ref{Proposition_one_dim}. The set $M^{\mathrm{reg}}$ where $\pi$ has
maximal rank is open and dense. Every regular symplectic leaf $(S,\omega_S)\subset M^{\mathrm{reg}}$ has a finite
holonomy group, which we denote by $\mathrm{Hol}(S)$, and a finite symplectic volume, which we denote by
$\mathrm{Vol}(S)$. The following function is continuous:
\[\mathrm{vh}:M\rmap \mathbb{R},\ \ \  \mathrm{vh}(x):=\left\{ \begin{array}{cc}
                                                                        \mathrm{Vol}(S_x)|\mathrm{Hol}(S_x)|, & x\in  M^{\mathrm{reg}} \\
                                                                        0, & x\notin  M^{\mathrm{reg}}
                                                                      \end{array}\right. ,\]
where $S_x$ denotes the symplectic leaf through $x$.
\end{lemma}

\begin{proof}[Proof of Proposition \ref{Proposition_one_dim}]

By Lemma \ref{Tame_Vanishing_Lemma} in the appendix, the second Poisson cohomology of $(M,\pi)$ vanishes. In
particular, the class $[\pi]$ is trivial, so there exists a vector field $X$ such that $L_X(\pi)=\pi$. This implies that the
flow of $X$ gives a Poisson diffeomorphism
\begin{equation}\label{eEQ_545}
\varphi_X^t:(M,\pi) \diffto (M,e^{-t}\pi).
\end{equation}
This and the Poisson geometric description of $\mathrm{vh}$ imply that $\mathrm{vh}\circ
\varphi_X^t=e^{tk}\mathrm{vh}$, where $2k$ denotes the maximal rank of $\pi$. By Lemma
\ref{Lemma_volume_holonomy}, $\mathrm{vh}$ is bounded, hence $\pi=0$. If $\pi=0$, then $\mathcal{G}\to M$ is a
bundle of tori, so by the cohomological condition, its fibers are at most one dimensional. Hence $M$ is at most one
dimensional as well.
\end{proof}

\begin{proof}[Proof of Lemma \ref{Lemma_volume_holonomy}]
Clearly, $M^{\mathrm{reg}}$ is open. To show that $M^{\mathrm{reg}}$ is dense, by connectedness of $M$, it
suffices to show that its closure $\overline{M^{\mathrm{reg}}}$ is open. This follows from the following property of
$\pi$, which we prove below: every leaf has a saturated neighborhood $U$, such that $U^{\mathrm{reg}}$ (i.e.\ the
regular part of $(U,\pi|_{U})$) is dense in $U$.

Let $(S,\omega_S)$ be a symplectic leaf of $M$. Since $\mathcal{G}|_{S}$ integrates $A_S$, by Theorem
\ref{The_normal_form_theorem}, the local model holds around $S$. So, for a compact, connected principal $G$-bundle
$P$, we have that $(M,\pi)$ is Poisson isomorphic around $S$ to an open set around $S$ in
\[(N(P),\pi_P)=(\Sigma,\Omega)/G, \]
where $\Sigma\subset P\times\mathfrak{g}^*$, $\Omega=p^*(\omega_S)-d\langle \mu|\theta\rangle$, $\theta$ is a
principal connection on $P$, and $\mu:\Sigma\to \mathfrak{g}^*$, $\mu(p,\xi)=\xi$ is an equivariant moment for the
action of $G$. The symplectic leaves of $\pi_P$ are of the form
\[(O_{\xi},\omega_{\xi}),\ \ \ O_{\xi}:=P\times_{G}(G\cdot \xi), \ \ \xi\in\mathfrak{g}^*,\]
hence they are the base of the principal $G_{\xi}$-bundle
\[p_{\xi}:P\times\{\xi\}\rmap O_{\xi},\]
where $G_{\xi}$ is the stabilizer of $\xi$, and the symplectic structure is determined by
\begin{equation}\label{eEQ_volume}
p_{\xi}^*(\omega_{\xi})=\Omega|_{P\times \{\xi\}}.
\end{equation}
This last equation follows from the fact that the action is Hamiltonian, and therefore, the symplectic leaves are
canonically isomorphic to the reduced spaces
\[\mu^{-1}(\xi)// G_{\xi}=(P\times \{\xi\})// G_{\xi}.\]

We will show that $\mathrm{vh}$ extends to a continuous map on $P\times_{G}\mathfrak{g}^*$. Let $T$ be a
maximal torus in $G$ and let $\mathfrak{t}$ be its Lie algebra. By compactness of $G$, we can consider an invariant
metric on $\mathfrak{g}$. This metric allows us to regard $\mathfrak{t}^*$ as a subspace in $\mathfrak{g}^*$ (i.e.\
the orthogonal to $\mathfrak{t}^{\circ}$), and it gives an isomorphism between the adjoint and the coadjoint
representation which sends $\mathfrak{t}$ to $\mathfrak{t}^*$. For the adjoint representation it is well know (see
e.g.\ \cite{DK}) that every orbit hits $\mathfrak{t}$, hence also every orbit of the coadjoint action hits
$\mathfrak{t}^*$. An element $\xi\in\mathfrak{t}^*$ is called regular if $\mathfrak{g}_{\xi}=\mathfrak{t}$, where
$\mathfrak{g}_{\xi}$ is the Lie algebra of $G_{\xi}$. Denote by $\mathfrak{t}^{*\mathrm{reg}}$ the set of regular
elements. Then $\mathfrak{t}^{*\mathrm{reg}}$ is open and dense in $\mathfrak{t}^*$ and it coincides with the set
of elements $\xi$ for which $G_{\xi}/T$ is finite (see e.g.\ \cite{DK}). Thus, for $\xi\in \mathfrak{t}^*$, a leaf
$O_{\xi}$ has maximal dimension if and only if $\xi\in \mathfrak{t}^{*\mathrm{reg}}$, hence the regular part of
$\pi_P$ equals
\[N(P)^{\mathrm{reg}}=(P \times_G G\cdot\mathfrak{t}^{*\mathrm{reg}})\cap N(P).\]
This implies also the claims made about $M^{\mathrm{reg}}$ at the beginning of the proof.

Now, we fix $\xi\in \mathfrak{t}^{*\mathrm{reg}}$. By Theorem 3.7.1 \cite{DK} we have that
$(G^{\circ})_{\xi}=T$, therefore also $(G_{\xi})^{\circ}=T$. Since $P$ is connected, the last terms in the long exact
sequence in homotopy associated to $p_{\xi}$ are
\begin{equation}\label{eEQ_last_term_short}
\ldots\rmap \pi_1(O_{\xi})\stackrel{\Theta}{\rmap} \pi_0(G_{\xi})\rmap 1.
\end{equation}
Thus we obtain a surjective group homomorphism $\Theta:\pi_1(O_{\xi})\to G_{\xi}/T$. Explicitly, let $[q,\xi]\in
O_{\xi}$ and $\gamma(t)$ be a closed loop at this point. Consider $\widetilde{\gamma}(t)$ a lift of $\gamma$ to $P$,
with $\widetilde{\gamma}(0)=q$. Since $p_{\xi}(\widetilde{\gamma}(1),\xi)=[q,\xi]$, it follows that
$\widetilde{\gamma}(1)=qg$, for some $g\in G_{\xi}$. The map in (\ref{eEQ_last_term_short}) is given by
$\Theta(\gamma)=[g]\in G_{\xi}/T$.

Next, we compute the holonomy group of $O_{\xi}$. Notice first that
\[T_{\xi}(G\cdot{\xi})=\mathfrak{g}_{\xi}^{\circ}=\mathfrak{t}^{\circ}\subset \mathfrak{g}^*\cong T_{\xi}\mathfrak{g}^*,\]
and, since $\mathfrak{t}^{*}=(\mathfrak{t}^{\circ})^{\perp}$, it follows that $\xi+\mathfrak{t}^*$ is transverse
at $\xi$ to the coadjoint orbit. Hence also the submanifold
\[\mathcal{T}:=\{q\}\times (\xi+\mathfrak{t}^*)\subset P\times_G\mathfrak{g}^*\]
is transverse to $O_{\xi}$ at $[q,\xi]$. Let $\gamma$ be a loop in $O_{\xi}$ based at $[q,\xi]$, and let
$\widetilde{\gamma}$ be a lift to $P$. Observe that, for $\eta\in\mathfrak{t}^*$, the path
\[t\mapsto [\widetilde{\gamma}(t),\xi+\eta]\in P\times_G\mathfrak{g}^*,\]
stays in the leaf $O_{\xi+\eta}$, and therefore, the map $[q,\xi+\eta]\mapsto [\widetilde{\gamma}(1),\xi+\eta]$ is the
holonomy action of $\gamma$ on $\mathcal{T}$. Writing $\widetilde{\gamma}(1)=qg$, for $g\in G_{\xi}$, it follows
that the holonomy of $\gamma$ is corresponds to the action of $g=\Theta(\gamma)$ on $\mathfrak{t}^*$. This and the
surjectivity of $\Theta$ imply that
\begin{equation}\label{EQ144}
\mathrm{Hol}(O_{\xi})\cong G_{\xi}/Z_{G}(T),
\end{equation}
where $Z_G(T)$ denotes the set of elements in $G$ which commute with all elements in $T$. In particular, the
holonomy groups are finite.

Since every coadjoint orbit hits $\mathfrak{t}^*$, it follows that the map $P\times\mathfrak{t}^*\to P\times_G
\mathfrak{g}^*$ is onto. Since this map is $T$-invariant, the induced map is smooth and onto:
\[\mathrm{pr}: (P/T)\times\mathfrak{t}^*\rmap  P\times_{G}\mathfrak{g}^*.\]
Clearly, $\mathrm{pr}$ is a proper map, therefore, to show that $\mathrm{vh}$ is continuous, it suffices to show that
$\mathrm{vh}\circ \mathrm{pr}$ extends continuously. Note that, for $\xi\in \mathfrak{t}^{*\mathrm{reg}}$, the
map $\mathrm{pr}$ restricts to a $|G_{\xi}/T|$-covering projection of the leaf
\[\overline{p}_{\xi}:(P/T)\times \{\xi\}\rmap P/G_{\xi}\cong O_{\xi}.\]
Thus, using also (\ref{EQ144}), we have that
\begin{align*}
\mathrm{Vol}((P/T)\times\{\xi\},\overline{p}_{\xi}^*(\omega_{\xi}))&=|G_{\xi}/T|\mathrm{Vol}(O_{\xi},\omega_{\xi})=\frac{|G_{\xi}/T|}{|G_{\xi}/Z_G(T)|}\mathrm{vh}(O_{\xi})=\\
 &=|Z_G(T)/T|\mathrm{vh}(O_{\xi}).
\end{align*}
Hence it suffices to show that the map
\begin{equation}\label{eEQ_volume_1}
\mathfrak{t}^*\ni\xi\mapsto \mathrm{Vol}((P/T)\times\{\xi\},\overline{p}_{\xi}^*(\omega_{\xi}))
\end{equation}
is continuous. By (\ref{eEQ_volume}), we have that the pullback of $\overline{p}_{\xi}^*(\omega_{\xi})$ to
$P\times\{\xi\}$ is given by $\Omega|_{P\times\{\xi\}}=p^*(\omega_S)-\langle \xi,d\theta\rangle$, in particular it
depends smoothly on $\xi$. Hence also $\overline{p}_{\xi}^*(\omega_{\xi})$ depends smoothly on $\xi$, and so the
map (\ref{eEQ_volume_1}) is continuous. To conclude the proof, we have to check that this map vanishes for $\xi\notin
\mathfrak{t}^{*\mathrm{reg}}$. For such $\xi$, since $\mathrm{dim}(G_{\xi}/T)>0$, we have that
\[2l=\mathrm{dim}(O_{\xi})=\mathrm{dim}(P/G_{\xi})<\mathrm{dim}(P/T)=2k.\]
This finishes the proof, since
\[\wedge^{k}\overline{p}_{\xi}^*(\omega_{\xi})=\overline{p}_{\xi}^*(\wedge^k\omega_{\xi})=0.\qedhere\]
\end{proof}

\subsection{$C^p$-$C^1$-rigidity and isotopies}

In the definition of $C^p$-$C^1$-rigid, we may assume that the maps $\psi_{\widetilde{\pi}}$ are isotopic to the
inclusion $\mathrm{Id}_{\overline{O}}$ of $\overline{O}$ in $M$, through a path of maps in
$C^{\infty}(\overline{O},M)$ that extend to embeddings on some neighborhood of $\overline{O}$. This follows from
the $C^p$-$C^1$-continuity of $\psi$ and the fact that $\mathrm{Id}_{\overline{O}}$ has a path-connected
$C^1$-neighborhood in $C^{\infty}(\overline{O},M)$ consisting of such maps.

\subsection{A comparison with the local normal form theorem from \cite{CrMa}}

Part (a) of Theorem \ref{The_normal_form_theorem} is a slight improvement of the normal form result from
\cite{CrMa}. Both theorems require the same conditions on a Lie groupoid, for us this groupoid could be any
integration of $A_S$, but in \emph{loc.cit.} it has to be the $s$-fiber 1-connected integration. In subsections
\ref{subsection_fixed_points}, respectively \ref{Trivial_symplectic_foliations}, we will study two extreme examples
which already reveal the wider applicability of Theorem \ref{The_normal_form_theorem}: the case of fixed points and
the case of regular Poisson structures whose underling foliation is simple.

\subsection{The case of fixed points}\label{subsection_fixed_points}
Consider a Poisson manifold $(M,\pi)$ and let $x\in M$ be a fixed point of $\pi$. In a chart centered at $x$, we write
\begin{equation}\label{EQ_poisson_bivector}
\pi=\sum_{i,j}\frac{1}{2}\pi_{i,j}(x)\frac{\partial}{\partial x_i}\wedge\frac{\partial}{\partial x_j}, \ \textrm{with}\ \pi_{i,j}(0)=0.
\end{equation}
The local model of $\pi$ around $0$ is given by its first jet at $0$:
\[\sum_{i,j,k}\frac{1}{2}\frac{\partial\pi_{i,j}}{\partial x_k}(0)x_k\frac{\partial}{\partial x_i}\wedge\frac{\partial}{\partial x_j}.\]
The coefficients $C_{i,j}^k:=\frac{\partial\pi_{i,j}}{\partial x_k}(0)$ are the structure constants of the isotropy Lie
algebra $\mathfrak{g}_x$ (see the Introduction). To apply Theorem \ref{The_normal_form_theorem} in this setting, we
need that $\mathfrak{g}_x$ be integrable by a compact Lie group with vanishing second de Rham cohomology. Such
Lie algebras have the following structure:

\begin{lemma}\label{Lemma_cohomology_lie_groups}
A Lie algebra $\mathfrak{g}$ is integrable by a compact Lie group with vanishing second de Rham cohomology if and
only if it is of the form
\[\mathfrak{g}=\mathfrak{k} \textrm{ or }\mathfrak{g}=\mathfrak{k}\oplus\mathbb{R},\]
where $\mathfrak{k}$ is a semisimple Lie algebra of compact type.
\end{lemma}
\begin{proof}
It is well known that a compact Lie algebra $\mathfrak{g}$ (i.e.\ a Lie algebra that is integrable by a compact Lie group) decomposes as a product %(see \cite{DK})
$\mathfrak{g}=\mathfrak{k}\oplus\mathfrak{z}$, where $\mathfrak{k}=[\mathfrak{g},\mathfrak{g}]$ is semisimple
of compact type and $\mathfrak{z}$ is the center of $\mathfrak{g}$. Hence, the Eilenberg-Chevalley complex of
$\mathfrak{g}$ is the tensor product of the respective complexes of $\mathfrak{k}$ and $\mathfrak{z}$. Therefore,
by the K\"unneth formula, $H^{\bullet}(\mathfrak{g})\cong H^{\bullet}(\mathfrak{k})\otimes
H^{\bullet}(\mathfrak{z})$. Since $\mathfrak{k}$ is semisimple, by Whitehead's Lemma $H^1(\mathfrak{k})=0$
and $H^2(\mathfrak{k})=0$, and since $\mathfrak{z}$ is abelian,
$H^{\bullet}(\mathfrak{z})=\bigwedge^{\bullet}\mathfrak{z}^*$. Thus, we obtain that
\begin{equation}\label{EQ_5}
H^{2}(\mathfrak{g})\cong \bigwedge\nolimits^{\!2}\mathfrak{z}^*.
\end{equation}

Consider now $G$ any compact connected integration of $\mathfrak{g}$. The cohomology of $G$ can be computed
using left invariant differential forms, therefore $H^{\bullet}(G)\cong H^{\bullet}(\mathfrak{g})$. By (\ref{EQ_5}),
we obtain that $H^2(G)=0$ is equivalent to $\mathrm{dim}(\mathfrak{z})\leq 1$.
\end{proof}

So, for fixed points, Theorem \ref{The_normal_form_theorem} gives:
\begin{corollary}\label{Corollary_Conn_gen}
Let $(M,\pi)$ be a Poisson manifold with a fixed point $x$ for which the isotropy Lie algebra $\mathfrak{g}_x$ is
compact and its center is at most one-dimensional. Then $\pi$ is rigid around $x$, and an open set around $x$ is
Poisson diffeomorphic to a neighborhood of $0$ in the linear Poisson manifold $(\mathfrak{g}_x^*,\pi_{\mathfrak{g}_x})$.
\end{corollary}

The linearization result in the semisimple case is Conn's theorem \cite{Conn} and the case when the isotropy has a
one-dimensional center is a consequence of the smooth Levi decomposition theorem of Monnier and Zung
\cite{Monnier}.

This fits into Weinstein's notion of a nondegenerate Lie algebra \cite{Wein}. Recall that a Lie algebra $\mathfrak{g}$
is called \textbf{nondegenerate}, if every Poisson structure which has isotropy Lie algebra $\mathfrak{g}$ at a fixed
point $x$, is Poisson-diffeomorphic around $x$ to the linear Poisson structure $(\mathfrak{g}^*,\pi_{\mathfrak{g}})$
around $0$.

A Lie algebra $\mathfrak{g}$, for which $\pi_{\mathfrak{g}}$ is rigid around $0$, is necessarily nondegenerate. To
see this, consider a Poisson bivector $\pi$ given in local coordinates by (\ref{EQ_poisson_bivector}), and whose
linearization at $0$ is $\pi_{\mathfrak{g}}$. The path of Poisson bivectors $\pi_t$ from the first description of the
local model (\ref{path}) satisfies $\pi_1=\pi$ and $\pi_0=\pi_\mathfrak{g}$, and for $t> 0$ is given by:
\[\pi_t:=t\mu_t^{*}(\pi)=\sum_{i,j}\frac{1}{2t}\pi_{i,j}(tx)\frac{\partial}{\partial x_i}\wedge\frac{\partial}{\partial x_j},\]
where $\mu_t$ denotes multiplication by $t>0$. If $\pi_\mathfrak{g}$ is rigid around $0$, then, for some $r>0$ and
some $t>0$, there is a Poisson isomorphism between
\[\psi: (B_{r},\pi_t)\diffto (\psi(B_{r}),\pi_{\mathfrak{g}}).\]
Now $\xi:=\psi(0)$ is a fixed point of $\pi_{\mathfrak{g}}$, which is the same as an element in
$(\mathfrak{g}/[\mathfrak{g},\mathfrak{g}])^{*}$. It is easy to see that translation by $\xi$ is a Poisson isomorphism
of $\pi_{\mathfrak{g}}$, therefore, replacing $\psi$ with $\psi-\xi$, we may assume that $\psi(0)=0$. Linearity of
$\pi_{\mathfrak{g}}$ implies that $\mu_t^*(\pi_{\mathfrak{g}})=\frac{1}{t}\pi_{\mathfrak{g}}$, and therefore
\[\pi=\frac{1}{t}\mu_{1/t}^*(\pi_t)=\frac{1}{t}\mu_{1/t}^*(\psi^*(\pi_{\mathfrak{g}}))=\mu_{1/t}^*\circ\psi^*\circ\mu_t^*(\pi_{\mathfrak{g}}).\]
Hence, $\pi$ is linearizable by the map
\[\mu_t\circ\psi\circ\mu_{1/t}:(B_{tr},\pi)\rmap (t\psi(B_r),\pi_{\mathfrak{g}}),\]
which maps $0$ to $0$. This shows that $\mathfrak{g}$ is nondegenerate.

\subsection{The Poisson sphere in $\mathfrak{g}^*$}

Let $\mathfrak{g}$ be a semisimple Lie algebra of compact type and let $G$ be the compact, 1-connected Lie group integrating it. The linear
Poisson structure $(\mathfrak{g}^*,\pi_{\mathfrak{g}})$ is integrable by the symplectic groupoid $(T^*G,\omega_{\mathrm{can}})\rightrightarrows
\mathfrak{g}^*$, with source and target map given by left and right trivialization. All $s$-fibers of $T^*G$ are diffeomorphic to $G$ and, since
$H^2(G)=0$, we can apply Theorem \ref{Main_Theorem_intro} to any compact Poisson submanifold of $\mathfrak{g}^*$. An example of such a
submanifold is the sphere $\mathbb{S}(\mathfrak{g}^*)\subset \mathfrak{g}^*$ with respect to some invariant metric. We obtain the following
result, whose formal version appeared in \cite{MaFormal} and served as an inspiration.
\begin{proposition}\label{proposition_rigidity_of_spheres}
Let $\mathfrak{g}$ be a semisimple Lie algebra of compact type and denote by $\mathbb{S}(\mathfrak{g}^*)\subset\mathfrak{g}^*$ the unit sphere
centered at the origin with respect to some invariant inner product. Then $\pi_{\mathfrak{g}}$ is $C^p$-$C^1$-rigid around
$\mathbb{S}(\mathfrak{g}^*)$ and, up to isomorphism, it is determined around $\mathbb{S}(\mathfrak{g}^*)$ by its first order jet.
\end{proposition}
Using this rigidity result, one can describe an open set around
$\pi_{\mathbb{S}}:=\pi_{\mathfrak{g}}|_{\mathbb{S}(\mathfrak{g}^*)}$ in the moduli space of all Poisson
structures on the sphere $\mathbb{S}(\mathfrak{g}^*)$. More precisely, any Poisson structure on
$\mathbb{S}(\mathfrak{g}^*)$ that is $C^p$-close to $\pi_{\mathbb{S}}$ is Poisson diffeomorphic to one of the type
$f\pi_{\mathbb{S}}$, where $f$ is a positive Casimir function. If the metric is $Aut(\mathfrak{g})$-invariant, then two
structures of this type $f_1\pi_{\mathbb{S}}$, $f_2\pi_{\mathbb{S}}$ are isomorphic if and only if $f_1=f_2\circ
\chi^*$, for some outer automorphism $\chi$ of the Lie algebra $\mathfrak{g}$. The details are given in
\cite{MarDef}.

\subsection{Relation with stability of symplectic leaves} \label{subsection_relation_stability}

Recall from \cite{CrFe-stab} that a symplectic leaf $(S,\omega_S)$ of a Poisson manifold $(M,\pi)$ is said to be
$C^p$-\textbf{strongly stable} if for every open set $U$ containing $S$ there exists an open neighborhood
$\mathcal{V}\subset \mathfrak{X}^2(U)$ of $\pi|_{U}$ with respect to the compact-open $C^p$-topology, such that
every Poisson structure in $\mathcal{V}$ has a leaf symplectomorphic to $(S,\omega_S)$. Recall also
\begin{startheorem}[Theorem 2.2 in \cite{CrFe-stab}]
If $S$ is compact and the Lie algebroid $A_S:=T^*M|_{S}$ satisfies $H^2(A_S)=0$, then $S$ is a strongly stable leaf.
\end{startheorem}

If $\pi$ is $C^p$-$C^1$-rigid around $S$, then $S$ is a strongly stable leaf. Also, the hypothesis of our Theorem
\ref{The_normal_form_theorem} imply those of Theorem 2.2 in \emph{loc.cit.} To see this, let $P\to S$ be a principal
$G$-bundle for which $A_S\cong TP/G$. Then
\[H^{\bullet}(A_S)\cong H^{\bullet}\big(\Omega(P)^G\big).\]
If $G$ is compact then, by averaging primitives, one easily shows that the inclusion $\Omega^{\bullet}(P)^G \subset
\Omega^{\bullet}(P)$ induces an injection $H^{\bullet}(\Omega(P)^G)\to H^{\bullet}(P)$. So $H^2(P)=0$ implies
that $H^2(A_S)=0$.

On the other hand, $H^2(A_S)=0$ doesn't imply rigidity, counterexamples can be found even for fixed points. Weinstein proves \cite{Wein4} that a
noncompact semisimple Lie algebra $\mathfrak{g}$ of real rank at least two is degenerate, so $\pi_{\mathfrak{g}}$ is not rigid (see subsection
\ref{subsection_fixed_points}). However, $0$ is a stable point for $\pi_{\mathfrak{g}}$, because by Whitehead's Lemma $H^2(\mathfrak{g})=0$.

According to Theorem 2.3 in \cite{CrFe-stab}, the condition $H^2(A_S)=0$ is also necessary for strong stability of the
symplectic leaf $(S,\omega_S)$ for Poisson structures of ``first order'', i.e.\ for Poisson structures which are isomorphic
to their local model around $S$. So, for this type of Poisson structures, $H^2(A_S)=0$ is also necessary for rigidity.

For regular Poisson structures whose underlying foliation is simple, we will prove below that the hypotheses of
Theorem \ref{The_normal_form_theorem} and of Theorem 2.2 \emph{loc.cit.} are equivalent.

\subsection{Simple symplectic foliations}\label{Trivial_symplectic_foliations}
We will discuss now rigidity and linearization of regular Poisson structures $\pi$ on $S\times\mathbb{R}^n$ with
symplectic leaves
\[\big(S\times\{y\},\omega_{y}:=\pi|_{S\times\{y\}}^{-1}\big), \ \ \ y\in \mathbb{R}^n,\]
where $\{\omega_y\}_{y\in\mathbb{R}^n}$ is a smooth family of symplectic forms on $S$. Let $(S,\omega_S)$ be the
symplectic leaf for $y=0$. To construct the local model around $S$, we use the first description. The path of Poisson
structure $\pi_t$ from (\ref{path}), for $t\neq 0$, corresponds to the following family of 2-forms on $S$:
\[\omega^t_y:=\omega_S+\frac{\omega_{ty}-\omega_S}{t}.\]
Therefore, the local model around $S$ corresponds to the family of 2-forms:
\[j^1_{S}(\omega)_{y}:=\omega_S+\delta_S\omega_y,\]
where $\delta_S\omega_y$ is the ``vertical derivative'' of $\omega$ at $S$:
\[\delta_S\omega_y:=\frac{d}{d\epsilon}\omega_{\epsilon y}|_{\epsilon=0}=y_1\omega_1+\ldots+y_n\omega_n.\]
The local model is defined on an open set $U\subset S\times\mathbb{R}^n$ containing $S$, such that
$j^1_{S}(\omega)_y$ is nondegenerate along $U\cap (S\times\{y\})$. Using the splitting
$T^*(S\times\mathbb{R}^n)|_{S}=T^*S\times\mathbb{R}^n$ and the isomorphism of $\omega_S^{\sharp}:TS\diffto
T^*S$, we identify $A_S\cong TS\times\mathbb{R}^n$. Under this identification, the Lie bracket becomes:
\begin{align}\label{EQ_Bracket_AS}
[(&X,f_1,\ldots,f_n),(Y,g_1,\ldots,g_n)]=\\
\nonumber &=([X,Y],X(g_1)-Y(f_1)+\omega_1(X,Y),\ldots,X(g_n)-Y(f_n)+\omega_n(X,Y)).
\end{align}
The conditions in Theorem \ref{The_normal_form_theorem} become more computable in this case.
\begin{lemma}\label{L_trivial_foliation}
If $S$ is compact, then the following are equivalent:
\begin{enumerate}[(a)]
\item $A_S$ is integrable by a compact principal bundle $P$, with $H^2(P)=0$,
\item $H^2(A_S)=0$,
\item The cohomological variation $[\delta_S\omega]:\mathbb{R}^n\to H^2(S)$, $y\mapsto [\delta_S\omega_y]$, satisfies:
\begin{enumerate}
\item[(c$_1$)] it is surjective,
\item[(c$_2$)] its kernel is at most 1-dimensional,
\item[(c$_3$)] the map $H^1(S)\otimes \mathbb{R}^n\to H^3(S)$, $\eta\otimes y\mapsto \eta\wedge [\delta_S\omega_y]$ is injective.
\end{enumerate}
\end{enumerate}
\end{lemma}
\begin{proof}
The complex computing $H^{\bullet}(A_S)$ can be identified with
\[\Omega^{k}(A_S):=\bigoplus_{p+q=k} \Omega^p(S)\otimes \bigwedge\nolimits^{\!q}\mathbb{R}^n,\]
endowed with the differential
\[d_{A_S}(\alpha\otimes w)= (d\alpha)\otimes w+(-1)^{p+1} \alpha\wedge \delta_S\omega(w),\]
for $\alpha \in \Omega^p(S)$ and $w\in \bigwedge\nolimits^{\!q}\mathbb{R}^n$, where the map \[\delta_S\omega:
\bigwedge\nolimits^{\!q}\mathbb{R}^n \rmap \Omega^2(S)\otimes \bigwedge\nolimits^{\!q-1}\mathbb{R}^n \] is
induced by the vertical derivative of $\omega$:
\[\delta_S\omega (y_1\wedge\ldots \wedge y_q)=\sum_{i=1}^q(-1)^{i-1}\delta_S\omega_{y_i}\otimes y_1\wedge\ldots\wedge y_{i-1}\wedge y_{i+1}\wedge\ldots \wedge y_q.\]
Consider the filtration $F^{p}\Omega^{\bullet}(A_S):= \Omega^{p}(S)\wedge \Omega^{{\bullet}-p}(A_S)$ of this
complex, and the corresponding spectral sequence (for general constructions of spectral sequences for computing Lie
algebroid cohomology see e.g.\ \cite{MackenzieGT}). We have that
\[E_2^{p,q}= H^p(S)\otimes \bigwedge\nolimits^{\!q}\mathbb{R}^n \Rightarrow H^{p+q}(A_S),\]
and the differentials on the second page $E_2$ are given by
\[[\delta_S\omega]:H^{p}(S)\otimes \bigwedge \nolimits^{\!q}\mathbb{R}^n \rmap H^{p+2}(S)\otimes \bigwedge \nolimits^{\!q-1}\mathbb{R}^n,\]
\[[\delta_S\omega]([\alpha]\otimes w)=(-1)^{p+1}[\alpha\wedge \delta_S\omega(w)].\]
In total degree 2, the cohomology of $E_2$ is given by:
\begin{align*}
E_3^{2,0}&:=\textrm{coker}\big{(}[\delta_S\omega]: \mathbb{R}^n\to H^2(S)\big{)}, \\
E_3^{1,1}&:=\textrm{ker}\big{(}[\delta_S\omega]: H^1(S)\otimes \mathbb{R}^n\to H^3(S)\big{)}, \\
E_3^{0,2}&:=\textrm{ker}\big{(}[\delta_S\omega]: \bigwedge \nolimits^{\!2}\mathbb{R}^n \to H^{2}(S)\otimes \mathbb{R}^n\big{)}.
\end{align*}
We claim that the last group is also given by:
\begin{equation}\label{EQ_10}
E_3^{0,2}=\bigwedge \nolimits^{\!2} \textrm{ker}\big{(}[\delta_S\omega]: \mathbb{R}^n \to H^{2}(S)\big{)}.
\end{equation}
This is based on a simple result from linear algebra: namely, if $A: V\to W$ is a linear map between finite dimensional
vector spaces, then the kernel of the map
\[\bigwedge \nolimits^{\!2}V \to W\otimes V, \ \ \ \ \ v_1\wedge v_2 \mapsto  A (v_1)\otimes v_2- A(v_2)\otimes v_1\]
is given by $\bigwedge^2 \textrm{ker}(A)$.

Next, we claim that the cohomology can be read from the third page:
\begin{equation}\label{EQ_12}
H^2(A_S)=E^{2,0}_3\oplus E_3^{1,1}\oplus E_3^{0,2}.
\end{equation}
Since $E^{2,0}_3=E^{2,0}_{\infty}$ and $E^{1,1}_3=E^{1,1}_{\infty}$, this is equivalent to the edge morphism
$e_F: E_{\infty}^{0,2}\to E_3^{0,2}$ being an isomorphism, or to surjectivity of the map
\begin{equation}\label{EQ_11}
H^2(A_S)\rmap E_3^{0,2}, \ \ [\alpha]\mapsto [\alpha^{0,2}],
\end{equation}
where $\alpha^{0,2}$ denotes the component in $\bigwedge^2\mathbb{R}^n$ of the closed form $\alpha\in
\Omega^2(A_S)$. By (\ref{EQ_10}), it suffice to show that every element of the form $v\wedge w$, with
$[\delta_S\omega_v]=[\delta_S\omega_w]=0$ is in the range of this map. Writing $\delta_S\omega_{v}=d\eta$,
$\delta_S\omega_{w}=d\theta$, for $\eta,\theta\in\Omega^1(S)$, one easily checks that \[\xi:=(\eta\wedge \theta,
\eta\otimes w-\theta\otimes v, v\wedge w)\in\Omega^2(A_S)\] is closed. Thus, the map in (\ref{EQ_11}) maps $[\xi]$ to
$v\wedge w$, which proves that it is surjective. Hence (\ref{EQ_12}) holds.

The three conditions in (c) are equivalent to the vanishing of the three components of $E_3^2$. So, by (\ref{EQ_12}),
(b) and (c) are equivalent.

The fact that (a) implies (b) was explained in subsection \ref{subsection_relation_stability}.

We prove now that (b) and (c) imply (a). Part (c$_1$) implies that, by taking a different basis of $\mathbb{R}^n$, we
may assume that $[\omega_1],\ldots,[\omega_n]\in H^2(S,\mathbb{Z})$. Let $P\to S$ be a principal $T^n$ bundle with
connection form $(\theta_1,\ldots,\theta_n)$ and curvature form $(-\omega_1,\ldots,-\omega_n)$. We claim that the Lie
algebroid $TP/T^n$ is isomorphic to $A_S$. A section of $TP/T^n$ is the same as a $T^n$-invariant vector field on
$P$, and as such, it can be decomposed uniquely as
\[\widetilde{X}+\sum f_i\partial_{\theta_i},\]
where $\widetilde{X}$ is the horizontal lift of a vector field $X$ on $S$, $f_1,\ldots,f_n$ are smooth functions on $S$,
and $\partial_{\theta_i}$ is the unique vertical vector field on $P$ which satisfies
\[\theta_j(\partial_{\theta_i})=\delta_{i,j}.\]
Using (\ref{EQ_Bracket_AS}) for the bracket on $A_S$ and that $d\theta_i=-p^*(\omega_i)$, it is straightforward to
check that the following map is a Lie algebroid isomorphism
\[TP/T^n\diffto A_S, \ \  \widetilde{X}+\sum f_i\partial_{\theta_i}\mapsto (X,f_1,\ldots,f_n).\]
Since $T^n$ compact and connected, using averaging, one shows that the complexes
\[\big(\Omega^\bullet(P)^{T^n},d\big), \ \ \ \big(\Omega^\bullet(P),d\big)\]
are quasi-isomorphic; in particular $H^2(P)\cong H^2(A_S)$. By (b), $H^2(P)=0$, and so $P$ satisfies the conditions
from (a). This finishes the proof.
\end{proof}

So, in the case of simple foliations, Theorem \ref{The_normal_form_theorem} becomes:

\begin{corollary}\label{corollary_trivial_foliations}
Let $\{\omega_y\in\Omega^2(S)\}_{y\in\mathbb{R}^n}$ be a smooth family of symplectic structures on a compact
manifold $S$. If the cohomological variation at $0$
\[[\delta_S\omega]:\mathbb{R}^n\rmap H^2(S),\]
satisfies the conditions from Lemma \ref{L_trivial_foliation}, then the Poisson manifold with leaves
\[(S\times\mathbb{R}^n,\{\omega_y^{-1}\}_{y\in\mathbb{R}^n})\]
is isomorphic to its local model at $S\times\{0\}$, and is $C^p$-$C^1$-rigid around this leaf.
\end{corollary}

For simple symplectic foliations Lemma \ref{L_trivial_foliation} shows that the condition in Theorem
\ref{The_normal_form_theorem} are equivalent to the vanishing of $H^2(A_S)$. This is precisely the assumption of the
stability result from Theorem 2.2 in \cite{CrFe-stab}. In \emph{loc.cit.}, it is also proven that under this assumption
there exists a smoothly parameterized family of symplectic leaves near $S$ that are symplectomorphic to
$(S,\omega_S)$. To describe the parameter space, consider the cohomology $H^{\bullet}(A_S;S)$ of the quotient
complex (here we use the notation from the proof of Lemma \ref{L_trivial_foliation}):
\[\Omega^{\bullet}(A_S;S):=\Omega^{\bullet}(A_S)/\Omega^{\bullet}(S),\]
and consider the canonical map induced by the quotient map:
\[\Phi:H^{\bullet}(A_S)\rmap H^{\bullet}(A_S;S).\]
Theorem 2.2 \cite{CrFe-stab} states that every Poisson structure near $\pi$ has a family of symplectic leaves
symplectomorphic to $(S,\omega_S)$, which is smoothly parameterized by the image of the map
\[\Phi:H^{1}(A_S)\rmap H^{1}(A_S;S).\]
Applying the same techniques as in the proof of Lemma \ref{L_trivial_foliation}, this map can be computed as follows:
\begin{lemma}
With the notation from the proof of Lemma \ref{L_trivial_foliation}, we have that
\[H^1(A_S)\cong H^1(S)\oplus \ker\big([\delta_S\omega]:\mathbb{R}^n\to H^2(S)\big), \ \ H^{1}(A_S;S)\cong\mathbb{R}^n.\]
Under these isomorphisms, the map $\Phi:H^{1}(A_S)\to H^{1}(A_S;S)$ becomes $([\eta],v)\mapsto v$.
\end{lemma}

So, under the assumptions from Corollary \ref{corollary_trivial_foliations}, the space of leaves symplectomorphic to
$(S,\omega_S)$ is parameterized by
\begin{equation}\label{EQ_14}
\textrm{ker}\big([\delta_S\omega]:\mathbb{R}^n\to H^2(S)\big)
\end{equation}
Of course, using the local model, this can be checked directly. By Lemma \ref{L_trivial_foliation}, this space is at most
one-dimensional. An example where the space (\ref{EQ_14}) is indeed one-dimensional, can be constructed as follows:
consider the 2-sphere $S:=\mathbb{S}^2$, endowed with a symplectic structure $\omega_S$. Then the Poisson
structure on $S\times\mathbb{R}^2$ with symplectic foliation given by
\begin{equation}\label{EQ_13}
\big(S\times\{(y_1,y_2)\}, e^{y_1}\omega_S\big), \ \ (y_1,y_2)\in\mathbb{R}^2,
\end{equation}
satisfies the conditions of Lemma \ref{L_trivial_foliation}. Note that every leaf $S\times\{(y_1,y_2)\}$ is part of a
one-parametric family of symplectomorphic leaves: $S\times\{(y_1,y_2+t)\}$, $t\in\mathbb{R}$.

We remark that the Poisson structure in this example is isomorphic to the regular part of the linear Poisson structure
corresponding the Lie algebra $\mathfrak{g}=\mathfrak{su}(2)\oplus \mathbb{R}$. In fact, for a semisimple Lie
algebra of compact type $\mathfrak{k}$, the linear Poisson structure $\pi_{\mathfrak{g}}$ of the product
$\mathfrak{g}:=\mathfrak{k}\oplus \mathbb{R}$ is rigid (c.f.\ Corollary \ref{Corollary_Conn_gen}), and the Poisson
structure has a 1-parameter family of isomorphisms that don't preserve leaves: the translation by elements in
$\mathfrak{k}^{\circ}$. Thus, any leaf has a line of symplectomorphic leaves nearby.

In the case of simple symplectic foliations, we also have an improvement compared to the result of \cite{CrMa}; the
hypothesis in there can be restated as follows (c.f.\ Corollary 4.1.22 in \cite{teza})
\begin{itemize}
\item $S$ is compact with finite fundamental group,
\item the map $p^*\circ[\delta_{S}\omega]:\mathbb{R}^n\to H^2(\widetilde{S})$ is an isomorphism,
\end{itemize}
where $p:\widetilde{S}\to S$ is the universal cover of $S$. So, for example when $S$ is simply connected, the
difference between the assumptions is that, in our case, the map $[\delta_{S}\omega]$ might still have a 1-dimensional
kernel, whereas in \cite{CrMa} it has to be injective. In particular, the example (\ref{EQ_13}) above falls out of the
framework of \emph{loc.cit.}

\section{Proof of Theorem \ref{Main_Theorem_intro}}\label{section_technical_rigidity}

We start by preparing the setting needed for applying the Nash-Moser method: we fix norms on the Fr\'echet spaces
involved, we construct smoothing operators adapted to the problem and we recall the interpolation inequalities. Next,
we prove a series of inequalities which assert tameness of some natural operations such as: the Lie derivative, the flow
of a vector field, and the pullback; and then we prove some inequalities for the composition of local diffeomorphisms.
We end the section with the proof of Theorem \ref{Main_Theorem_intro}, which is mostly inspired by Conn's proof
\cite{Conn}.

\begin{remark}\rm
A usual convention when dealing with the Nash-Moser techniques (see e.g.\ \cite{Ham}), which we also adopt, is to
denote all constants by the same symbol. In the series of preliminary results below we work with ``big enough''
constants $C$ and $C_n$, and with ``small enough'' constants $\theta>0$; these depend on the trivialization data for
the vector bundle $E$ and on the smoothing operators. In the proof of Proposition \ref{Proposition_technical}, $C_n$
depends also on the Poisson structure $\pi$.
\end{remark}

\subsection{The ingredients of the tame category}\label{subsection_norms}
We borrow the terminology from \cite{Ham}. A Fr\'echet space $F$ endowed with an increasing family of semi-norms
$\{\|\cdot\|_{n}\}_{n\geq 0}$ generating its topology is called a \textbf{graded Fr\'echet space}. A linear map
$T:F_1\to F_2$ between two graded Fr\'echet spaces is called \textbf{tame} of degree $d$ and base $b$, if it satisfies
inequalities of the form
\[\|Tf\|_{n}\leq C_n\|f\|_{n+d},\ \forall \ n\geq b, f\in F_1.\]

Let $E\to N$ be a vector bundle over a compact manifold $N$ and fix a metric on $E$. For $r>0$, consider the closed
tube in $E$ of radius $r$
\[E_r:=\{v\in E: |v|\leq r\}.\]
The space of multivector fields on $E_r$, denoted by $\mathfrak{X}^{\bullet}(E_r)$, when endowed with
$C^n$-norms becomes a graded Fr\'echet space. We recall here the construction of such norms. Fix a finite open cover
of $N$ by domains of charts $\{\chi_i:O_i\diffto \mathbb{R}^d\}_{i=1}^{I}$ and vector bundle isomorphisms
\[\widetilde{\chi}_i:E|_{O_i}\diffto \mathbb{R}^d\times \mathbb{R}^D\]
covering $\chi_i$. We will assume that $\widetilde{\chi}_i(E_{r}|_{O_i})=\mathbb{R}^d\times \overline{B}_r$ and that
the family
\[\{O_i^{\delta}:=\chi_i^{-1}(B_{\delta})\}_{i=1}^I\]
covers $N$ for all $\delta\geq 1$. Moreover, we assume that the cover satisfies
\begin{equation}\label{EQ_6}
\textrm{if}\ \ O_i^{3/2}\cap O_j^{3/2}\neq \emptyset\  \ \textrm{then}\ \ O_j^1\subset O_i^4.
\end{equation}
This holds if $\chi_i^{-1}|_{B_4}:B_4\to O_i$ is the exponential corresponding to some metric on $N$, with injectivity
radius lager than $4$.

For $W\in\mathfrak{X}^{\bullet}(E_r)$, denote its local expression in the chart $\widetilde{\chi}_i$ by
\[W_i(z):=\sum_{1\leq i_1<\ldots<i_p\leq d+D}W_{i}^{i_1,\ldots,i_p}(z)\frac{\partial}{\partial z_{i_1}}\wedge \ldots\wedge \frac{\partial}{\partial z_{i_p}},\]
and let the $C^n$-norm of $W$ be given by
\[\|W\|_{n,r}:=\sup_{i, i_1,\ldots, i_p}\left\{\left|\frac{\partial^{|\alpha|}}{\partial{z}^{\alpha}} W_{i}^{i_1,\ldots,i_p}(z)\right|:z\in B_{1}\times B_r, 0\leq |\alpha|\leq n \right\}.\]
For $s<r$, the restriction maps are norm decreasing
\[\mathfrak{X}^{\bullet}(E_{r})\ni W\mapsto W|_{{s}}:=W|_{E_{s}}\in \mathfrak{X}^{\bullet}(E_{s}), \ \ \  \|W|_{{s}}\|_{n,s}\leq \|W\|_{n,r}.\]

We will work also with the closed subspaces of multivector fields on $E_{r}$ whose first jet vanishes along $N$, which
we denote by
\[\mathfrak{X}^{k}(E_r)^{(1)}:=\{W\in \mathfrak{X}^{k}(E_r) : j^1W|_{N}=0\}.\]

The main technical tool used in the Nash-Moser method are the smoothing operators. We will call a family $\{S_t:F\to
F\}_{t>1}$ of linear operators on the graded Fr\'echet space $F$ \textbf{smoothing operators} of degree $d\geq 0$, if
there exist constants $C_{n,m}>0$, such that for all $n,m\geq 0$ and $f\in F$, the following inequalities hold:
\begin{equation}\label{EQ_7}
\|S_t(f)\|_{n+m}\leq t^{m+d}C_{n,m}\|f\|_{n},\ \ \|S_t(f)-f\|_{n}\leq t^{-m}C_{n,m}\|f\|_{n+m+d}.
\end{equation}

The construction of such operators is standard, but since we are dealing with a Fr\'echet space for each $r\in (0,1]$, we
give the explicit dependence of the constants $C_{n,m}$ from (\ref{EQ_7}) on the parameter $r$.
\begin{lemma}\label{Lemma_smoothing_operators}
The family of graded Fr\'echet spaces $\{(\mathfrak{X}^{k}(E_r),\|\cdot\|_{n,r})\}_{r\in (0,1]}$ has a family of
smoothing operators of degree $d=0$
\[\{S_t^r:\mathfrak{X}^{k}(E_r)\rmap \mathfrak{X}^{k}(E_r) \}_{t>1,0<r\leq 1},\]
which satisfy (\ref{EQ_7}) with constants of the form $C_{n,m}(r)=C_{n,m}r^{-(n+m+k)}$.

Similarly, the family $\{(\mathfrak{X}^{k}(E_r)^{(1)},\|\cdot\|_{n,r})\}_{r\in (0,1]}$ has smoothing operators
\[\{S_t^{r,1}:\mathfrak{X}^{k}(E_r)^{(1)}\rmap \mathfrak{X}^{k}(E_r)^{(1)} \}_{t>1,0<r\leq 1},\]
of degree $d=1$ and constants $C_{n,m}(r)=C_{n,m}r^{-(n+m+k+1)}$.
\end{lemma}
\begin{proof}
The existence of smoothing operators of degree zero on the Fr\'echet space of sections of a vector bundle over a
compact manifold (possibly with boundary) is standard (see \cite{Ham}). We fix such a family
$\{S_t:\mathfrak{X}^{k}(E_{1})\to \mathfrak{X}^{k}(E_{1})\}_{t>1}$. Denote by
\[ \mu_{\rho}: E_{R}\rmap E_{\rho R}, \ \ \mu_{\rho}(v):=\rho v,\]
the rescaling operators. For $r\in (0,1]$, define $S_t^r$ by conjugating $S_t$ with $\mu_r^*$:
\[S_t^r:\mathfrak{X}^{k}(E_r)\rmap \mathfrak{X}^{k}(E_r),\ \ S_t^r:=\mu_{r^{-1}}^*\circ S_t\circ \mu_{r}^*.\]
Using the straightforward inequality
\[\|\mu_{\rho}^*(W)\|_{n,R}\leq \textrm{max}\{\rho^{-k},\rho^{n}\}\|W\|_{n,\rho R}, \ \forall\ W\in\mathfrak{X}^{k}(E_{\rho R}),\]
we obtain that $S_t^r$ satisfies (\ref{EQ_7}) with $C_{n,m}(r)=C_{n,m}r^{-(n+m+k)}$.

To construct the operators $S_t^{r,1}$, we first define a tame projection $P:\mathfrak{X}^{k}(E_r)\to
\mathfrak{X}^{k}(E_r)^{(1)}$. Choose $\{\lambda_i\}_{i=1}^I$ a smooth partition of unit on $N$ subordinated to
the cover $\{O_i^{1}\}_{i=1}^I$, and let $\{\widetilde{\lambda}_i\}_{i=1}^I$ be the pullback to $E$. For
$W\in\mathfrak{X}^{k}(E_r)$, denote its local representatives by
$W_i:=\widetilde{\chi}_{i,*}(W|_{E_{r}|_{O_i}})\in \mathfrak{X}^{k}(\mathbb{R}^d\times \overline{B}_r)$. Define
$P$ as follows:
\[P(W):=\sum_{i=1}^I\widetilde{\lambda}_i\cdot \widetilde{\chi}_{i,*}^{-1}(W_i-T_y^1(W_i)),\]
where $T^1_y(W_i)$ is the degree one Taylor polynomial of $W_i$ in the fiber direction
\[T^1_y(W_i)(x,y):=W_i(x,0)+\sum_{j=1}^D y_j\frac{\partial W_i}{\partial y_j}(x,0).\]
If $W\in \mathfrak{X}^{k}(E_r)^{(1)}$, then $T^1_y(W_i)=0$; so $P$ is a projection. It is easy to check that $P$ is
tame of
degree $1$, that is, there are constants $C_{n}>0$ such that %the following inequalities hold:
\[\|P(W)\|_{n,r}\leq C_n\|W\|_{n+1,r}.\]
Define the smoothing operators on $\mathfrak{X}^{k}(E_r)^{(1)}$ as follows:
\[S_t^{r,1}:\mathfrak{X}^{k}(E_r)^{(1)}\rmap \mathfrak{X}^{k}(E_r)^{(1)},\ \ S_t^{r,1}:=P\circ S_t^{r}.\]
Using tameness of $P$, the inequalities for $S_t^{r,1}$ are straightforward.
\end{proof}

The norms $\|\cdot\|_{n,r}$ satisfy the classical interpolation inequalities with constants which are polynomials in
$r^{-1}$.
\begin{lemma}\label{L_interpolation}
The norms $\|\cdot\|_{n,r}$ satisfy:
\[\|W\|_{l,r}\leq C_{n}r^{k-l} (\|W\|_{k,r})^{\frac{n-l}{n-k}}(\|W\|_{n,r})^{\frac{l-k}{n-k}},\  \forall \ r\in(0,1]\]
for all $0\leq k\leq l\leq n$, not all equal and all $W\in\mathfrak{X}^{\bullet}(E_r)$.
\end{lemma}
\begin{proof}
By the interpolation inequalities from \cite{Conn}, it follows that these inequalities hold for the $C^n$-norms on the
spaces $C^{\infty}(\overline{B}_1\times \overline{B}_r)$. Applying these to the components of the restrictions to the
charts $(E_{r}|_{O^{1}_i},\widetilde{\chi}_i)$ of a multivector field in $\mathfrak{X}^{\bullet}(E_r)$, we obtain the
interpolation inequalities on $\mathfrak{X}^{\bullet}(E_r)$.
\end{proof}

\subsection{Tameness of some natural operators}
In this subsection we prove some tameness properties of the Lie bracket, the pullback and the flow of vector fields.
\subsubsection*{The tame Fr\'echet Lie algebra of multivector fields} We prove that \[(\mathfrak{X}^{\bullet}(E_r),[\cdot,\cdot],\{\|\cdot\|_{n,r}\}_{n\geq 0})\]
is a tame Fr\'echet graded Lie algebra.
\begin{lemma}\label{L_Bracket}
The Schouten bracket on $\mathfrak{X}^{\bullet}(E_r)$ satisfies
\[\|[W,V]\|_{n,r}\leq C_nr^{-(n+1)}(\|W\|_{0,r}\|V\|_{n+1,r}+\|W\|_{n+1,r}\|V\|_{0,r}),\  \forall \ r\in(0,1].\]
\end{lemma}
\begin{proof}
By a local computation, the bracket satisfies inequalities of the form:
\[\|[W,V]\|_{n,r}\leq C_n\sum_{i+j=n+1}\|W\|_{i,r}\|V\|_{j,r}.\]
Using the interpolation inequalities, a term in this sum can be bounded by:
\[\|W\|_{i,r}\|V\|_{j,r}\leq C_nr^{-(n+1)}(\|W\|_{0,r}\|V\|_{n+1,r})^{\frac{j}{n+1}}(\|V\|_{0,r}\|W\|_{n+1,r})^{\frac{i}{n+1}}.\]
The following inequality, which will be used again later, implies the conclusion
\begin{equation}\label{EQ_simple}
x^{\lambda}y^{1-\lambda}\leq x+y, \ \  \forall\  x,y\geq 0, \lambda\in[0,1].\qedhere
\end{equation}
\end{proof}

\subsubsection*{The space of local diffeomorphisms}
We consider now the space of smooth maps $E_r\to E$ which are $C^1$-close to the inclusion $I_r:E_r\hookrightarrow
E$. We call a map $\varphi:E_r\to E$ \textbf{a local diffeomorphism}, if it can be extended on some open set to a
diffeomorphism onto its image. Since $E_r$ is compact, this is equivalent to injectivity of $d\varphi:TE_r\to TE$. To be
able to measure $C^n$-norms of such maps, we work with the following open neighborhood of $I_r$ in
$C^{\infty}(E_r;E)$:
\[\mathcal{U}_r:=\{\varphi:E_r\to E: \varphi(E_{r}|_{\overline{O}_i^{1}})\subset E|_{O_i}, 1\leq i\leq I\}.\]
Denote the local representatives of a map $\varphi\in\mathcal{U}_r$ by
\[\varphi_i:\overline{B}_{1}\times \overline{B}_{r}\rmap \mathbb{R}^{d}\times\mathbb{R}^{D}.\]
Define $C^n$-distances between maps $\varphi,\psi\in \mathcal{U}_r$ as follows
\[\mathrm{d}(\varphi,\psi)_{n,r}:=\sup_{1\leq i\leq I} \left\{\left|\frac{\partial^{|\alpha|}}{\partial z^{\alpha}}(\varphi_i-\psi_i)(z)\right|: z\in B_{1}\times B_{r},0\leq |\alpha|\leq n\right\}.\]
To control compositions of maps, we will also need the following $C^n$-distances
\[\mathrm{d}(\varphi,\psi)_{n,r,\delta}:=\sup_{1\leq i\leq I} \left\{\left|\frac{\partial^{|\alpha|}}{\partial z^{\alpha}}(\varphi_i-\psi_i)(z)\right|: z\in B_{\delta}\times B_{r},0\leq |\alpha|\leq n\right\},\]
which are well-defined only on the open set
\[\mathcal{U}^{\delta}_r:=\left\{\chi\in\mathcal{U}_r: \chi(E_r|_{\overline{O}_i^{\delta}})\subset E|_{O_i}\right\},\]
Similarly, we define also on $\mathfrak{X}^{\bullet}(E_r)$ norms $\|\cdot\|_{n,r,\delta}$ (these measure the
$C^n$-norms in all our local charts on $B_{\delta}\times B_{r}$).

These norms and distances are equivalent.
\begin{lemma}\label{Lemma_equivalent_norms}
There exist $C_n>0$, such that $\forall$  $r\in(0,1]$ and $\delta\in [1,4]$
\[\mathrm{d}(\varphi,\psi)_{n,r}\leq \mathrm{d}(\varphi,\psi)_{n,r,\delta}\leq C_n\mathrm{d}(\varphi,\psi)_{n,r}, \ \forall \ \varphi,\psi\in\mathcal{U}_{r}^{\delta},\]
\[\|W\|_{n,r}\leq \|W\|_{n,r,\delta}\leq C_n\|W\|_{n,r},\ \forall \ W\in\mathfrak{X}^{\bullet}(E_r).\]
\end{lemma}

We also use the simplified notations: %To simplify some formulas, we also use the notations
\[\mathrm{d}(\psi)_{n,r}:=\mathrm{d}(\psi,I_r)_{n,r}, \ \ \mathrm{d}(\psi)_{n,r,\delta}:=\mathrm{d}(\psi,I_r)_{n,r,\delta}.\]

The lemma below is used to check that compositions are defined.
\begin{lemma}\label{Lemma_control_embedding}
There exists a constant $\theta>0$, such that for all $r\in (0,1]$, $\epsilon\in (0,1]$, $\delta\in [1,4]$ and all
$\varphi\in\mathcal{U}_{r}$, satisfying $\mathrm{d}(\varphi)_{0,r}<\epsilon\theta$,
\[\varphi(E_r|_{\overline{O}_i^\delta})\subset E_{r+\epsilon}|_{O_i^{\delta+\epsilon}}.\]
\end{lemma}

We now prove that $I_r$ has a $C^1$-neighborhood of local diffeomorphisms.
\begin{lemma}\label{Lemma_embedding}
There exists $\theta>0$, such that, for all $r\in(0,1]$, if $\psi\in \mathcal{U}_r$ satisfies
$\mathrm{d}(\psi)_{1,r}<\theta$, then $\psi$ is a local diffeomorphism.
\end{lemma}
\begin{proof}
By Lemma \ref{Lemma_control_embedding}, if we shrink $\theta$, we may assume that
\begin{equation}\label{EQ_D3}
\psi(E_{r}|_{\overline{O}_i^1})\subset E|_{O_i^{3/2}},\ \ \psi(E_{r}|_{\overline{O}_i^4})\subset E|_{O_i}.
\end{equation}
In a local chart, we write $\psi$ as follows
\[\psi_i:=\textrm{Id}+g_{i}: \overline{B}_4\times \overline{B}_{r}\rmap \mathbb{R}^d\times\mathbb{R}^D.\]
By Lemma \ref{Lemma_equivalent_norms}, if we shrink $\theta$, we may also assume that
\begin{equation}\label{EQ_D1}
|\frac{\partial g_i}{\partial z_j}(z)|<\frac{1}{2(d+D)},  \forall \ z\in \overline{B}_4\times \overline{B}_{r}.
\end{equation}
This ensures that $\textrm{Id}+(dg_i)_z$ is close enough to $\textrm{Id}$ so that it is invertible for all $z\in
\overline{B}_1\times \overline{B}_{r}$, thus, $(d\psi)_p$ is invertible for all $p\in E_r$.

We check now injectivity of $\psi$. Let $p^i\in E_{r}|_{O^1_i}$ and $p^j\in E_{r}|_{O^1_j}$ be such that
$\psi(p^i)=q=\psi(p^j)$. Then, by (\ref{EQ_D3}), $q\in E|_{O^{3/2}_i}\cap E|_{O^{3/2}_j}$, so, by the property
(\ref{EQ_6}), we know that $O^{1}_j\subset O^{4}_i$, hence $p^i,p^j\in E_{r}|_{O^4_i}$. Denoting by
$w^i:=\widetilde{\chi}_i(p^i)$ and $w^j:=\widetilde{\chi}_i(p^j)$ we have that $w^i,w^j\in \overline{B}_4\times
\overline{B}_{r}$. Since $w^i+g_i(w^i)=w^j+g_i(w^j)$, using (\ref{EQ_D1}), we obtain
\begin{align*}%\label{EQ_D2}
|w^i-&w^j|=|g_i(w^i)-g_i(w^j)|=\\
&=|\int_0^1\sum_{k=1}^{D+d}\frac{\partial g_i}{\partial z_k}(tw^i+(1-t)w^j)(w^i_k-w^j_k)dt|\leq \frac{1}{2}|w^i-w^j|.
\end{align*}
Thus $w^i=w^j$, and so $p^i=p^j$. This finishes the proof.
\end{proof}

The composition satisfies the following tame inequalities.
\begin{lemma}\label{Lemma_composition}
There are constants $C_n> 0$ such that for all $1\leq \delta\leq \sigma\leq 4$ and all $0<s\leq r\leq 1$, we have that if
$\varphi\in\mathcal{U}_{s}$ and $\psi\in\mathcal{U}_{r}$ satisfy
\[\varphi(E_{s}|_{\overline{O}_i^{\delta}})\subset E_{r}|_{O_i^{\sigma}}, \ \psi(E_{r}|_{\overline{O}_i^{\sigma}})\subset E|_{O_i}, \ \forall\  1\leq i\leq I,\]
and $\mathrm{d}(\varphi)_{1,s}<1$, then the following inequalities hold:
\begin{align*}
\mathrm{d}(\psi\circ \varphi&)_{n,s,\delta}\leq \mathrm{d}(\psi)_{n,r,\sigma}+\mathrm{d}(\varphi)_{n,s,\delta}+\\
&+C_ns^{-n}(\mathrm{d}(\psi)_{n,r,\sigma}\mathrm{d}(\varphi)_{1,s,\delta}+\mathrm{d}(\varphi)_{n,s,\delta}\mathrm{d}(\psi)_{1,r,\sigma}),\\
\mathrm{d}(\psi\circ \varphi&,\psi)_{n,s,\delta}\leq \mathrm{d}(\varphi)_{n,s,\delta}+\\
&+C_ns^{-n}(\mathrm{d}(\psi)_{n+1,r,\sigma}\mathrm{d}(\varphi)_{1,s,\delta}+\mathrm{d}(\varphi)_{n,s,\delta}\mathrm{d}(\psi)_{1,r,\sigma}).
\end{align*}
\end{lemma}
\begin{proof}
Denote the local expressions of $\varphi$ and $\psi$ as follows:
\[\varphi_i:=\textrm{Id}+g_{i}:\overline{B}_{\delta}\times\overline{B}_{s}\rmap B_{\sigma}\times B_{r},\]
\[\psi_{i}:=\textrm{Id}+f_i:\overline{B}_{\sigma}\times \overline{B}_{r}\rmap \mathbb{R}^d\times \mathbb{R}^D.\]
Then for all $z\in \overline{B}_{\delta}\times\overline{B}_{s}$, we can write
\[\psi_i(\varphi_i(z))-z=f_i(z+g_i(z))+g_i(z).\]
By computing the $\frac{\partial^{|\alpha|}}{\partial z^{\alpha}}$ of the right hand side, for a multi-index $\alpha$
with $|\alpha|=n$, we obtain an expression of the form
\begin{align*}
\frac{\partial^{|\alpha|}g_i}{\partial z^{\alpha}}(z)+\frac{\partial^{|\alpha|}f_{i}}{\partial z^{\alpha}}(\varphi_{i}(z))+\sum_{\beta,\gamma_1,\ldots,\gamma_p}\frac{\partial^{|\beta|}f_{i}}{\partial z^{\beta}}(\varphi_i(z))\frac{\partial^{|\gamma_1|}g^{j_1}_{i}}{\partial z^{\gamma_1}}(z)\ldots\frac{\partial^{|\gamma_p|}g^{j_p}_{i}}{\partial z^{\gamma_p}}(z),
\end{align*}
where the multi-indices in the sum satisfy
\begin{equation}\label{EQ_E2}
1\leq p\leq n,\ 1\leq|\beta|,|\gamma_j|\leq n, \ |\beta|+\sum_{j=1}^p(|\gamma_j|-1)=n.
\end{equation}
The first two terms can be bounded by $\mathrm{d}(\psi)_{n,r,\sigma}+\mathrm{d}(\varphi)_{n,s,\delta}$. For the
last term we use the interpolation inequalities to obtain
\[\|f_i\|_{|\beta|,r,\sigma}\leq C_ns^{1-|\beta|}\|f_i\|_{1,r,\sigma}^{\frac{n-|\beta|}{n-1}}\|f_i\|_{n,r,\sigma}^{\frac{|\beta|-1}{n-1}},\]
\[\|g_i\|_{|\gamma_i|,s,\delta}\leq C_ns^{1-|\gamma_i|}\|g_i\|_{1,s,\delta}^{\frac{n-|\gamma_i|}{n-1}}\|g_i\|_{n,s,\delta}^{\frac{|\gamma_i|-1}{n-1}}.\]
Multiplying all these, and using (\ref{EQ_E2}), the sum is bounded by
\[ C_ns^{1-n}\|g_i\|_{1,s,\delta}^{p-1}(\|f_i\|_{1,r,\sigma}\|g_i\|_{n,s,\delta})^{\frac{n-|\beta|}{n-1}}(\|f_i\|_{n,r,\sigma}\|g_i\|_{1,s,\delta})^{\frac{|\beta|-1}{n-1}}.\]
By Lemma \ref{Lemma_equivalent_norms}, it follows that $\|g_i\|_{1,s,\delta}<C$, and dropping this term, the first
part follows using inequality (\ref{EQ_simple}).

For the second part, write for $z\in \overline{B}_{\delta}\times\overline{B}_{s}$:
\[\psi_i(\varphi_i(z))-\psi_i(z)=f_i(z+g_i(z))-f_i(z)+g_i(z).\]
We compute $\frac{\partial^{|\alpha|}}{\partial z^{\alpha}}$ of the right hand side, for $\alpha$ a multi-index with
$|\alpha|=n$:
\begin{align*}
\frac{\partial^{|\alpha|}f_{i}}{\partial z^{\alpha}}&(\varphi_{i}(z))-\frac{\partial^{|\alpha|}f_{i}}{\partial z^{\alpha}}(z)+\frac{\partial^{|\alpha|}g_{i}}{\partial z^{\alpha}}(z)+\\
&+ \sum_{\beta,\gamma_1,\ldots,\gamma_p}\frac{\partial^{|\beta|}f_{i}}{\partial z^{\beta}}(\varphi_{i}(z))\frac{\partial^{|\gamma_1|}g^{j_1}_{i}}{\partial z^{\gamma_1}}(z)\ldots\frac{\partial^{|\gamma_p|}g^{j_p}_{i}}{\partial z^{\gamma_p}}(z).
\end{align*}
where the multi-indices in the sum satisfy (\ref{EQ_E2}). The last term we bound as before, and the third by
$\mathrm{d}(\varphi)_{n,s,\delta}$. Writing the first two terms as
\begin{align*}
\sum_{j=1}^{d+D}\int_{0}^{1}\frac{\partial^{|\alpha|+1}f_{i}}{\partial z_j\partial z^{\alpha}}(z+tg_{i}(z))g^j_{i}(z)dt,
\end{align*}
they are less than $C\mathrm{d}(\psi)_{n+1,r,\sigma}\mathrm{d}(\varphi)_{0,s,\delta}$. Adding up, the result
follows.
\end{proof}

We give now conditions for infinite compositions of maps to converge.
\begin{lemma}\label{Lemma_convergent_embbedings}
There exists $\theta>0$, such that for all sequences
\[\{\varphi_{k}\in \mathcal{U}_{r_k}\}_{k\geq 1},\ \ \varphi_k:E_{r_k}\rmap E_{r_{k-1}},\]
where $0<r<r_{k}<r_{k-1}\leq r_0<1$, which satisfy
\[\sigma_0:=\sum_{k\geq 1}\mathrm{d}(\varphi_k)_{0,r_k}<\theta,\ \ \ \sigma_n:=\sum_{k\geq 1}\mathrm{d}(\varphi_k)_{n,r_k}<\infty,\  \forall \ n\geq 1,\]
the sequence of maps
\[\psi_k:=\varphi_1\circ\ldots\circ \varphi_k:E_{r_k}\rmap E_{r_0},\]
converges in all $C^n$-norms on $E_r$ to a map $\psi:E_r\to E_{r_0}$, with $\psi\in\mathcal{U}_r$. Moreover, there
are $C_n>0$, such that if $\mathrm{d}(\varphi_k)_{1,r_k}<1$, $\forall$ $k\geq 1$, then
\[\mathrm{d}(\psi)_{n,r}\leq e^{C_nr^{-n}\sigma_n}C_nr^{-n}\sigma_n.\]
\end{lemma}
\begin{proof}
Consider the following sequences of numbers:
\[\epsilon_k:=\frac{\mathrm{d}(\varphi_k)_{0,r_k}}{\sum_{l\geq 1}\mathrm{d}(\varphi_l)_{0,r_l}}, \ \ \delta_k:=2-\sum_{l=1}^k\epsilon_l.\]
We have that $\mathrm{d}(\varphi_k)_{0,r_k}\leq \epsilon_k\theta$. So, by Lemma \ref{Lemma_control_embedding},
we may assume that
\[\varphi_k(E_{r_k}|_{\overline{O}_i^2})\subset E_{r_{k-1}}|_{O_i},\ \ \varphi_{k}(E_{r_k}|_{\overline{O}_i^{\delta_k}})\subset E_{r_{k-1}}|_{O^{\delta_{k-1}}_i},\]
and this implies that
\[\psi_{k-1}(E_{r_{k-1}}|_{\overline{O}_i^{\delta_{k-1}}})\subset E_{r_0}|_{O_i}.\]
So we can apply Lemma \ref{Lemma_composition} to the pair $\psi_{k-1}$ and $\varphi_k$ for all $k> k_0$. The first
part of Lemma \ref{Lemma_composition} and Lemma \ref{Lemma_equivalent_norms} imply an inequality of the form:
\begin{align*}
1+\mathrm{d}(\psi_{k})_{n,r_k,\delta_k}\leq (1+\mathrm{d}(\psi_{k-1})_{n,r_{k-1},\delta_{k-1}})(1+C_nr^{-n}\mathrm{d}(\varphi_{k})_{n,r_{k}}).
\end{align*}
Iterating this inequality, we obtain that
\begin{align*}
1+\mathrm{d}(\psi_{k})_{n,r_k,\delta_k}&\leq (1+\mathrm{d}(\psi_{k_0})_{n,r_{k_0},\delta_{k_0}})\prod_{l=k_0+1}^k(1+C_nr^{-n}\mathrm{d}(\varphi_{l})_{n,r_{l}})\leq\\
&\leq (1+\mathrm{d}(\psi_{k_0})_{n,r_{k_0},\delta_{k_0}})e^{C_nr^{-n}\sum_{l> k_0}\mathrm{d}(\varphi_{l})_{n,r_{l}}}\leq\\
&\leq (1+\mathrm{d}(\psi_{k_0})_{n,r_{k_0},\delta_{k_0}})e^{C_nr^{-n}\sigma_n}.
\end{align*}
The second part of Lemma \ref{Lemma_composition} and Lemma \ref{Lemma_equivalent_norms} imply
\begin{align*}
 \mathrm{d}(\psi_{k},\psi_{k-1})_{n,r}&\leq (1+\mathrm{d}(\psi_{k-1})_{n+1,r_{k-1},\delta_{k-1}})C_nr^{-n}\mathrm{d}(\varphi_{k})_{n,r_{k},\delta_{k}}\leq\\
\nonumber &\leq (1+\mathrm{d}(\psi_{k_0})_{n+1,r_{k_0},\delta_{k_0}})e^{C_{n+1}r^{-1-n}\sigma_{n+1}}C_nr^{-n}\mathrm{d}(\varphi_{k})_{n,r_{k}}.
\end{align*}
This shows that the sum $\sum_{k\geq 1}\mathrm{d}(\psi_{k},\psi_{k-1})_{n,r}$ converges for all $n$, hence the
sequence $\{\psi_{k}|_{E_r}\}_{k\geq 1}$ converges in all $C^n$-norms to a smooth function $\psi:E_{r}\to
E_{r_0}$.

If $\mathrm{d}(\varphi_k)_{1,r_k}<1$ for all $k\geq 1$, then we can take $k_0=0$. So, we obtain
\begin{align*}
1+\mathrm{d}(\psi_{k})_{n,r_k,\delta_k}\leq\prod_{l=1}^k(1+C_nr^{-n}\mathrm{d}(\varphi_{l})_{n,r_{l}})\leq e^{C_nr^{-n}\sum_{1}^k\mathrm{d}(\varphi_{l})_{n,r_{l}}}\leq e^{C_nr^{-n}\sigma_n}.
\end{align*}
Using the trivial inequality $e^x-1\leq xe^x$, for $x\geq 0$, the result follows.
\end{proof}

\subsubsection*{Tameness of the flow} The $C^0$-norm of a vector field controls the size of the domain of its flow.

\begin{lemma}\label{Lemma_domain_of_flow}
There exists $\theta>0$ such that for all $0<s<r\leq 1$ and all $X\in \mathfrak{X}^1(E_r)$ with
$\|X\|_{0,r}<(r-s)\theta$, we have that $\varphi_X^t$, the flow of $X$, is defined for all $t\in [0,1]$ on $E_{s}$ and
belongs to $\mathcal{U}_{s}$.
%\[\varphi^t_{X}: E_{r-\epsilon}\rmap E_{r}.\]
\end{lemma}
\begin{proof}
We denote the restriction of $X$ to a chart by $X_i\in \mathfrak{X}^{1}(\mathbb{R}^d \times \overline{B}_r)$.
Consider $p\in\overline{B}_1\times \overline{B}_{s}$. Let $t\in(0,1]$ be such that the flow of $X_i$ is defined up to
time $t$ at $p$ and such that for all $\tau\in [0,t)$ it satisfies $\varphi^{\tau}_{X_{i}}(p)\in B_2\times B_r$. Then we
have that
\begin{align*}
|\varphi^t_{X_{i}}(p)-p|=|\int_0^td \left(\varphi^{\tau}_{X_{i}}(p)\right)|\leq\int_0^t|X_i(\varphi^{\tau}_{X_{i}}(p))|d\tau\leq \|X_i\|_{0,r,2}\leq C\|X\|_{0,r},
\end{align*}
where for the last step we used Lemma \ref{Lemma_equivalent_norms}. Hence, if $\|X\|_{0,r}<(r-s)/C$, we have that
$\varphi^t_{X_{i}}(p)\in B_2\times B_r$, and this implies the result.
\end{proof}

We prove now that the map which associates to a vector field its flow is tame (this proof was inspired by the proof of
Lemma B.3 in \cite{Miranda}).
\begin{lemma}\label{Lemma_size of _the flow}
There exists $\theta>0$ such that for all $0<s<r\leq 1$, and all $X\in \mathfrak{X}^1(E_r)$ with
\[\|X\|_{0,r}<(r-s)\theta,\ \ \|X\|_{1,r}<\theta\]
we have that $\varphi_X:=\varphi_X^1$ belongs to $\mathcal{U}_{s}$ and it satisfies:
\[\mathrm{d}(\varphi_X)_{0,s}\leq C_0\|X\|_{0,r},\ \ \mathrm{d}(\varphi_{X})_{n,s}\leq r^{1-n}C_n\|X\|_{n,r},\ \forall\ n\geq 1.\]
%where $C_n>0$ are constants depending only on $n$.
\end{lemma}

\begin{proof}
By Lemma \ref{Lemma_domain_of_flow}, for $t\in[0,1]$, we have that $\varphi^t_X\in\mathcal{U}_{s}$, and by its
proof that the local representatives take values in $B_2\times B_r$
\[\varphi^t_{X_i}:=\textrm{Id}+g_{i,t}: \overline{B}_1\times \overline{B}_{s}\rmap B_2\times B_r.\]
We will prove by induction on $n$ that $g_{i,t}$ satisfies inequalities of the form:
\begin{equation}\label{EQ_poly}
\|g_{i,t}\|_{n,s}\leq C_nP_{n}(X),
\end{equation}
where $P_n(X)$ denotes the following polynomials in the norms of $X$
\[P_0(X)=\|X\|_{0,r}, \ P_1(X)=\|X\|_{1,r},\]
\[ P_n(X)=\sum_{\stackrel{j_1+\ldots+j_p=n-1}{1\leq j_k\leq n-1}} \| X \|_{j_1+1,r}\ldots \|X\|_{j_p+1,r}.\]
Observe that (\ref{EQ_poly}) implies the conclusion, since by the interpolation inequalities and the fact that
$\|X\|_{1,r}<\theta \leq 1$ we have that
\[\|X\|_{j_k+1,r}\leq C_nr^{-j_k}(\|X\|_{1,r})^{1-\frac{j_k}{n-1}}(\|X\|_{n,r})^{\frac{j_k}{n-1}}\leq C_n r^{-j_k}\|X\|_{n,r}^{\frac{j_k}{n-1}},\]
hence
\[P_n(X) \leq C_nr^{1-n}\|X\|_{n,r}.\]

The map $g_{i,t}$ satisfies the ordinary differential equation
\[\frac{d g_{i,t}}{dt}(z)=\frac{d \varphi^t_{X_i}}{dt}(z)=X_i(\varphi_{X_i}^t(z))=X_i(g_{i,t}(z)+z).\]
Since $g_{i,0}=0$, it follows that
\begin{equation}\label{EQ_for_the_difference}
g_{i,t}(z)=\int_0^tX_i(z+g_{i,\tau}(z))d\tau.
\end{equation}
Using also Lemma \ref{Lemma_equivalent_norms}, we obtain the result for $n=0$:
\[\|g_{i,t}\|_{0,s} \leq \|X\|_{0,r,2}\leq C_0\|X\|_{0,r}.\]
We will use the following version of the Gronwall inequality: if $u:[0,1]\to \mathbb{R}$ is a continuous map and there
are positive constants $A$, $B$ such that
\[u(t)\leq A+B\int_{0}^tu(\tau)d\tau,\]
then $u$ satisfies $u(t)\leq Ae^B$.

Computing the partial derivative $\frac{\partial}{\partial z_j}$ of equation (\ref{EQ_for_the_difference}) we obtain
\begin{align*}
\frac{\partial g_{i,t}}{\partial z_j}(z)&=\int_0^t\left(\frac{\partial X_{i}}{\partial z_j}(z+g_{i,\tau}(z))+\sum_{k=1}^{D+d}\frac{\partial X_i}{\partial z_k}(z+g_{i,\tau}(z))\frac{\partial g^k_{i,\tau}}{\partial z_j}(z)\right) d\tau.
\end{align*}
Therefore, using again Lemma \ref{Lemma_equivalent_norms}, the function $|\frac{\partial g_{i,t}}{\partial z_j}(z)|$
satisfies:
\begin{align*}
|\frac{\partial g_{i,t}}{\partial z_j}(z)|\leq C\|X\|_{1,r}+(D+d)\|X\|_{1,r}\int_0^t|\frac{\partial g_{i,\tau}}{\partial z_j}(z)|d\tau.
\end{align*}
The case $n=1$ follows now by Gronwall's inequality:
\[\|\frac{\partial g_{i,t}}{\partial z_j}\|_{0,s}\leq C \|X\|_{1,r}e^{(D+d)\|X\|_{1,r}}\leq C\|X\|_{1,r}.\]

For a multi-index $\alpha$, with $|\alpha|=n\geq 2$, applying $\frac{\partial^{|\alpha|}}{\partial z^{\alpha}}$ to
(\ref{EQ_for_the_difference}), we obtain
\begin{align}\label{EQ_3}
\frac{\partial^{|\alpha|}g_{i,t}}{\partial z^{\alpha}}(z)&=\int_{0}^t\sum_{2\leq |\beta|\leq |\alpha|} \frac{\partial^{|\beta|}X_{i}}{\partial z^{\beta}}(z+g_{i,\tau}(z))\frac{\partial^{|\gamma_1|}g_{i,\tau}^{i_1}}{\partial z^{\gamma_1}}(z)\ldots\frac{\partial^{|\gamma_p|}g_{i,\tau}^{i_p}}{\partial z^{\gamma_p}}(z)d\tau+\\
\nonumber&+\int_0^t\sum_{j=1}^{D+d}\frac{\partial X_i}{\partial z_j}(z+g_{i,\tau}(z))\frac{\partial^{|\alpha|}g_{i,\tau}^{j}}{\partial z^{\alpha}}(z) d\tau,
\end{align}
where the multi-indices satisfy
\[1\leq |\gamma_k|\leq n-1,\ \  (|\gamma_1|-1)+\ldots + (|\gamma_p|-1)+|\beta|=n.\]
Since $|\gamma_k|\leq n-1$, we can apply induction to conclude that
\[\|\frac{\partial^{|\gamma_k|}g_{i,\tau}^{i_k}}{\partial z^{\gamma_k}}\|_{0,s}\leq P_{|\gamma_k|}(X).\]
So, the first part of the sum can be bounded by
\begin{align}\label{EQ_2}
C_n\sum_{\stackrel{j_0+\ldots+j_p=n-1}{1\leq j_k\leq n-1}} \| X \|_{j_0+1,1}P_{j_1+1}(X)\ldots P_{j_p+1}(X).
\end{align}
It is easy to see that the polynomials $P_k(X)$ satisfy:
\begin{equation}\label{EQ_mult_prop_P}
P_{u+1}(X)P_{v+1}(X)\leq C_{u,v}P_{u+v+1}(X),
\end{equation}
therefore (\ref{EQ_2}) can be bounded by $C_nP_n(X)$. Using this in (\ref{EQ_3}), we obtain
\[|\frac{\partial^{|\alpha|} g_{i,t}}{\partial z^{\alpha}}(z)|\leq C_nP_n(X)+(D+d)\|X\|_{1,r}\int_0^t|\frac{\partial^{|\alpha|} g_{i,\tau}}{\partial z^{\alpha}}(z)|d\tau.\]
Applying Gronwall's inequality, we obtain the conclusion.
\end{proof}

We show now how to approximate pullbacks by flows of vector fields.
\begin{lemma}\label{Lemma_tame_flow_pull_back}
There exists $\theta >0$, such that for all $0<s<r\leq 1$ and all $X\in \mathfrak{X}^1(E_r)$ with
$\|X\|_{0,r}<(r-s)\theta$ and $\|X\|_{1,r}<\theta$, we have that
\begin{align*}
\|\varphi_{X}^*(W)&\|_{n,s} \leq C_n r^{-n}(\|W\|_{n,r}+\|W\|_{0,r}\|X\|_{n+1,r}),\\
\|\varphi_{X}^*(W)&-W|_{s}\|_{n,s}\leq C_nr^{-2n-1}(\|X\|_{n+1,r}\|W\|_{1,r}+\|X\|_{1,r}\|W\|_{n+1,r}),\\
\|\varphi_{X}^*(W)&-W|_{s}- \varphi_{X}^*([X,W])\|_{n,s} \leq \\
&\leq C_nr^{-3(n+2)}\|X\|_{0,r}(\|X\|_{n+2,r}\|W\|_{2,r}+\|X\|_{2,r}\|W\|_{n+2,r}),
\end{align*}
for all $W\in\mathfrak{X}^{\bullet}(E_r)$, where $C_n>0$ is a constant depending only on $n$.
\end{lemma}
\begin{proof}
As in the proof above, the local expression of $\varphi_X$ is defined as follows:
\[\varphi_{X_i}=\textrm{Id}+g_{i}:\overline{B}_1\times \overline{B}_{s}\rmap B_2\times B_r.\]
Let $W\in\mathfrak{X}^{\bullet}(E_r)$, and denote by $W_i$ its local expression on $E_{r|\overline{O}_2^i}$:
\[W_i:=\sum_{J=\{j_1<\ldots<j_k\}} W_i^J(z)\frac{\partial}{\partial z_{j_1}}\wedge \ldots \wedge \frac{\partial}{\partial z_{j_k}}\in\mathfrak{X}^{\bullet}(\overline{B}_2\times\overline{B}_r).\]
The local representative of $\varphi_X^*(W)$, is given for $z\in \overline{B}_1\times \overline{B}_{s}$ by
\[(\varphi_X^*W)_i=\sum_{J} W_i^J(z+g_i(z))(\textrm{Id}+d_zg_i)^{-1}\frac{\partial}{\partial z_{j_1}}\wedge \ldots \wedge (\textrm{Id}+d_zg_i)^{-1}\frac{\partial}{\partial z_{j_k}}.\]
By the Cramer rule, the matrix $(\textrm{Id}+d_zg_i)^{-1}$ has entries of the form
\[\Psi\left(\frac{\partial g_i^{l}}{\partial z_{j}}(z)\right)det(\textrm{Id}+d_zg_i)^{-1},\]
where $\Psi$ is a polynomial in the variables $Y^l_j$, which we substitute by $\frac{\partial g_i^{l}}{\partial
z_{j}}(z)$. Therefore, any coefficient of the local expression of $\varphi_X^*(W)_i$, will be a sum of elements of the
form
\[W_i^{J}(z+g_i(z))\Psi\left(\frac{\partial g_i^{l}}{\partial z_{j}}(z)\right)det(\textrm{Id}+d_zg_i)^{-k}.\]
When computing $\frac{\partial^{|\alpha|}}{\partial z^{\alpha}}$ of such an expression, with $|\alpha|=n$, using an
inductive argument, one proves that the outcome is a sum of terms of the form
\begin{equation}\label{EQ_9}
\frac{\partial^{|\beta|}W_{i}^J}{\partial z^{\beta}}(z+g_i(z))\frac{\partial^{|\gamma_1|} g_i^{v_{1}}}{\partial z^{\gamma_1}}(z)\ldots \frac{\partial^{|\gamma_p|} g_i^{v_{p}}}{\partial z^{\gamma_p}}(z)det(\textrm{Id}+d_zg_i)^{-M},
\end{equation}
with coefficients depending only on $\alpha$ and on the multi-indices, which satisfy
\[0\leq p,\ 0\leq M,  \ 1\leq |\gamma_j|, \ |\beta|+(|\gamma_1|-1)+\ldots+(|\gamma_p|-1)=n.\]
By Lemma \ref{Lemma_size of _the flow}, $\|g_i\|_{1,s}<C\theta$, so, if we shrink $\theta$, we find that
\[det(\textrm{Id}+d_zg_i)^{-1}<2, \ \forall z\in\overline{B}_1\times \overline{B}_{s}.\]
Using this, Lemma \ref{Lemma_equivalent_norms} for $W$ and $|\frac{\partial g_i^{l}}{\partial z_{j}}(z)|\leq C$, we
bound (\ref{EQ_9}) by
\[C_n\sum_{j,j_1,\ldots, j_p}\|W\|_{j,r}\|g_i\|_{j_1+1,s}\ldots\|g_i\|_{j_p+1,s},\]
where the indexes satisfy
\[0\leq j,\ 0\leq j_k,\ j+j_1+\ldots +j_p=n.\]
The term with $p=0$ can be simply bounded by $C_n\|W\|_{n,r}$. For the other terms, we will use the bound
$\|g_i\|_{j_k+1,s}\leq P_{j_k+1}(X)$ from the proof of Lemma \ref{Lemma_size of _the flow}. The multiplicative
property (\ref{EQ_mult_prop_P}) of the polynomials $P_l(X)$ implies
\[\|\varphi_X^*(W)\|_{n,s}\leq C_n\sum_{j=0}^n\|W\|_{j,r}P_{n-j+1}(X).\]
Applying interpolation to $W_{j,r}$ and to a term of $P_{n-j+1}(X)$ we obtain
\begin{align*}
\|W\|_{j,r}&\leq C_nr^{-j}\|W\|_{0,r}^{1-j/n}\|W\|_{n,r}^{j/n},\\
\|X\|_{j_k+1,r}&\leq C_nr^{-j_k}\|X\|_{1,r}^{1-j_k/n}\|X\|_{n+1,r}^{j_k/n}\leq C_nr^{-j_k} \|X\|_{n+1,r}^{j_k/n}.
\end{align*}
Multiplying all these terms, and using (\ref{EQ_simple}), we conclude the first part of the proof:
\begin{align*}
\|W\|_{j,r}\|X\|_{j_1+1,r}\ldots\|X\|_{j_p+1,r}&\leq C_nr^{-n}(\|W\|_{0,r}\|X\|_{n+1,r})^{1-j/n}\|W\|_{n,r}^{j/n}\leq\\
&\leq C_nr^{-n}(\|W\|_{n,r}+\|W\|_{0,r}\|X\|_{n+1,r}).
\end{align*}

For the second inequality, denote by
\[W_t:=\varphi^{t*}_{X}(W)-W|_{s}\in\mathfrak{X}^{\bullet}(E_{s}).\]
Then $W_0=0$, $W_1=\varphi_{X}^*(W)-W|_{s}$ and $\frac{d}{dt}W_t=\varphi^{t*}_{X}([X,W])$, therefore
\[\varphi_{X}^*(W)-W|_{s}=\int_{0}^1\varphi^{t*}_{X}([X,W])dt.\]
By the first part, we obtain
\[\|\varphi_{X}^*(W)-W|_{s}\|_{n,s}\leq C_nr^{-n}(\|[X,W]\|_{n,r}+\|[X,W]\|_{0,r}\|X\|_{n+1,r}).\]
Using now Lemma \ref{L_Bracket} and that $\|X\|_{1,r}\leq \theta$ we obtain the second part:
\[\|\varphi_{X}^*(W)-W|_{s}\|_{n,s}\leq C_nr^{-2n-1}(\|X\|_{n+1,r}\|W\|_{1,r}+\|W\|_{1,r}\|X\|_{n+1,r}).\]

For the last inequality, denote by
\[W_t:=\varphi^{t*}_{X}(W)-W|_{s}- t\varphi_{X}^{t*}([X,W]).\]
Then we have that $W_0=0$ and $W_1=\varphi^{*}_{X}(W)-W|_{s}- \varphi_{X}^*([X,W])$ and
\[\frac{d}{dt}W_t=-t\varphi_{X}^{t*}([X,[X,W]]),\]
therefore
\[W_1=-\int_{0}^1t\varphi_{X}^{t*}([X,[X,W]])dt.\]
Using again the first part, it follows that
\begin{equation}\label{EQ_1}
\|W_1\|_{n,s}\leq C_nr^{-n}(\|[X,[X,W]]\|_{n,r}+\|[X,[X,W]]\|_{0,r}\|X\|_{n+1,r}).
\end{equation}
Applying twice Lemma \ref{L_Bracket}, for all $k\leq n$ we obtain that:
\begin{align*}
\|[X,[X,W]]&\|_{k,r}\leq C_{n}(r^{-(k+3)}\|X\|_{k+1,r}(\|X\|_{0,r}\|W\|_{1,r}+\|X\|_{1,r}\|W\|_{0,r})+\\
&+r^{-(2k+3)}\|X\|_{0,r}(\|X\|_{0,r}\|W\|_{k+2,r}+\|X\|_{k+2,r}\|W\|_{0,r}))\leq\\
&\leq C_{n}r^{-(2k+5)}\|X\|_{0,r}(\|W\|_{k+2,r}\|X\|_{0,r}+\|W\|_{2,r}\|X\|_{k+2,r}),
\end{align*}
where we have used the interpolation inequality
\[\|X\|_{1,r}\|X\|_{k+1,r}\leq C_nr^{-(k+2)}\|X\|_{0,r}\|X\|_{k+2,r}.\]
The first term in (\ref{EQ_1}) can be bounded using this inequality for $k=n$. For $k=0$, using also that
$\|X\|_{1,r}\leq \theta$ and the interpolation inequality
\[\|X\|_{2,r}\|X\|_{n+1,r}\leq C_nr^{-(n+1)}\|X\|_{1,r}\|X\|_{n+2,r},\]
we can bound the second term in (\ref{EQ_1}), and this concludes the proof:
\begin{align*}
\|[X,[X,W]]\|_{0,r}\|X\|_{n+1,r}&\leq C_nr^{-(n+6)}\|W\|_{2,r}\|X\|_{0,r}\|X\|_{n+2,r}.\qedhere
\end{align*}
\end{proof}

\subsection{An invariant tubular neighborhood and tame homotopy operators}

We start now the proof of Theorem \ref{Main_Theorem_intro}. We will use two results presented in the appendix:
existence of invariant tubular neighborhood (Lemma \ref{Lemma_tubular_neighborhood}) and the Tame Vanishing
Lemma (Lemma \ref{Tame_Vanishing_Lemma}).

Let $(M,\pi)$ and $N\subset M$ be as in the statement. Let $\mathcal{G}\rightrightarrows M$ be a Lie groupoid
integrating $T^*M$. By restricting to the connected components of the identities in the $s$-fibers of $\mathcal{G}$
\cite{MM}, we may assume that $\mathcal{G}$ has connected $s$-fibers.

By Lemma \ref{Lemma_tubular_neighborhood}, $N$ has an invariant tubular neighborhood $E\cong \nu_N$ endowed
with a metric, such that the closed tubes $E_r:=\{v\in E |  |v|\leq r\}$, for $r>0$, are also $\mathcal{G}$-invariant. We
endow $E$ with all the structure from subsection \ref{subsection_norms}.

Since $E$ is invariant, the cotangent Lie algebroid of $(E,\pi)$ is integrable by $\mathcal{G}|_{E}$, which has
compact $s$-fibers with vanishing $H^2$. Therefore, by the Tame Vanishing Lemma and Corollaries
\ref{corollary_unu}, \ref{corollary_unu_prim} from the appendix, there are linear homotopy operators
\[\mathfrak{X}^1(E)\stackrel{h_1}{\longleftarrow}\mathfrak{X}^2(E)\stackrel{h_2}{\longleftarrow}\mathfrak{X}^3(E),\]
\[[\pi,h_1(V)]+h_2([\pi,V])=V,\ \ \forall \  V\in \mathfrak{X}^2(E),\]
which satisfy:
\begin{itemize}
\item they induce linear homotopy operators $h_1^r$ and $h_2^r$ on $(E_r,\pi|_{r})$;
\item there are constants $C_{n}>0$ such that, for all $r\in(0,1]$,
\[\|h_1^{r}(X)\|_{n,r}\leq C_{n} \|X\|_{n+s,r},\ \ \|h_2^{r}(Y)\|_{n,r}\leq C_{n} \|Y\|_{n+s,r},\]
for all $X\in\mathfrak{X}^2(E_r)$, $Y\in\mathfrak{X}^3(E_r)$, where
$s=\lfloor\frac{1}{2}\mathrm{dim}(M)\rfloor+1$;
\item they induce homotopy operators on the subcomplex of vector fields vanishing along $N$.
\end{itemize}

\subsection{The Nash-Moser method}

We fix radii $0<r<R<1$. Let $s$ be as in the previous subsection, and let
\[\alpha:=2(s+5), \ \ \  \ \ p:=7(s+4).\]
Then $p$ is the integer from the statement of Theorem \ref{Main_Theorem_intro}. Consider $\widetilde{\pi}$ a second
Poisson structure defined on $E_R$. To $\widetilde{\pi}$ we associate the inductive procedure:

\noindent\textbf{Procedure P$_0$}: Consider
\begin{itemize}
\item the number
\[t(\widetilde{\pi}):=\|\pi-\widetilde{\pi}\|_{p,R}^{-1/\alpha},\]
\item the sequences of numbers
\[\begin{array}{ccc}
  \epsilon_0:=(R-r)/4, & r_0:=R, &  t_0:=t(\widetilde{\pi}), \\
  \epsilon_{k+1}:=\epsilon_k^{3/2}, & r_{k+1}:=r_k-\epsilon_k, & t_{k+1}:=t_k^{3/2},
\end{array}\]
\item the sequences of Poisson bivectors and vector fields
\[\{\pi_k\in\mathfrak{X}^{2}(E_{r_k})\}_{k\geq 0},\ \ \ \{X_k\in\mathfrak{X}^1(E_{r_k})\}_{k\geq 0}, \]
defined inductively by
\begin{equation}\label{EQ_procedure}
\pi_0:=\widetilde{\pi}, \ \ \ \pi_{k+1}:=\varphi_{X_k}^*(\pi_k),\ \ \ X_k:=S_{t_k}^{r_k}(h_1^{r_k}(\pi_k-\pi|_{{r_k}})),
\end{equation}
\item the sequence of maps
\[\psi_k:=\varphi_{X_0}\circ\ldots\circ \varphi_{X_{k}}:E_{r_{k+1}}\rmap E_R.\]
\end{itemize}
By our choice of $\epsilon_0$, observe that $r<r_k< R$ for all $k\geq 1$:
\begin{align*}
\sum_{k=0}^{\infty} \epsilon_k=\sum_{k=0}^{\infty} \epsilon_0^{(3/2)^k}<\sum_{k=0}^{\infty} \epsilon_0^{1+\frac{k}{2}}=\frac{\epsilon_0}{1-\sqrt{\epsilon_0}}\leq (R-r),
\end{align*}

For \textbf{Procedure P$_0$} to be well-defined, we need that
\begin{itemize}
\item [$(C_k)$] the time-one flow of $X_k$ is defined as a map between \[\varphi_{X_k}:E_{r_{k+1}}\rmap E_{r_{k}}.\]
\end{itemize}

For part (b) of Theorem \ref{Main_Theorem_intro}, we consider also the \textbf{Procedure P$_1$}, associated to
$\widetilde{\pi}$ such that $j^1\widetilde{\pi}|_{N}=j^1\pi|_{N}$. We define \textbf{Procedure P$_1$} the same as
\textbf{Procedure P$_0$}, except that in (\ref{EQ_procedure}) we use the smoothing operators $S_{t_k}^{r_k,1}$.

To show that \textbf{Procedure P$_1$} is well-defined, in addition to $(C_k)$, we need that
$h_1^{r_k}(\pi_k-\pi|_{{r_k}})\in \mathfrak{X}^1(E_{r_k})^{(1)}$. Since the operators $h_1^{r_k}$ preserve the
space of tensors vanishing up to first order, it suffice to show that $j^1(\pi_k-\pi|_{r_k})|_N=0$. This is proven
inductively: By hypothesis, $j^1(\pi_0-\pi|_{R})|_{N}=0$. Assume that $j^1(\pi_k-\pi|_{{r_k}})|_{N}=0$, for some
$k\geq 0$. Then, as before, also $X_k\in \mathfrak{X}^1(E_{r_k})^{(1)}$, hence the first order jet of $\varphi_{X_k}$
along $N$ is that of the identity, and so
\[j^1(\pi_{k+1})|_N=j^1(\pi_{k})|_N=j^1(\pi)|_N.\] Therefore $j^1(\pi_{k+1}-\pi|_{{r_{k+1}}})|_{N}=0$.

\textbf{Procedure P$_0$} produces the map $\psi$ from Theorem \ref{Main_Theorem_intro}.

\begin{proposition}\label{Proposition_technical}
There exists $\delta>0$ and an integer $d\geq 0$, for which procedure \textbf{P$_0$} is well defined for every Poisson
bivector $\widetilde{\pi}$ satisfying
\begin{equation}\label{EQ_8}
\|\widetilde{\pi}-\pi\|_{p,R}<\delta (r(R-r))^{d}.
\end{equation}
If in addition, $j^1\pi|_N=j^1\widetilde{\pi}|_N$, then \textbf{P$_1$} is also well defined for $\widetilde{\pi}$. In
both cases, the resulting sequence $\psi_{k}|_{r}$ converges uniformly on $E_r$ with all its derivatives to a local
diffeomorphism $\psi$, which is a Poisson map between
\[\psi:(E_r,\pi|_{r})\rmap (E_R,\widetilde{\pi}),\]
and it satisfies
\begin{equation}\label{EQ_continuity}
\mathrm{d}(\psi)_{1,r}\leq \|\pi-\widetilde{\pi}\|^{1/\alpha}_{p,R}.
\end{equation}
In the case of \textbf{P$_1$}, the map $\psi$ is the identity along $N$ up to first order.
\end{proposition}
\begin{proof}
We will prove the statement for the two procedures simultaneously. We denote by $S_k$ the used smoothing operators,
that is, in \textbf{P$_0$} we let $S_{k}:=S_{t_k}^{r_k}$ and in \textbf{P$_1$} we let $S_{k}:=S_{t_k}^{r_k,1}$.
In both cases, these satisfy the inequalities:
\begin{align*}
\|S_k(X)\|_{m,r_k}&\leq C_{m}r^{-c_m}{t}^{l+1}_k\|X\|_{m-l,r_k},\\
\|S_k(X)-X\|_{m-l,r_k}&\leq C_{m}r^{-c_m}t^{-l}_k\|X\|_{m+1,r_k}.
\end{align*}

For the procedures to be well-defined and to converge, we need that $t_0=t(\widetilde{\pi})$ is big enough, more
precisely it will have to satisfy a finite number of inequalities of the form
\begin{equation}\label{EQ_0}
t_0=t(\widetilde{\pi})>C(r(R-r))^{-c}.
\end{equation}
Taking $\widetilde{\pi}$ such that it satisfies (\ref{EQ_8}), it suffices to ask that $\delta$ is small enough and $d$ is
big enough, such that a finite number of inequalities of the form
\[\delta((R-r)r)^d< \frac{1}{C}(r(R-r))^c\]
hold, and then $t_0$ will satisfy (\ref{EQ_0}).

Also, since $t_0>4(R-r)^{-1}=\epsilon_0^{-1}$, it follows that
\[t_k>\epsilon_k^{-1},\ \ \forall\ k\geq 0.\]

We will prove inductively that the bivectors
\[Z_k:=\pi_k-\pi|_{{r_k}}\in \mathfrak{X}^2(E_{r_{k}})\]
satisfy the inequalities ($a_k$) and ($b_k$)
\begin{equation*}
(a_k)\ \ \ \ \|Z_k\|_{s,r_k}\leq t_k^{-\alpha}, \ \ \ \ \  \ \  (b_k) \ \ \ \ \|Z_k\|_{p,r_k}\leq t_k^{\alpha}.
\end{equation*}
Since $t_0^{-\alpha}=\|Z_0\|_{p,R}$, $(a_0)$ and $(b_0)$ hold. Assuming that $(a_k)$ and $(b_k)$ hold for some
$k\geq 0$, we will show that condition $(C_k)$ holds (i.e.\ the procedure is well-defined up to step $k$) and also that
$(a_{k+1})$ and $(b_{k+1})$ hold.

First we give a bound for the norms of $X_k$ in terms of the norms of $Z_k$
\begin{align}\label{EQ_X_n_l}
\|X_k\|_{m,r_k}&=\|S_{k}(h_1^{r_k}(Z_k))\|_{m,r_k}\leq C_mr^{-c_m}t_k^{1+l}\|h_1^{r_k}(Z_k)\|_{m-l,r_k}\leq\\
\nonumber &\leq C_mr^{-c_m}t_k^{1+l}\|Z_k\|_{m+s-l,r_k},\ \ \forall \ \ 0\leq l\leq m.
\end{align}
In particular, for $m=l$, we obtain
\begin{align}\label{EQ_X_n_n}
\|X_k\|_{m,r_k}&\leq C_mr^{-c_m}t_k^{1+m-\alpha}.
\end{align}
Since $\alpha>4$ and $t_k>\epsilon_k^{-1}$, this inequality implies that
\begin{equation}\label{EQ_X_12}
\|X_k\|_{1,r_k}\leq Cr^{-c}t_k^{2-\alpha}\leq Cr^{-c}t_0^{-1}t_k^{-1}<Cr^{-c}t_0^{-1}\epsilon_k.
\end{equation}
Since $t_0>Cr^{-c}/\theta$, we have that $\|X_k\|_{1,r_k}\leq \theta\epsilon_k$, and so by Lemma
\ref{Lemma_domain_of_flow} $(C_k)$ holds. Moreover, $X_k$ satisfies the inequalities from Lemma \ref{Lemma_size
of _the flow} and Lemma \ref{Lemma_tame_flow_pull_back}.

We deduce now an inequality for all norms $\|Z_{k+1}\|_{n,r_{k+1}}$, with $n\geq s$
\begin{align}\label{EQ_Z_n}
\|Z_{k+1}&\|_{n,r_{k+1}}=\|\varphi_{X_k}^{*}(Z_k)+\varphi_{X_k}^{*}(\pi)-\pi\|_{n,r_{k+1}}\leq\\
\nonumber &\leq C_nr^{-c_n}(\|Z_k\|_{n,r_{k}}+\|X_k\|_{n+1,r_k}\|Z_k\|_{0,r_k}+\|X_k\|_{n+1,r_k}\|\pi\|_{n+1,r_k})\leq\\
\nonumber&\leq C_nr^{-c_n}(\|Z_k\|_{n,r_{k}}+\|X_k\|_{n+1,r_k})\leq C_nr^{-c_n}t_k^{s+2}\|Z_k\|_{n,r_k},
\end{align}
where we used Lemma \ref{Lemma_tame_flow_pull_back}, the inductive hypothesis and inequality (\ref{EQ_X_n_l})
with $m=n+1$ and $l=s+1$. For $n=p$, using also that $s+2+\alpha\leq \frac{3}{2}\alpha-1$, this gives
$(b_{k+1})$:
\begin{align*}
\|Z_{k+1}\|_{p,r_{k+1}}&\leq Cr^{-c}t_k^{s+2+\alpha}\leq Cr^{-c}t_k^{\frac{3}{2}\alpha-1}\leq Cr^{-c}t_0^{-1}t_{k+1}^{\alpha}\leq t_{k+1}^{\alpha}.
\end{align*}

To prove $(a_{k+1})$, we write $Z_{k+1}=V_k+\varphi_{X_k}^{*}(U_{k})$, where
\[V_k:=\varphi_{X_k}^{*}(\pi)-\pi-\varphi_{X_k}^{*}([X_k,\pi]),\ \ U_{k}:=Z_{k}-[\pi,X_k].\]
Using Lemma \ref{Lemma_tame_flow_pull_back} and inequality (\ref{EQ_X_n_n}), we bound the two terms by
\begin{align}\label{EQ_V}
&\|V_k\|_{s,r_{k+1}}\leq Cr^{-c}\|\pi\|_{s+2,r_k}\|X_k\|_{0,r_k}\|X_k\|_{s+2,r_k}\leq  C r^{-c}t_k^{s+4-2\alpha},\\
\label{EQ_U}&\|\varphi_{X_k}^{*}(U_{k})\|_{s,r_{k+1}}\leq Cr^{-c}(\|U_k\|_{s,r_k}+\|U_k\|_{0,r_k}\|X_k\|_{s+1,r_{k}})\leq\\
\nonumber&\phantom{\|\varphi_{X_k}^{*}(U_{k})\|_{s,r_{k+1}}}\leq Cr^{-c}(\|U_k\|_{s,r_k}+t_k^{s+2-\alpha}\|U_k\|_{0,r_k}).
\end{align}
To compute the $C^s$-norm for $U_k$, we rewrite it as
\begin{align*}
U_k&=Z_k-[\pi,X_k]=[\pi,h_1^{r_k}(Z_k)]+h_2^{r_k}([\pi,Z_k])-[\pi,X_k]=\\
&=[\pi,(I-S_k)h_1^{r_k}(Z_k)]-\frac{1}{2}h_2^{r_k}([Z_k,Z_k]).
\end{align*}
By tameness of the Lie bracket, the first term can be bounded by
\begin{align*}
\|[\pi,(I-S_k)& h_1^{r_k}(Z_k)]\|_{s,r_k}\leq C r^{-c}\|(I-S_k)h_1^{r_k}(Z_k)\|_{s+1,r_k}\leq\\
&\leq Cr^{-c} t_k^{2-p+2s} \|h_1^{r_k}(Z_k)\|_{p-s,r_k}\leq C r^{-c}t_k^{2-p+2s} \|Z_k\|_{p,r_k}\leq\\
&\leq Cr^{-c} t_k^{2-p+2s+\alpha}=Cr^{-c} t_k^{-\frac{3}{2}\alpha-1},
\end{align*}
and using also the interpolation inequalities, for the second term we obtain
\begin{align*}
\|\frac{1}{2}h_2^{r_k}([Z_k,Z_k])&\|_{s,r_k}\leq C\|[Z_k,Z_k]\|_{2s,r_k}\leq Cr^{-c}\|Z_k\|_{0,r_k}\|Z_k\|_{2s+1,r_k}\leq\\
&\leq Cr^{-c}t_{k}^{-\alpha}\|Z_k\|_{s,r_k}^{\frac{p-(2s+1)}{p-s}}\|Z_k\|_{p,r_k}^{\frac{s+1}{p-s}}\leq Cr^{-c}t_k^{-\alpha(1+\frac{p-(3s+2)}{p-s})}.
\end{align*}
Since $-\alpha(1+\frac{p-(3s+2)}{p-s})\leq -\frac{3}{2}\alpha-1$, these two inequalities imply that
\begin{equation}\label{EQ_U_s}
\|U_k\|_{s,r_k}\leq Cr^{-c}t_k^{-\frac{3}{2}\alpha-1}.
\end{equation}
Using (\ref{EQ_X_n_n}), we bound the $C^0$-norm of $U_k$ by
\begin{equation}\label{EQ_U_0}
\|U_k\|_{0,r_{k}}\leq \|Z_k\|_{0,r_{k}}+\|[\pi,X_{k}]\|_{0,r_k}\leq t_k^{-\alpha}+Cr^{-c}\|X_{k}\|_{1,r_k}\leq Cr^{-c}t_{k}^{2-\alpha}.
\end{equation}
By (\ref{EQ_V}), (\ref{EQ_U}), (\ref{EQ_U_s}), (\ref{EQ_U_0}) and $s+4-2\alpha=-\frac{3}{2}\alpha-1$, $(a_{k+1})$
follows
\begin{align*}
\|Z_{k+1}\|_{s,r_{k+1}}&\leq Cr^{-c}(t_k^{s+4-2\alpha}+t_k^{-\frac{3}{2}\alpha-1})\leq\\
&\leq Cr^{-c}t_k^{-\frac{3}{2}\alpha-1}\leq(r^{-c}C/t_0)t_{k}^{-\frac{3}{2}\alpha}\leq t_{k+1}^{-\alpha}.
\end{align*}
This finishes the induction.

Using (\ref{EQ_Z_n}), for every $n\geq 1$, we find $k_n\geq 0$, such that
\[\|Z_{k+1}\|_{n,r_{k+1}}\leq t_k^{s+3}\|Z_{k}\|_{n,r_{k}}, \ \ \forall\ k\geq k_n.\]
Iterating this, we obtain
\[t_k^{s+3}\|Z_{k}\|_{n,r_{k}}\leq (t_kt_{k-1}\ldots t_{k_n})^{s+3}\|Z_{k_n}\|_{n,r_{k_n}}.\]
On the other hand we have that
\[t_kt_{k-1}\ldots t_{k_n}=t_{k_n}^{1+\frac{3}{2}+\ldots+(\frac{3}{2})^{k-k_n}}\leq t_{k_n}^{2(\frac{3}{2})^{k+1-k_n}}=t_k^3.\]
Therefore, we obtain a bound valid for all $k>k_n$
\[\|Z_{k}\|_{n,r_{k}}\leq t_{k}^{2(s+3)}\|Z_{k_n}\|_{n,r_{k_n}}.\]

Consider now $m>s$ and denote by $n:=4m-3s$. Applying the interpolation inequalities, for $k> k_{n}$, we obtain
\begin{align*}
\|Z_k\|_{m,r_k}\leq& C_m r^{-c_m} \|Z_k\|_{s,r_k}^{\frac{n-m}{n-s}}\|Z_k\|_{n,r_k}^{\frac{m-s}{n-s}}= C_mr^{-c_m}\|Z_k\|_{s,r_k}^{\frac{3}{4}}\|Z_k\|_{n,r_k}^{\frac{1}{4}}\leq \\
\leq& C_m r^{-c_m}t_k^{-\alpha\frac{3}{4}+2(s+3)\frac{1}{4}}\|Z_{k_{n}}\|_{n,r_{k_{n}}}^{\frac{1}{4}}=C_m r^{-c_m}t_k^{-(s+6)}\|Z_{k_{n}}\|_{n,r_{k_{n}}}^{\frac{1}{4}}.
\end{align*}
Using also inequality (\ref{EQ_X_n_l}), for $l=s$, we obtain
\begin{align*}
\|X_{k}\|_{m,r_k}\leq C_mr^{-c_m}t_k^{s+1}\|Z_k\|_{m,r_k}\leq  t_k^{-5}\left( C_mr^{-c_m}\|Z_{k_{n}}\|_{n,r_{k_{n}}}^{\frac{1}{4}}\right).
\end{align*}
This shows that the series $\sum_{k\geq 0}\|X_k\|_{m,r_k}$ converges for all $m$. By Lemma \ref{Lemma_size of _the
flow}, also $\sum_{k\geq 0}\mathrm{d}(\varphi_{X_k})_{m,r_{k+1}}$ converges for all $m$ and, moreover, by
(\ref{EQ_X_12}), we have that
\begin{equation*}
\sigma_1:=\sum_{k\geq 1}\mathrm{d}(\varphi_{X_k})_{1,r_{k+1}}\leq Cr^{-c}\sum_{k\geq 1}\|X_k\|_{1,r_k}\leq Cr^{-c}t_0^{-4}\sum_{k\geq 1}\epsilon_k\leq t_0^{-3}.
\end{equation*}
So, we may assume that $\sigma_1\leq \theta$ and $\mathrm{d}(\varphi_{X_k})_{1,r_{k+1}}<1$. Then, by applying
Lemma \ref{Lemma_convergent_embbedings} we conclude that the sequence $\psi_{k}|_{r}$ converges uniformly in
all $C^n$-norms to a map $\psi:E_r\to E_R$ in $\mathcal{U}_r$ which satisfies
\begin{equation*}
\mathrm{d}(\psi)_{1,r}\leq e^{Cr^{-c}\sigma_1}Cr^{-c}\sigma_1\leq e^{t_0^{-2}}t_0^{-2}\leq Ct_0^{-2}\leq t_0^{-1}.
\end{equation*}
So (\ref{EQ_continuity}) holds, and we can also assume that $\mathrm{d}(\psi)_{1,r}<\theta$, which, by Lemma
\ref{Lemma_embedding}, implies that $\psi$ is a local diffeomorphism. Since $\psi_{k}|_{r}$ converges in the
$C^1$-topology to $\psi$ and $\psi_k^*(\widetilde{\pi})=(d\psi_k)^{-1}(\widetilde{\pi}_{\psi_k})$, it follows that
$\psi_k^*(\widetilde{\pi})|_{r}$ converges in the $C^0$-topology to $\psi^*(\widetilde{\pi})$. On the other hand,
$Z_{k}|_{r}=\psi_k^*(\widetilde{\pi})|_{r}-\pi|_{r}$ converges to $0$ in the $C^0$-norm, hence
$\psi^*(\widetilde{\pi})=\pi|_{r}$. So $\psi$ is a Poisson map and a local diffeomorphism between
\[\psi:(E_r,\pi|_{r})\rmap (E_R,\widetilde{\pi}).\]

For \textbf{Procedure P$_1$}, as noted before the proposition, the first jet of $\varphi_{X_k}$ is that of the identity
along $N$. This clearly holds also for $\psi_k$, and since $\psi_{k}|_{r}$ converges to $\psi$ in the $C^1$-topology,
also $\psi$ is the identity along $N$ up to first order.
\end{proof}

We are ready now to finish the proof of Theorem \ref{Main_Theorem_intro}.

\subsection{Proof of part (a) of Theorem \ref{Main_Theorem_intro}}

We have to check the properties from the definition of $C^p$-$C^1$-rigidity. Consider $U:=\mathrm{int}(E_{\rho})$,
for some ${\rho}\in (0,1)$ and let $O\subset U$ be an open set such that $N\subset O\subset \overline{O}\subset U$. Let
$r<R$ be such that $O\subset E_r\subset E_R\subset U$. With $d$ and $\delta$ from Proposition
\ref{Proposition_technical}, we let
\[\mathcal{V}_O:=\{W\in\mathfrak{X}^2(U): \|W|_{R}-\pi|_{R}\|_{p,R}<\delta(r(R-r))^d\}.\]
For $\widetilde{\pi}\in \mathcal{V}_O$, define $\psi_{\widetilde{\pi}}$ to be the restriction to $\overline{O}$ of the
map $\psi$, obtained by applying \textbf{Procedure P$_0$} to $\widetilde{\pi}|_{R}$. Then $\psi$ is a Poisson
diffeomorphism $(O,\pi|_{O})\to (U,\widetilde{\pi})$, and by (\ref{EQ_continuity}), the assignment
$\widetilde{\pi}\mapsto \psi$ has the required continuity property.

\subsection{Proof of part (b) of Theorem \ref{Main_Theorem_intro}}

Consider $\widetilde{\pi}$ a Poisson structure on some neighborhood of $N$ with
$j^1\widetilde{\pi}|_N=j^1\pi|_N$. First we show that $\pi$ and $\widetilde{\pi}$ are formally Poisson
diffeomorphic around $N$. By \cite{MaFormal}, this property is controlled by the groups
$H^2(A_N,\mathcal{S}^k(\nu_N^*))$. The Lie groupoid $\mathcal{G}|_{N}\rightrightarrows N$ integrates $A_N$ and
is $s$-connected. Since $\nu_N^*\subset A_N$ is an ideal, by Lemma \ref{Lemma_integrating_ideals} from the
appendix, the action of $A_N$ on $\nu_N^*$ (hence also on $\mathcal{S}^k(\nu_N^*)$) also integrates to
$\mathcal{G}|_{N}$. Since $\mathcal{G}|_{N}$ has compact $s$-fibers with vanishing $H^2$, the Tame Vanishing
Lemma implies that $H^2(A_N, \mathcal{S}^k(\nu_N^*))=0$. So we can apply Theorem 1.1 \cite{MaFormal} to
conclude that there exists a diffeomorphism $\varphi$ between open neighborhoods of $N$, which is the identity on $N$
up to first order, and such that $j^{\infty}\varphi^*(\widetilde{\pi})|_N=j^{\infty}\pi|_N$.

Let $R\in(0,1)$ be such that $\varphi^*(\widetilde{\pi})$ is defined on $E_R$. Using the Taylor expansion up to order
$2d+1$ around $N$ for the bivector $\pi-\varphi^*(\widetilde{\pi})$ and its partial derivatives up to order $p$, we find
a constant $M>0$ such that
\[\|\varphi^*(\widetilde{\pi})|_{r}-\pi|_{r}\|_{p,r}\leq M r^{2d+1},\ \ \forall\ 0<r<R.\]
If we take $r<2^{-d}\delta/M $, we obtain that $\|\varphi^*(\widetilde{\pi})|_{r}-\pi|_{r}\|_{p,r}<\delta(r(r-r/2))^d$.
So, we can apply Proposition \ref{Proposition_technical}, and \textbf{Procedure P$_1$} produces a Poisson
diffeomorphism
\[\tau:(E_{r/2},\pi|_{r/2})\rmap (E_r,\varphi^*(\widetilde{\pi})|_{r}),\]
which is the identity up to first order along $N$. We obtain (b) with $\psi=\varphi\circ\tau$.

\begin{remark}\rm\label{Remark_SCI}
As mentioned already in the Introduction, Conn's proof has been formalized in \cite{Miranda,Monnier} into an abstract
Nash Moser normal form theorem, and it is likely that one could use Theorem 6.8 \cite{Miranda} to prove partially our
rigidity result. Nevertheless, the continuity assertion, which is important in applications (see \cite{MarDef}), is not a
consequence of this result. There are also several technical reasons why we cannot apply \cite{Miranda}: we need the
size of the $C^p$-open set to depend polynomially on $r^{-1}$ and $(R-r)^{-1}$, because we use a formal linearization
argument (this dependence is not given in \emph{loc.cit.}); to obtain diffeomorphisms which fix $N$, we work with
vector fields which vanish along $N$ up to first order, and it is unlikely that this Fr\'echet space admits smoothing
operators of degree $0$ (in \emph{loc.cit.} this is the overall assumption); for the inequalities in Lemma
\ref{Lemma_composition} we need special norms for the embeddings (indexed also by ``$\delta$''), which are not
considered in \emph{loc.cit.}
\end{remark}

%%%%%%%%%%%%%%%%%%%%%%%%%%%%%%%%%%%%%%%%%%%%%%%%%%%%
%%%%%%%%%%%%%%%%%%%%%%%%%%%%%%%%%%%%%%%%%%%%%%%%%%%%
%%%%%%%%%%%%%%%%%%%%%%%%%%%%%%%%%%%%%%%%%%%%%%%%%%%%
%%%%%%%%%%%%%%%%%%%%%%%%%%%%%%%%%%%%%%%%%%%%%%%%%%%%
%%%%%%%%%%%%%%%%%%%%%%%%%%%%%%%%%%%%%%%%%%%%%%%%%%%%
\begin{appendix}
\section{Invariant tubular neighborhoods}

In the proof of Theorem \ref{Main_Theorem_intro}, we have used the following result:

\begin{lemma}\label{Lemma_tubular_neighborhood}
Let $\mathcal{G}\rightrightarrows M$ be a proper Lie groupoid with connected $s$-fibers and let $N\subset M$ be a
compact invariant submanifold. There exists a tubular neighborhood $E\subset M$ (where $E\cong T_NM/TN$) and a
metric on $E$, such that for all $r>0$ the closed tube $E_r:=\{v\in E: |v|\leq r\}$ is $\mathcal{G}$-invariant.
\end{lemma}

This lemma follows from results in \cite{Posthuma}; in particular we will use:

\begin{lemma}[Proposition 3.14, Proposition 6.4 \cite{Posthuma}]
On the base of a proper Lie groupoid there exist Riemannian metrics such that every geodesic which emanates
orthogonally from an orbit stays orthogonal to any orbit it passes through. Such metrics are called \textbf{adapted}.
\end{lemma}

\begin{proof}[Proof of Lemma \ref{Lemma_tubular_neighborhood}]
Let $g$ be an adapted metric on $M$ and let $E:=TN^{\perp}\subset T_NM$ be the normal bundle. By rescaling $g$,
we may assume that
\begin{enumerate}[(1)]
\item the exponential is defined on $E_2$ and on $\textrm{int}(E_2)$ it is an open embedding;
\item for all $r\in (0,1]$ we have that
\[\exp(E_r)=\{p\in M: d(p,N)\leq r\},\]
where $d$ denotes the distance induced by the Riemannian structure.
\end{enumerate}
%For these assertions, see the proof of Theorem 20, Chapter 9 in \cite{Spivak}.

Let $v\in E_1$ with base point $x$, and denote by $r:=|v|$. We claim that the geodesic $\gamma(t):=\exp(tv)$ at $t=1$ is normal
to $\exp (\partial E_r)$ at $\gamma(1)$:
\[T_{\gamma(1)}\exp (\partial E_r)=\dot{\gamma}(1)^{\perp}.\]
Let $S_r(x)$ be the sphere of radius $r$ centered at $x$. By the Gauss Lemma
\[\dot{\gamma}(1)^{\perp}=T_{\gamma(1)}S_r(x),\]
and by (2), $\overline{B}_r(x)\subset \exp(E_r)$, where $\overline{B}_r(x)$ is the closed ball of radius $r$ around $x$.
Since $\overline{B}_r(x)$ and $\exp(E_r)$ intersect at $\gamma(1)$, their boundaries must be tangent at this point,
and this proves the claim.

By assumption, $N$ is a union of orbits, therefore the geodesics $\gamma(t):=\exp(tv)$, for $v\in E$, start normal to the
orbits of $\mathcal{G}$, thus, by the property of the metric, they continue to be orthogonal to the orbits. Hence, by
our claim, the orbits which intersect $\exp(\partial E_r)$ are tangent to $\exp(\partial E_r)$. By connectivity of the
orbits, $\exp(\partial E_r)$ is invariant, for all $r\in(0,1)$. Define the embedding $E\hookrightarrow M$ by
\[v\mapsto \exp \left(\frac{\lambda(|v|)}{|v|}v\right),\] where $\lambda: [0,\infty)\to [0,1)$ is a diffeomorphism which is the identity on $[0,1/2)$.
\end{proof}

\section{Integrating ideals}

Representations of a Lie groupoid $\mathcal{G}$ can be differentiated to representations of its Lie algebroid $A$, but
in general, a representation of $A$ does only integrate to a representation of the $s$-fiber 1-connected groupoid of
$A$, and not necessarily to one of $\mathcal{G}$. In this subsection, we prove that representations of $A$ on ideals
can be integrated to representations of any $s$-connected integration. This result was used in the proof of part (b) of
Theorem \ref{Main_Theorem_intro}.

Let $(A,[\cdot,\cdot],\rho)$ be a Lie algebroid. We call a subbundle $I\subset A$ an \textbf{ideal} of $A$, if $\rho(I)=0$
and $\Gamma(I)$ is an ideal of the Lie algebra $\Gamma(A)$. Using the Leibniz rule, one easily sees that, if $I\neq A$,
then the second condition implies the first. An ideal $I$ is naturally a representation of $A$, with $A$-connection given
by the Lie bracket
\[\nabla_X(Y):=[X,Y], \ \ X\in\Gamma(A), \ Y\in\Gamma(I).\]

\begin{lemma}\label{Lemma_integrating_ideals}
Let $\mathcal{G}\rightrightarrows M$ be a Hausdorff Lie groupoid with Lie algebroid $A$ and let $I\subset A$ be an
ideal. If the $s$-fibers of $\mathcal{G}$ are connected, then the representation of $A$ on $I$ given by the Lie bracket
integrates to $\mathcal{G}$.
\end{lemma}

\begin{proof}

First observe that $\mathcal{G}$ acts on the possibly singular bundle of isotropy Lie algebras $\ker(\rho)\to M$ via the
formula:
\begin{align}\label{EQ_Action}
g\cdot Y=\frac{d}{d\epsilon}\left(g\phantom{\cdot} exp(\epsilon Y) g^{-1} \right)|_{\epsilon=0},  \forall\  Y\in \ker(\rho)_{s(g)}.
\end{align}
Let $N(I)\subset \mathcal{G}$ be the subgroupoid consisting of elements $g$ which satisfy $g\cdot I_{s(g)}\subset
I_{t(g)}$. We will prove that $N(I)=\mathcal{G}$ and that the induced action of $\mathcal{G}$ on $I$ differentiates
to the Lie bracket.

Recall that a derivation on a vector bundle $E\to M$ (see section 3.4 in \cite{MackenzieGT}) is a pair $(D,V)$, with
$D$ a linear operator on $\Gamma(E)$ and $V$ a vector field on $M$, satisfying
\[D(f\alpha)=fD(\alpha)+V(f)\alpha,  \ \forall\ \alpha\in\Gamma(E), f\in C^{\infty}(M) .\]
The flow of a derivation $(D,V)$, denoted by $\varphi_D^t$, is a vector bundle map covering the flow $\varphi_V^t$ of
$V$, $\varphi_D^t:E_x\to E_{\varphi_V^t(x)}$, (whenever $\varphi_V^t(x)$ is defined), which is the solution to the
following differential equation
\[\varphi_{D}^0=\textrm{Id}_E,\ \ \frac{d}{dt}(\varphi_D^t)^*(\alpha)=(\varphi_D^t)^*(D \alpha),\]
where $(\varphi_D^t)^*(\alpha)_x=\varphi_D^{-t}(\alpha_{\varphi_V^{t}(x)})$.

For $X\in\Gamma(A)$, denote by $\Psi^t(X,g)$ the flow of the corresponding right invariant vector field on
$\mathcal{G}$, and by $\varphi^t(X,x)$ the flow of $\rho(X)$ on $M$. Conjugation by $\Psi^t(X)$ is an automorphism
of $\mathcal{G}$ covering $\varphi^t(X)$, which we denote by
\[C(\Psi^t(X)):\mathcal{G}\rmap \mathcal{G}, \ g\mapsto \Psi^t(X,t(g))g\Psi^t(X, s(g))^{-1}.\]
Since $C(\Psi^t(X))$ sends the $s$-fiber over $x$ to the $s$-fiber over $\varphi^t(X,x)$, its differential at the identity
$1_x$ gives an isomorphism
\[Ad(\Psi^t(X)):A_{x}\rmap A_{\varphi^t(X,x)}, \ Ad(\Psi^t(X))_x:=dC(\Psi^t(X))|_{A_x}.\]
On $\ker(\rho)_x$, we recover the action (\ref{EQ_Action}) of $g=\Psi^t(X,x)$. We have that:
\begin{align}\label{EQ_adjoint_diff}
\frac{d}{dt} &(Ad(\Psi^t(X))^*Y)_x=\frac{d}{dt}Ad(\Psi^{-t}(X,\varphi^t(X,x)))Y_{\varphi^t(X,x)}=\\
\nonumber &=-\frac{d}{ds}\left(Ad(\Psi^{-t}(X,\varphi^t(X,x)))Ad(\Psi^s(X,\varphi^{t-s}(X,x)))Y_{\varphi^{t-s}(X,x)}\right)|_{s=0}=\\
\nonumber &=Ad(\Psi^{-t}(X,\varphi^t(X,x)))[X,Y]_{\varphi^t(X,x)}=Ad(\Psi^t(X))^*([X,Y])_x,
\end{align}
for $Y\in\Gamma(A)$, where we have used the adjoint formulas from Proposition 3.7.1 in \cite{MackenzieGT}. This
shows that $Ad(\Psi^{t}(X))$ is the flow of the derivation $([X,\cdot],\rho(X))$ on $A$. Since $I$ is an ideal, the
derivation $[X,\cdot]$ restricts to a derivation on $I$, and therefore $I$ is invariant under $Ad(\Psi^{t}(X))$. This
proves that for all $Y\in I_x$,
\[Ad(\Psi^t(X,x))Y=\Psi^t(X,x)\cdot Y\in I.\]
So $N(I)$ contains all the elements in $\mathcal{G}$ of the form $\Psi^t(X,x)$. The set of such elements contains an
open neighborhood $O$ of the unit section in $\mathcal{G}$. Since the $s$-fibers of $\mathcal{G}$ are connected,
$O$ generates $\mathcal{G}$ (see Proposition 1.5.8 in \cite{MackenzieGT}), therefore $N(I)=\mathcal{G}$ and so
(\ref{EQ_Action}) defines an action of $\mathcal{G}$ on $I$.

Using that $\Psi^{-t}(X,\varphi^t(X,x))=\Psi^{t}(X,x)^{-1}$, equation (\ref{EQ_adjoint_diff}) gives
\begin{align*}
\frac{d}{dt}\left(\Psi^{t}(X,x)^{-1}\cdot Y_{\varphi^t(X,x)}\right)|_{t=0}=[X,Y]_x,\  \forall  \ X\in\Gamma(A), Y\in\Gamma(I).
\end{align*}
Thus, the action differentiates to the Lie bracket (see Definition 3.6.8 \cite{MackenzieGT}).
\end{proof}

\section{The Tame Vanishing Lemma}\label{section_homotopy_operators}

In this subsection we prove the Tame Vanishing Lemma, an existence result for tame homotopy operators on the
complex computing Lie algebroid cohomology with coefficients. In the proof of Theorem \ref{Main_Theorem_intro},
this lemma was applied to the Poisson complex. In combination with the Nash-Moser techniques, the Tame Vanishing
Lemma is very useful when applied to various geometric problems (see the appendix in \cite{teza}).

\subsection{The weak $C^{\infty}$-topology}\label{subsection_weak_topology}

The compact-open $C^k$-topology on the space of sections of a vector bundle can be generated by a family of
semi-norms, and we recall here a construction of such semi-norms, generalizing the construction from section
\ref{section_technical_rigidity}. These semi-norms will be used to express the tameness property of the homotopy
operators.

Let $W\to M$ be a vector bundle. Consider $\mathcal{U}:=\{U_i\}_{i\in I}$ a locally finite open cover of $M$ by
relatively compact domains of coordinate charts $\{\chi_i:U_i\diffto \mathbb{R}^m\}_{i\in I}$ and choose
trivializations for $W|_{U_i}$. Let $\mathcal{O}:=\{O_i\}_{i\in I}$ be a second open cover, with $\overline{O}_i$
compact and $\overline{O}_i\subset U_i$. A section $\sigma\in\Gamma(W)$ can be represented in these charts by a
family of smooth functions $\{\sigma_i:\mathbb{R}^m\to \mathbb{R}^k\}_{i\in I}$, where $k$ is the rank of $W$. For
$U\subset M$, an open set with compact closure, we have that $\overline{U}$ intersects only a finite number of the
coordinate charts $U_i$. Denote the set of such indexes by $I_{U}\subset I$. Define the $n$-th norm of $\sigma$ on
$U$ by
\[\|\sigma\|_{n,\overline{U}}:=\sup \left\{\left|\frac{\partial^{|\alpha|}\sigma_{i}}{\partial x^{\alpha}}(x)\right|: |\alpha|\leq n,\   x\in \chi_i(U\cap O_i),\  i\in I_U\right\}.\]

For a fixed $n$, the family of semi-norms $\|\cdot\|_{n,\overline{U}}$, with $U$ being a relatively compact open set in $M$,
generate the \textbf{compact-open} $C^{n}$-\textbf{topology} on $\Gamma(W)$. The union of all these topologies,
for $n\geq 0$, is called the \textbf{weak} $C^{\infty}$-\textbf{topology} on $\Gamma(W)$. Observe that the
semi-norms $\{\|\cdot\|_{n,\overline{U}}\}_{n\geq 0}$ induce norms on $\Gamma(W|_{\overline{U}})$.

\subsection{The statement of the Tame Vanishing Lemma}

\begin{lemma}[The Tame Vanishing Lemma]\label{Tame_Vanishing_Lemma}
Let ${\mathcal{G}}\rightrightarrows M$ be a Hausdorff Lie groupoid with Lie algebroid $A$ and let $V$ be a
representation of $\mathcal{G}$. If the $s$-fibers of ${\mathcal{G}}$ are compact and their de Rham cohomology
vanishes in degree $p$, then \[H^p(A,V)=0.\] Moreover, there exist linear homotopy operators
\[\Omega^{p-1}(A,V) \stackrel{h_1}{\longleftarrow}\Omega^{p}(A,V)\stackrel{h_2}{\longleftarrow}\Omega^{p+1}(A,V), \]
\[d_{\nabla}h_1+h_2d_{\nabla}=\mathrm{Id},\]
which satisfy
\begin{enumerate}[(1)]
\item invariant locality: for every orbit $O$ of $A$, they induce linear maps
\[\Omega^{p-1}(A|_{O},V|_{O}) \stackrel{h_{1,O}}{\longleftarrow}\Omega^{p}(A|_{O},V|_{O})\stackrel{h_{2,O}}{\longleftarrow}\Omega^{p+1}(A|_{O},V|_{O}),\]
such that for all $\omega\in \Omega^{p}(A,V)$, $\eta\in \Omega^{p+1}(A,V)$, we have that
\[h_{1,O}(\omega|_{O})=(h_1\omega)|_{O},\  \ h_{2,O}(\eta|_{O})=(h_2 \eta)|_{O},\]
\item tameness: for every invariant open $U\subset M$, with $\overline{U}$ compact, there are constants
$C_{n,U}>0$, such that
\[\|h_1(\omega)\|_{n,\overline{U}}\leq C_{n,U} \|\omega \|_{n+s,\overline{U}},\ \ \|h_2(\eta)\|_{n,\overline{U}}\leq C_{n,U}  \|\eta \|_{n+s,\overline{U}},\]
for all $\omega\in\Omega^{p}(A|_{\overline{U}},V|_{\overline{U}})$ and
$\eta\in\Omega^{p+1}(A|_{\overline{U}},V|_{\overline{U}})$, where \[s=\lfloor\frac{1}{2}rank(A)\rfloor+1.\]
\end{enumerate}
\end{lemma}
We also note the following consequences of the proof:
\begin{corollary}\label{corollary_unu}
The constants $C_{n,U}$ can be chosen such that they are uniform over all invariant open subsets of $U$. More
precisely: if $V\subset U$ is a second invariant open set, then one can choose $C_{n,V}:=C_{n,U}$, assuming that the
norms on $\overline{U}$ and $\overline{V}$ are computed using the same charts and trivializations. % $\|\cdot\|_{n,\overline{U}}$ and $\|\cdot\|_{n,\overline{V}}$.
\end{corollary}
\begin{corollary}\label{corollary_unu_prim}
The homotopy operators preserve the order of vanishing around orbits. More precisely: if $O$ is an orbit of $A$, and
$\omega\in \Omega^{p}(A,V)$ is a form such that $j^k\omega|_{O}=0$, then $j^kh_1(\omega)|_{O}=0$; and
similarly for $h_2$.
\end{corollary}

\subsection{The de Rham complex of a fiber bundle}\label{subsection_de Rham_complex}

To prove the Tame Vanishing Lemma, we first construct tame homotopy operators for the foliated de Rham complex of
a fiber bundle. For this, we use a result on the family of inverses of elliptic operators (Proposition
\ref{Theorem_family_of_operators}), which we prove at the end of the section.

Let $\pi:\mathcal{B}\to M$ be a locally trivial fiber bundle whose fibers $\mathcal{B}_x:=\pi^{-1}(x)$ are
diffeomorphic to a compact, connected manifold $F$ and let $V\to M$ be a vector bundle. The space of vertical vectors
on $\mathcal{B}$ will be denoted by $T^{\pi}\mathcal{B}$ and the space of foliated forms with values in $\pi^*(V)$
by $\Omega^{\bullet}(T^{\pi}\mathcal{B},\pi^*(V))$. An element
$\omega\in\Omega^{\bullet}(T^{\pi}\mathcal{B},\pi^*(V))$ is a smooth family of forms on the fibers of $\pi$ with
values in $V$
\[\omega=\{\omega_{x}\}_{x\in M}, \ \ \ \ \ \omega_x\in \Omega^{\bullet}(\mathcal{B}_x,V_x).\]
The fiberwise exterior derivative induces the differential
\[d\otimes I_V:\Omega^{\bullet}(T^{\pi}\mathcal{B},\pi^*(V))\rmap \Omega^{\bullet+1}(T^{\pi}\mathcal{B},\pi^*(V)), \]
\[d\otimes I_V(\omega)_x:=(d\otimes I_{V_x})(\omega_x),\  x\in M.\]

We construct the homotopy operators using Hodge theory. Let $m$ be a metric on $T^{\pi}{\mathcal{B}}$, or
equivalently a smooth family of Riemannian metrics $\{m_x\}_{x\in M}$ on the fibers of $\pi$. Integration against the
volume density gives an inner product on $\Omega^{\bullet}(\mathcal{B}_x)$
\[(\eta,\theta):=\int_{\mathcal{B}_x}m_x(\eta,\theta)|dVol(m_x)|,\ \eta,\theta\in\Omega^q(\mathcal{B}_x).\]
Let $\delta_x$ denote the formal adjoint of $d$ with respect to this inner product
\[\delta_x:\Omega^{\bullet+1}({\mathcal{B}}_x)\rmap \Omega^{\bullet}({\mathcal{B}}_x),\]
i.e.\ $\delta_x$ is the unique linear first order differential operator satisfying
\[(d\eta,\theta)=(\eta,\delta_x\theta),\ \ \forall\ \eta\in\Omega^{\bullet}(\mathcal{B}_x), \theta\in\Omega^{{\bullet}+1}(\mathcal{B}_x).\]
%Since $d^2=0$, it is easy to see that $\delta_x^2=0$.
The Laplace-Beltrami operator associated to $m_x$ will be denoted by
\[\Delta_x:\Omega^{\bullet}({\mathcal{B}}_x)\rmap \Omega^{\bullet}({\mathcal{B}}_x),\ \ \Delta_x:=d\delta_x+\delta_x d.\]
Both these operators induce linear differential operators on $\Omega^{\bullet}(T^{\pi}\mathcal{B},\pi^{*}(V))$
\[\delta\otimes I_V:\Omega^{\bullet+1}(T^{\pi}\mathcal{B},\pi^{*}(V))\to \Omega^{\bullet}(T^{\pi}\mathcal{B},\pi^{*}(V)),\ \ \ \delta\otimes I_V(\omega)_x:=(\delta_x\otimes I_{V_x})(\omega_x),\]
\[\Delta\otimes I_V:\Omega^{\bullet}(T^{\pi}\mathcal{B},\pi^{*}(V))\to \Omega^{\bullet}(T^{\pi}\mathcal{B},\pi^{*}(V)),\ \ \ \Delta\otimes I_V(\omega)_x:=(\Delta_x\otimes I_{V_x})(\omega_x).\]

By the Hodge theorem, if the fiber $F$ of $\mathcal{B}$ has vanishing de Rham cohomology in degree $p$, then
$\Delta_x$ is invertible in degree $p$.
\begin{lemma}\label{Lemma_fiberwise_Green}
If $H^p(F)=0$ then the following hold:
\begin{enumerate}[(a)]
\item $\Delta\otimes I_V$ is invertible in degree $p$ and its inverse is given by
\begin{align*}
G\otimes I_V :\Omega^{p}(T^{\pi}{\mathcal{B}},\pi^{*}(V))\rmap \Omega^{p}(T^{\pi}{\mathcal{B}},\pi^{*}(V)),\\
(G\otimes I_V)(\omega)_x:=(\Delta_x^{-1}\otimes I_{V_x})(\omega_x), \  x\in M;
\end{align*}
%\[G\otimes I_V :\Omega^{p}(T^{\pi}{\mathcal{B}},\pi^{*}(V))\to \Omega^{p}(T^{\pi}{\mathcal{B}},\pi^{*}(V)),\ \ \ (G\otimes I_V)(\omega)_x:=(\Delta_x^{-1}\otimes
%I_{V_x})(\omega_x),\]
\item the maps $H_1:=(\delta\otimes I_V)\circ (G\otimes I_V)$ and $H_2:=(G\otimes I_V)\circ (\delta\otimes I_V)$
\[\Omega^{p-1}(T^{\pi}\mathcal{B},\pi^{*}(V))\stackrel{H_1}{\longleftarrow}\Omega^{p}(T^{\pi}\mathcal{B},\pi^{*}(V))\stackrel{H_2}{\longleftarrow}\Omega^{p+1}(T^{\pi}\mathcal{B},\pi^{*}(V))\]
are linear homotopy operators in degree $p$;
\item $H_1$  and $H_2$ satisfy the following local-tameness property: for every relatively compact open $U\subset M$, there are
constants $C_{n,U}>0$ such that
\[\|H_1(\eta) \|_{n,\mathcal{B}|_{\overline{U}}}\leq C_{n,U} \|\eta \|_{n+s,\mathcal{B}|_{\overline{U}}},\ \forall\ \eta\in\Omega^{p}(T^{\pi}{\mathcal{B}|_{\overline{U}}},\pi^{*}(V|_{\overline{U}})),\]
\[\|H_2(\omega) \|_{n,\mathcal{B}|_{\overline{U}}}\leq C_{n,U} \|\omega \|_{n+s,\mathcal{B}|_{\overline{U}}},\ \forall\ \omega\in\Omega^{p+1}(T^{\pi}{\mathcal{B}|_{\overline{U}}},\pi^{*}(V|_{\overline{U}})).\]
where $s=\lfloor\frac{1}{2}\mathrm{dim}(F)\rfloor+1$. \\
\noindent Moreover, if $U'\subset U$, then one can take $C_{n,U'} := C_{n,U}$.
\end{enumerate}
\end{lemma}
\begin{proof}
In a trivialization chart the operator $\Delta\otimes I_V$ is given by a smooth family of Laplace-Beltrami operators:
\[\Delta_x:\Omega^p(F)^k\rmap \Omega^p(F)^k,\]
where $k$ is the rank of $V$. These operators are elliptic and invertible, therefore, by Proposition
\ref{Theorem_family_of_operators}, $\Delta_x^{-1}(\omega_x)$ is smooth in $x$, for every smooth family $\omega_x\in
\Omega^p(F)^k$. This shows that $G\otimes I_V$ maps smooth forms to smooth forms. Clearly $G\otimes I_V$ is the
inverse of $\Delta\otimes I_V$, so we have proven (a).

For part (c), let $U\subset M$ be a relatively compact open set. Applying part (2) of Proposition
\ref{Theorem_family_of_operators} to a family of coordinate charts which cover $\overline{U}$, we find constants
$D_{n,U}$ such that
\[\|G\otimes I_V(\eta) \|_{n,\mathcal{B}|_{\overline{U}}}\leq D_{n,U} \|\eta \|_{n+s-1,\mathcal{B}|_{\overline{U}}},\ \forall\ \eta\in\Omega^{p}(T^{\pi}{\mathcal{B}|_{\overline{U}}},\pi^{*}(V|_{\overline{U}})).\]
Moreover, the constants can be chosen such that they are decreasing in $U$. Since $H_1$ and $H_2$ are defined as
the composition of $G\otimes I_V$ with a linear differential operator of degree one, it follows that we can also find
constants $C_{n,U}$ such that the inequalities form (c) are satisfied, and which are also decreasing in $U$.

For part (b), using that $\delta_x^2=0$, we obtain that $\Delta_x$ commutes with $d \delta_x$
\[\Delta_x d\delta_x=(d\delta_x+\delta_x d)d\delta_x=d\delta_x d\delta_x+ \delta_x d^2\delta_x=d\delta_x d\delta_x,\]
\[d\delta_x\Delta_x=d\delta_x(d\delta_x+\delta_x d)=d\delta_x d\delta_x+d\delta_x^2 d=d\delta_x d\delta_x.\]
This implies that $\Delta\otimes I_V$ commutes with $(d\otimes I_V)(\delta\otimes I_V)$, and thus $G\otimes I_V$
commutes with $(d\otimes I_V)( \delta\otimes I_V)$. Using this, we obtain that $H_1$ and $H_2$ are homotopy
operators:
\begin{align*}
I=&(G\otimes I_V)(\Delta\otimes I_V)=(G\otimes I_V)((d\otimes I_V)(\delta\otimes I_V)+(\delta\otimes I_V)( d\otimes I_V))=\\
&=(d\otimes I_V)(\delta\otimes I_V)(G\otimes I_V)+(G\otimes I_V)(\delta\otimes I_V)( d\otimes I_V)=\\
&=(d\otimes I_V) H_1+H_2(d\otimes I_V).
\end{align*}
\end{proof}

\subsection{Proof of the Tame Vanishing Lemma}

Let $\mathcal{G}\rightrightarrows M$ be as in the statement. By passing to the connected components of the
identities in the $s$-fibers \cite{MM}, we may assume that $\mathcal{G}$ is $s$-connected. Then $s:\mathcal{G}\to
M$ is a locally trivial fiber bundle with compact fibers whose cohomology vanishes in degree $p$. We will apply
Lemma \ref{Lemma_fiberwise_Green} to the complex of $s$-foliated forms with coefficients in $s^*(V)$
\[(\Omega^{{\bullet}}(T^s{\mathcal{G}}, s^*(V)),d\otimes I_V).\]

Recall that the \textbf{right translation} by an arrow $g\in {\mathcal{G}}$ is the diffeomorphism between the
$s$-fibers above $y=t(g)$ and above $x=s(g)$, given by:
\[r_g:{\mathcal{G}}_y\diffto {\mathcal{G}}_x, \ \  r_g(h):=hg.\]
A form $\omega\in\Omega^{\bullet}(T^s{\mathcal{G}}, s^*(V))$ is called \textbf{invariant}, if it satisfies
\[(r_g^*\otimes g)(\omega_{hg})=\omega_h,\  \forall \ h,g\in {\mathcal{G}}, \textrm{ with }s(h)=t(g),\]
where $r_g^*\otimes g$ is the linear isomorphism $\eta\mapsto g\cdot\eta\circ dr_g$. Denote the space of invariant
$V$-valued forms on ${\mathcal{G}}$ by $\Omega^{\bullet}(T^s{\mathcal{G}},s^*(V))^{\mathcal{G}}$.

It is well-known that forms on $A$ with values in $V$ are in one to one correspondence with invariant $V$-valued
forms on $\mathcal{G}$; this correspondence is given by
\[J:\Omega^{\bullet}(A,V)\rmap\Omega^{\bullet}(T^s{\mathcal{G}},s^*(V)), \ \ J(\eta)_g:=(r_{g^{-1}}^*\otimes g^{-1})(\eta_{t(g)}).\]
The map $J$ is also a chain map, thus it induces an isomorphism of complexes (see Theorem 1.2 \cite{WeinXu} and
also subsection 2.3.2 \cite{teza} for coefficients)
\begin{equation}\label{EQ_isomorphism}
J:(\Omega^{\bullet}(A,V),d_{\nabla})\diffto (\Omega^{\bullet}(T^s{\mathcal{G}},s^*(V))^{\mathcal{G}},d\otimes I_V).
\end{equation}
A left inverse for $J$ (i.e.\ a map $P$ such that $P\circ J=\textrm{Id}$) is given by
\[P:\Omega^{\bullet}(T^s{\mathcal{G}},s^*(V))\rmap \Omega^{\bullet}(A,V), \ \ P(\omega)_x:=\omega_{u(x)}.\]

Let $\langle\cdot,\cdot\rangle$ be an inner product on $A$. Using right translations, we extend
$\langle\cdot,\cdot\rangle$ to an invariant metric $m$ on $T^s{\mathcal{G}}$:
\[m(X,Y)_g:=\langle dr_{g^{-1}}X,dr_{g^{-1}}Y\rangle_{t(g)},\  \forall \ X,Y\in T_g^s{\mathcal{G}}.\]
Invariance of $m$ implies that the right translation by an arrow $g:x\to y$ is an isometry between the $s$-fibers
\[r_g:({\mathcal{G}}_y,m_y)\diffto ({\mathcal{G}}_x,m_x).\]
The corresponding operators from subsection \ref{subsection_de Rham_complex} are also invariant.
\begin{lemma}\label{Lemma_delta_inv}
The operators $\delta\otimes I_V$, $\Delta\otimes I_V$, $H_1$ and $H_2$, corresponding to $m$, send invariant forms
to invariant forms.
\end{lemma}
\begin{proof}
Since right translations are isometries and the operators $\delta_z$ are invariant under isometries we have that
$r_g^*\circ \delta_x=\delta_y\circ r_g^*$, for all arrows $g:x\to y$.

For $\eta\in\Omega^{\bullet}(T^s{\mathcal{G}},s^*(V))^{\mathcal{G}}$ we have that
\begin{align*}
(r_g^*\otimes g)&(\delta\otimes I_V(\eta))|_{{\mathcal{G}}_x}=(r_g^*\circ \delta_x\otimes g)(\eta|_{{\mathcal{G}}_x})=(\delta_y\circ r_g^*\otimes g)(\eta|_{{\mathcal{G}}_x})=\\
&=(\delta_y\otimes I_{V_y})(r_g^*\otimes g)(\eta|_{{\mathcal{G}}_x})=(\delta_y\otimes I_{V_y})(\eta|_{{\mathcal{G}}_y})=(\delta\otimes I_V)(\eta)|_{{\mathcal{G}}_y}.
\end{align*}
This shows that $\delta\otimes I_V(\eta)\in\Omega^{\bullet}(T^s{\mathcal{G}},s^*(V))^{\mathcal{G}}$. The other
operators are constructed in
terms of $\delta\otimes I_V$ and $d\otimes I_V$, thus they also preserve $\Omega^{\bullet}(T^s{\mathcal{G}},s^*(V))^{\mathcal{G}}$.%the space of invariant forms.
\end{proof}

This lemma and the isomorphism (\ref{EQ_isomorphism}) imply that the maps
\[\Omega^{p-1}(A,V) \stackrel{h_1}{\longleftarrow}\Omega^{p}(A,V)\stackrel{h_2}{\longleftarrow}\Omega^{p+1}(A,V),\]
\[h_1:=P\circ H_{1}\circ J, \ \  h_2:=P\circ H_{2}\circ J,\]
are linear homotopy operators for the Lie algebroid complex in degree $p$.

For part (1) of the Tame Vanishing Lemma, let $\omega\in\Omega^p(A,V)$ and $O\subset M$ an orbit of $A$. Since
$\mathcal{G}$ is $s$-connected we have that $s^{-1}(O)=t^{-1}(O)=\mathcal{G}|_{O}$. Clearly
$J(\omega)|_{s^{-1}(O)}$ depends only on $\omega|_{O}$. By the construction of $H_1$, for all $x\in O$, we have that
\[h_1(\omega)_x=H_1(J(\omega))_{1_x}=(\delta_x\circ\Delta_x^{-1}\otimes I_{V_x})(J(\omega)|_{s^{-1}(x)})_{1_x}.\]
Thus $h_1(\omega)|_{O}$ depends only on $\omega|_{O}$. The same argument applies also to $h_2$.

Before checking part (2), we give a simple lemma:
\begin{lemma} \label{Lemma_pullback_tame}
Consider a vector bundle map $A:F_1\to F_2$ between vector bundles $F_1\to M_1$ and $F_2\to M_2$, covering a map
$f:M_1\to M_2$. If $A$ is fiberwise invertible and $f$ is proper, then the pullback map
\[A^*:\Gamma(F_2)\rmap \Gamma(F_1), \ A(\sigma)_x:=A_{x}^{-1}(\sigma_{f(x)})\]
satisfies the following tameness inequalities: for every open $U\subset M_2$, with $\overline{U}$ compact, there are
constants $C_{n,U}>0$ such that
\[\|A^*(\sigma)\|_{n, \overline{f^{-1}(U)}}\leq C_{n,U}\|\sigma\|_{n, U},\ \ \forall\  \sigma \in \Gamma(F_2|_{\overline{U}}).\]
Moreover:
\begin{enumerate}[(a)]
\item if $U'\subset U$ is open, and one uses the same charts when computing the norms, then one can choose $C_{n,U'}:=C_{n,U}$;
\item if $N\subset M_2$ is a submanifold and $\sigma\in\Gamma(F_2)$ satisfies $j^k(\sigma)|_{N}=0$, then its pullback satisfies
$j^k(A^*(\sigma))|_{f^{-1}(N)}=0$.
\end{enumerate}
\end{lemma}
\begin{proof}
Since $A$ is fiberwise invertible, we can assume that $F_1=f^*(F_2)$ and $A^*=f^*$. By choosing a vector bundle
$F'$ such that $F_2\oplus F'$ is trivial, we reduce the problem to the case when $F_2$ is the trivial line bundle. So, we
have to check that $f^*:C^{\infty}(M_2)\to C^{\infty}(M_1)$ has the desired properties. But this is straightforward:
we just cover both $\overline{f^{-1}(U)}$ and $\overline{U}$ by charts, and apply the chain rule. The constants
$C_{n,U}$ are the $C^n$-norm of $f$ over $\overline{f^{-1}(U)}$, and therefore are getting smaller if $U$ gets
smaller. This implies (a). For part (b), just observe that $j^k_{f(x)}(\sigma)=0$ implies $j^k_x(\sigma\circ f)=0$.
\end{proof}

Part (2) of the Tame Vanishing Lemma follows by Lemma \ref{Lemma_fiberwise_Green} (c) and by applying Lemma
\ref{Lemma_pullback_tame} to $J$ and $P$. Corollary \ref{corollary_unu} follows from Lemma
\ref{Lemma_pullback_tame} (a) and Lemma \ref{Lemma_fiberwise_Green} (c). To prove Corollary
\ref{corollary_unu_prim}, consider $\omega$ a form with $j^k\omega|_{O}=0$, for $O$ an orbit. Then, by Lemma
\ref{Lemma_pullback_tame} (b), it follows that $J(\omega)$ vanishes up to order $k$ along
$t^{-1}(O)=\mathcal{G}|_{O}$. By construction, we have that $H_1$ is $C^{\infty}(M)$ linear, therefore also
$H_1(J(\omega))$ vanishes up to order $k$ along $\mathcal{G}|_{O}$; and again by Lemma
\ref{Lemma_pullback_tame} (b) $h_1(\omega)=u^*(H_1(J(\omega)))$ vanishes along $O=u^{-1}(\mathcal{G}|_{O})$
up to order $k$.

\subsection{The inverse of a family of elliptic operators}\label{section_inverse_family}

This subsection is devoted to proving the following result:

\begin{proposition}\label{Theorem_family_of_operators}
Consider a smooth family of linear differential operators
\[P_{x}:\Gamma(V)\rmap \Gamma(W),\ \ x\in \mathbb{R}^m,\]
between sections of vector bundles $V$ and $W$ over a compact base $F$. If $P_x$ is elliptic of degree $d\geq 1$ and
invertible for all $x\in \mathbb{R}^m$, then
\begin{enumerate}[(1)]
\item the family of inverses $\{Q_x:=P_{x}^{-1}\}_{x\in\mathbb{R}^m}$ induces a linear operator
\[Q:\Gamma(p^*(W))\rmap \Gamma(p^*(V)),\ \  \{\omega_{x}\}_{x\in \mathbb{R}^m}\mapsto \{Q_x\omega_x\}_{x\in \mathbb{R}^m},\]
where $p^*(V):=V\times \mathbb{R}^m\to F\times\mathbb{R}^m$ and $p^*(W):=W\times \mathbb{R}^m\to
F\times\mathbb{R}^m$;
\item $Q$ is locally tame, in the sense that for all bounded open sets $U\subset\mathbb{R}^m$, there exist constants $C_{n,U}>0$, such that
the following inequalities hold
\[\|Q(\omega)\|_{n,F\times\overline{U}}\leq C_{n,U}\|\omega\|_{n+s-1,F\times \overline{U}}, \ \forall \omega \in \Gamma(p^*(W)|_{ F\times\overline{U}}),\]
with $s=\lfloor\frac{1}{2}\mathrm{dim}(F)\rfloor+1$. If $U'\subset U$, then one can take $C_{n,U'}:= C_{n,U}$.
\end{enumerate}
\end{proposition}

Fixing $C^n$-norms $\|\cdot\|_{n}$ on $\Gamma(V)$, we induce semi-norms on $\Gamma(p^*(V))$:
\[\|\omega\|_{n,F\times\overline{U}}:=\sup_{0\leq k+|\alpha|\leq n}\sup_{x\in U}\|\frac{\partial^{|\alpha|}\omega_{x}}{\partial x^{\alpha}}\|_{k},\]
where $\omega\in\Gamma(p^*(V))$ is regarded as a smooth family $\omega=\{\omega_x\in
\Gamma(V)\}_{x\in\mathbb{R}^m}$. Similarly, fixing norms on $\Gamma(W)$, we define also norms on
$\Gamma(p^*(W))$.

Endow $\Gamma(V)$ and $\Gamma(W)$ also with Sobolev norms, denoted by $\{|\cdot|_n\}_{n\geq 0}$. Loosely
speaking, $|\omega|_n$, measures the $L^2$-norm of $\omega$ and its partial derivatives up to order $n$ (for a
precise definition see e.g. \cite{Gilkey}). Denote by $H_n(\Gamma(V))$ and by $H_n(\Gamma(W))$ the completion of
$\Gamma(V)$, respectively of $\Gamma(W)$, with respect to the Sobolev norm $|\cdot|_n$.

We will use the standard inequalities between the Sobolev and the $C^n$-norms, which follow from the Sobolev
embedding theorem
\begin{equation}\label{EQ_Sobolev}
\|\omega\|_n\leq C_n |\omega|_{n+s},  \ \ |\omega|_n\leq C_n \|\omega\|_{n},
\end{equation}
for all $\omega\in\Gamma(V)$ (resp. $\Gamma(W)$), where $s=\lfloor\frac{1}{2}\mathrm{dim}(F)\rfloor+1$ and
$C_n>0$ are constants.

Since $P_x$ is of order $d$, it induces continuous linear maps between the Sobolev spaces, denoted by
\[[P_x]_n: H_{n+d}(\Gamma(V))\rmap  H_{n}(\Gamma(W)).\]

These maps are invertible.

\begin{lemma}\label{L_invertible_on_Sobolev}
If an elliptic differential operator of degree $d$ \[P:\Gamma(V)\rmap\Gamma(W)\] is invertible, then for every $n\geq 0$
the induced map
\[[P]_n:H_{n+d}(\Gamma(V))\rmap H_{n}(\Gamma(W))\] is also invertible and its inverse is induced by the inverse of $P$.
\end{lemma}

\begin{proof}
Since $P$ is elliptic, it is invertible modulo smoothing operators (see Lemma 1.3.5 in \cite{Gilkey}), i.e.\ there exists a
pseudo-differential operator
\[\Psi:\Gamma(W)\rmap \Gamma(V),\]
of degree $-d$ such that $\Psi P-\textrm{Id}=K_1$ and $P \Psi-\textrm{Id}=K_2$, where $K_1$ and $K_2$ are
smoothing operators. Since $\Psi$ is of degree $-d$, it induces continuous maps
\[[\Psi]_n:H_{n}(\Gamma(W))\rmap H_{n+d}(\Gamma(V)),\]
and since $K_1$ and $K_2$ are smoothing operators, they induce continuous maps
\[[K_1]_n:H_{n}(\Gamma(V))\rmap \Gamma(V), \ \ \ [K_2]_n:H_{n}(\Gamma(W))\rmap \Gamma(W).\]

We show now that $[P]_n$ is a bijection:

\underline{injective}: For $\eta\in H_{n+d}(\Gamma(V))$, with $[P]_n\eta=0$, we have that
\[\eta=(\textrm{Id}-[\Psi]_n [P]_n)\eta=-[K_1]_n\eta\in \Gamma(V),\]
hence $[P]_n\eta=P\eta$. By injectivity of $P$, we have that $\eta=0$.

\underline{surjective}: For $\theta\in H_{n}(\Gamma(W))$, we have that
\[([P]_n[\Psi]_n-\textrm{Id})\theta=[K_2]_n\theta\in \Gamma(W),\]
and, since $P$ is onto, $[K_2]_n\theta=P\eta$ for some $\eta\in \Gamma(V)$. So $\theta$ is in the range of $[P]_n$:
\[\theta=[P]_n([\Psi]_n\theta-\eta).\]

The inverse of a bounded operator between Banach spaces is bounded, therefore $[P]_n^{-1}$ is continuous. Since on
smooth sections $[P]_n^{-1}$ coincides with $P^{-1}$, and since the space of smooth sections is dense in all Sobolev
spaces, it follows that $P^{-1}$ induces a continuous map $H_{n}(\Gamma(W))\to H_{n+d}(\Gamma(V))$, and that
this map is $[P]_n^{-1}$.
\end{proof}

For two Banach spaces $B_1$ and $B_2$ denote by $Lin(B_1,B_2)$ the Banach space of bounded linear maps between
them and by $Iso(B_1,B_2)$ the open subset consisting of invertible maps. The following proves that the family
$[P_x]_n$ is smooth.

\begin{lemma}\label{L_smooth_family_operators}
Let $\{P_x\}_{x\in \mathbb{R}^m}$ be a smooth family of linear differential operators of order $d$ between the
sections of vector bundles $V$ and $W$, both over a compact manifold $F$. Then the map induced by $P$ from
$\mathbb{R}^m$ to the space of bounded linear operators between the Sobolev spaces
\[\mathbb{R}^m\ni x\mapsto [P_x]_n\in Lin(H_{n+d}(\Gamma(V)),H_{n}(\Gamma(W)))\]
is smooth and its derivatives are induced by the derivatives of $P_x$.
\end{lemma}
\begin{proof}
Linear differential operators of degree $d$ form $V$ to $W$ are sections of the vector bundle
$Hom(J^d(V);W)=J^d(V)^*\otimes W$, where $J^d(V)\to F$ is the $d$-th jet bundle of $V$. Therefore, $P$ can be
viewed as a smooth section of the pullback bundle $p^*(Hom(J^d(V);W)):=Hom(J^d(V);W)\times \mathbb{R}^m \to
F\times \mathbb{R}^m$. Since $F$ is compact, by choosing a partition of unity on $F$ with supports inside some
opens on which $V$ and $W$ trivialize, one can write any section of $p^*(Hom(J^d(V);W))$ as a linear combination of
sections of $Hom(J^d(V);W)$ with coefficients in $C^{\infty}(\mathbb{R}^m\times F)$. Hence, there are constant
differential operators $P_i$ and functions $f_i\in C^{\infty}(\mathbb{R}^m\times F)$, for $i=1,2\ldots, N$, such that
\[P_x=\sum_{i=1}^N f_i(x)P_i.\]
So it suffices to prove that for $f\in C^{\infty}(\mathbb{R}^m\times F)$, multiplication with $f(x)$ induces a smooth
map
\[\mathbb{R}^m\ni x\mapsto [f(x)\textrm{Id}]_n\in Lin(H_{n}(\Gamma(W)),H_{n}(\Gamma(W))).\]
First, it is easy to see that for any smooth function $g\in C^{\infty}(\mathbb{R}^m\times F)$ and every compact
$K\subset \mathbb{R}^m$, there are constants $C_n(g,K)$ such that $|g(x)\sigma|_n\leq C_n(g,K)|\sigma|_n$ for all
$x\in K$ and $\sigma\in H_n(\Gamma(W))$; or equivalently that the operator norm satisfies
$|[g(x)\mathrm{Id}]_n|_{op}\leq C_n(g,K)$, for $x\in K$.

Consider $f\in C^{\infty}(\mathbb{R}^m\times F)$, and $\overline{x}\in\mathbb{R}^m$ and $K$ is a closed ball
centered at $\overline{x}$. Using the Taylor expansion of $f$ at $\overline{x}$, write
\[f(x)-f(\overline{x})=\sum_{i=1}^m(x_i-\overline{x}_i)T^i_{\overline{x}}(x),\]
\[f(x)-f(\overline{x})-\sum_{i=1}^m(x_i-\overline{x}_i)\frac{\partial f}{\partial x_i}(\overline{x})=\sum_{1\leq i\leq j\leq m}(x_i-\overline{x}_i)(x_j-\overline{x}_j)T^{i,j}_{\overline{x}}(x),\]
where $T^{i}_{\overline{x}},T^{i,j}_{\overline{x}}\in C^{\infty}(\mathbb{R}^m\times F)$. Thus, for all $x\in K$,
we have that
\begin{align*}
|[f(x)\mathrm{Id}]_n-[f(\overline{x})\mathrm{Id}]_n&|_{op}\leq |x-\overline{x}| \sum_{1\leq i\leq m}C_n(T^{i}_{\overline{x}},K),\\
|[f(x)\mathrm{Id}]_n-[f(\overline{x})\mathrm{Id}]_n-&\sum_{i=1}^m(x_i-\overline{x}_i)[\frac{\partial f}{\partial x_i}(\overline{x})\mathrm{Id}]_n|_{op}\leq\\
&\leq |x-\overline{x}|^2\sum_{1\leq i\leq j\leq m}C_n(T^{i,j}_{\overline{x}},K).
\end{align*}
The first inequality implies that the map $x\mapsto [f(x)\mathrm{Id}]_n$ is $C^0$ and the second that it is $C^1$,
with partial derivatives given by
\[\frac{\partial }{\partial x_i}[f\mathrm{Id}]_n=[\frac{\partial f}{\partial x_i}\mathrm{Id}]_n.\]
The statement follows now inductively.
\end{proof}

\subsubsection*{Proof of Proposition \ref{Theorem_family_of_operators}}

By Lemma \ref{L_invertible_on_Sobolev}, $Q_x=P_x^{-1}$ induces continuous operators
\[[Q_x]_n:H_{n}(\Gamma(W))\rmap H_{n+d}(\Gamma(V)).\]
We claim that the following map is smooth
\[\mathbb{R}^m\ni x\mapsto [Q_x]_n\in Lin(H_{n}(\Gamma(W)),H_{n+d}(\Gamma(V))).\]
This follows by Lemma \ref{L_invertible_on_Sobolev} and Lemma \ref{L_smooth_family_operators}, since we can write
\[[Q_x]_n=[P_x^{-1}]_n=[P_x]_n^{-1}=\iota([P_x]_n),\]
where $\iota$ is the (smooth) inversion map
\[\iota:Iso(H_{n+d}(\Gamma(V)),H_{n}(\Gamma(W)))\rmap Iso(H_{n}(\Gamma(W)),H_{n+d}(\Gamma(V))).\]

Let $\omega=\{\omega_x\}_{x\in \mathbb{R}^m}\in\Gamma(p^*(W))$. By our claim and Lemma
\ref{L_smooth_family_operators}, it follows that
\[x\mapsto [Q_x]_n[\omega_x]_n=[Q_x\omega_x]_{n+d}\in H_{n+d}(\Gamma(V))\] is a smooth map. On the other hand, the Sobolev inequalities
(\ref{EQ_Sobolev}) show that the inclusion $\Gamma(V)\to  \Gamma^n(V)$, where $\Gamma^n(V)$ is the space of
sections of $V$ of class $C^n$ (endowed with the norm $\|\cdot\|_{n}$), extends to a continuous map
\[H_{n+s}(\Gamma(V))\rmap \Gamma^n(V).\]
Since also evaluation $ev_p:\Gamma^n(V)\to V_p$ at $p\in F$ is continuous, it follows that the map $x\mapsto
Q_x\omega_x(p)\in V_p$ is smooth. This is enough to conclude smoothness of the family $\{Q_x\omega_x\}_{x\in
\mathbb{R}^m}$, so $Q(\omega)\in\Gamma(p^*(V))$. This finishes the proof of the first part.

For the second part, let $U\subset \mathbb{R}^m$ be an open set with $\overline{U}$ compact. Since the map $x\mapsto
[Q_x]_n$ is smooth, it follows that
\begin{equation}\label{EQ_finite_numbers}
D_{n,m,U}:=\sup_{x\in U}\sup_{|\alpha|\leq m}|\frac{\partial^{|\alpha|}}{\partial x^\alpha}[Q_x]_n|_{op} < \infty,
\end{equation}
where $|\cdot|_{op}$ denotes the operator norm. Let $\omega=\{\omega_x\}_{x\in \overline{U}}$ be an element of
$\Gamma(p^*(W)|_{F\times \overline{U}})$. By Lemma \ref{L_smooth_family_operators}, also the map $x\mapsto
[\omega_x]_{n}\in H_{n}(\Gamma(W))$ is smooth and that for all multi-indices $\gamma$
\[\frac{\partial^{|\gamma|}}{\partial x^\gamma}[\omega_x]_{n}=[\frac{\partial^{|\gamma|}}{\partial x^\gamma}\omega_x]_{n}.\]
Let $k$ and $\alpha$ be such that $|\alpha|+k\leq n$. Using (\ref{EQ_Sobolev}), (\ref{EQ_finite_numbers}) we obtain
\begin{align*}
\|\frac{\partial^{|\alpha|}}{\partial x^\alpha}& (Q_x\omega_x)\|_{k}\leq \|\frac{\partial^{|\alpha|}}{\partial x^\alpha} (Q_x\omega_x)\|_{k+d-1}\leq
 C_{k+d-1} |\frac{\partial^{|\alpha|}}{\partial x^\alpha} (Q_x \omega_x)|_{k+s+d-1}\leq\\
&\leq C_{k+d-1} \sum_{\beta+\gamma=\alpha}\binom{\alpha}{\beta\  \gamma} |\frac{\partial^{|\beta|}}{\partial x^\beta}Q_x\frac{\partial^{|\gamma|}}{\partial x^\gamma}\omega_x|_{k+s+d-1}\leq\\
&\leq C_{k+d-1} \sum_{\beta+\gamma=\alpha}\binom{\alpha}{\beta\  \gamma} D_{k+s-1,|\beta|,U}|\frac{\partial^{|\gamma|}}{\partial x^\gamma}\omega_x|_{k+s-1}\leq\\
&\leq C_{k+d-1}C_{k+s-1} \sum_{\beta+\gamma=\alpha}\binom{\alpha}{\beta\  \gamma} D_{k+s-1,|\beta|,U}\|\frac{\partial^{|\gamma|}}{\partial x^\gamma}\omega_x\|_{k+s-1}\leq\\
&\leq C_{n,U}\|\omega\|_{n+s-1,F\times\overline{U}}.
\end{align*}
This proves the second part:
\[\|Q(\omega)\|_{n,F\times\overline{U}}\leq C_{n,U}\|\omega\|_{n+s-1,F\times\overline{U}}.\]
The constants $D_{n,m,U}$ are clearly decreasing in $U$, hence for $U'\subset U$ we also have that $C_{n,U'}\leq
C_{n,U}$. This finishes the proof of Proposition \ref{Theorem_family_of_operators}.

\end{appendix}

\bibliographystyle{amsplain}
\def\lllll{}

\end{document}